\documentclass[10pt]{amsart}

\usepackage{amsmath, amsthm, graphics,setspace,xypic}
\usepackage{latexsym}
\usepackage{mathtools}
\usepackage{eucal}
\usepackage{caption}

\usepackage{amssymb}
\usepackage{color}
\usepackage{amscd}
\usepackage{palatino}
\usepackage{latexsym}
\usepackage{epsfig}
\usepackage{graphicx}
\usepackage{amsfonts}
\usepackage{psfrag}
\usepackage[all]{xy}
\usepackage{youngtab, ytableau}

\usepackage{url}
\usepackage{cite}

\usepackage[margin=1in]{geometry}

\definecolor{ltgrey}{RGB}{180, 187, 198}

\input xy
\xyoption{all}
\UseComputerModernTips

\numberwithin{equation}{section}
\numberwithin{figure}{section}

\newtheorem{theorem}{Theorem}[section]
\newtheorem{lemma}[theorem]{Lemma}
\newtheorem{proposition}[theorem]{Proposition}

\newtheorem{conjecture}[theorem]{Conjecture}
\newtheorem{problem}[theorem]{Problem}
\newtheorem{remark}[theorem]{Remark}
\newtheorem{example}[theorem]{Example}

\newtheorem{THM}{Theorem}
\newtheorem{PROB}{Problem}

\theoremstyle{definition}

\newcommand{\C}{{\mathbb{C}}}

\newcommand{\fh}{\mathfrak{h}}

\newcommand{\into}{\hookrightarrow}

\DeclareMathOperator{\Stab}{Stab}

\DeclareMathOperator{\pt}{pt}

\newcommand{\hsm}{{\hspace{1mm}}}

\usepackage{multicol}
\usepackage{color}

\definecolor{gold}{rgb}{0.85,.66,0}
\definecolor{cherry}{rgb}{0.9,.1,.2}
\definecolor{burgundy}{rgb}{0.8,.2,.2}
\definecolor{orangered}{rgb}{0.85,.3,0}
\definecolor{orange}{rgb}{0.85,.4,0}
\definecolor{olive}{rgb}{.45,.4,0}
\definecolor{lime}{rgb}{.6,.9,0}
\definecolor{green}{rgb}{.2,.7,0}
\definecolor{grey}{rgb}{.4,.4,.2}
\definecolor{brown}{rgb}{.4,.3,.1}

\renewcommand{\S}{{\mathcal{S}}}

\renewcommand{\Stab}{{\mathrm{Stab}}}
\newcommand{\Flags}{{\mathcal{F}\ell ags}}

\newcommand{\Hess}{{\mathcal{H}ess}}

\def\Tn{T}
\newcommand{\Sn}{S_n}

\DeclareMathOperator{\Par}{P}

\begin{document}

\title{Toward permutation bases in the equivariant cohomology rings of regular semisimple Hessenberg 
varieties}

\author{Megumi Harada}
\address{Department of Mathematics and
Statistics\\ McMaster University\\ 1280 Main Street West\\ Hamilton, Ontario L8S4K1\\ Canada}
\email{Megumi.Harada@math.mcmaster.ca}
\urladdr{\url{http://www.math.mcmaster.ca/Megumi.Harada/}}
\thanks{The first author is partially supported by an NSERC Discovery Grant and a Canada Research Chair (Tier 2) award.}

\author{Martha Precup}
\address{Department of Mathematics and Statistics\\ Washington University in St. Louis \\ One Brookings Drive \\ St. Louis, Missouri  63130 \\ U.S.A. }
\email{mprecup@math.northwestern.edu}
\urladdr{\url{https://www.math.wustl.edu/~precup/}}
\thanks{The second author is partially supported by NSF Grant DMS 1954001}

\author [J. Tymoczko]{Julianna Tymoczko}
\address{Department of Mathematics \& Statistics \\ 
Clark Science Center \\ 
Smith College \\ 
Burton Hall 115 \\ 
Northampton, MA 01063 \\ 
U.S.A. }
\email{tymoczko@smith.edu}


\date{\today}


\begin{abstract} 
Recent work of Shareshian and Wachs, Brosnan and Chow, and Guay-Paquet connects the well-known Stanley--Stembridge conjecture in combinatorics to the dot action of the symmetric group $\Sn$ on the cohomology rings $H^*(\Hess(\mathsf{S},h))$ of regular semisimple Hessenberg varieties. In particular, in order to prove the Stanley--Stembridge conjecture, it suffices to construct (for any Hessenberg function $h$) a permutation basis of $H^*(\Hess(\mathsf{S},h))$ whose elements have stabilizers isomorphic to Young subgroups. 
In this manuscript we give several results which contribute toward this goal. Specifically, in some special cases, we give a new, purely combinatorial construction of classes in the $T$-equivariant cohomology ring $H^*_T(\Hess(\mathsf{S},h))$ which form permutation bases for subrepresentations in $H^*_T(\Hess(\mathsf{S},h))$. Moreover, from the definition of our classes it follows that the stabilizers are isomorphic to Young subgroups. Our constructions use a presentation of the $T$-equivariant cohomology rings $H^*_T(\Hess(\mathsf{S},h))$ due to Goresky, Kottwitz, and MacPherson. The constructions presented in this manuscript generalize past work of Abe--Horiguchi--Masuda, Chow, and Cho--Hong--Lee. 
\end{abstract}

\maketitle

\setcounter{tocdepth}{1}
\tableofcontents

\section{Introduction} 

Hessenberg varieties (in Lie type A) are subvarieties of the full flag
variety $\Flags(\C^n)$ of nested sequences of linear subspaces in
$\C^n$. Research concerning Hessenberg varieties lies in a fruitful intersection of algebraic geometry, combinatorics, and representation theory, and they have been studied
extensively since the late 1980s. These varieties are parameterized by a choice of linear operator $\mathsf{S}\in \mathfrak{gl}(n,\C)$ and nondecreasing function $h: [n] \to[n]$, where $[n]:= \{1,2,\ldots, n\}$, called a Hessenberg function. When $\mathsf{S}$ is a regular semisimple element and $h(i)\geq i$ for all $i$, $\Hess(\mathsf{S},h)$ is called a regular semisimple Hessenberg variety.

The dot action of the symmetric group $\Sn$ on the cohomology\footnote{In this paper, we focus exclusively on cohomology rings with coefficients in $\C$. We will omit the notation of coefficients in our cohomology rings for this reason.}  rings $H^*(\Hess(\mathsf{S},h))$ of regular semisimple Hessenberg varieties, defined by the third author in \cite{Tymoczko2008}, has received considerable recent attention due to its connection to the well-known Stanley--Stembridge conjecture in combinatorics. This conjecture states that the chromatic symmetric function of the incomparability graph of a (3+1)-free poset is $e$-positive, i.e., it is a non-negative linear combination of elementary
symmetric functions. The Stanley--Stembridge conjecture is a well-known conjecture in the field of algebraic combinatorics and is related, for example, to various other deep conjectures about immanants \cite{Stembridge1992}. 
The relationship between the Stanley--Stembridge conjecture and Hessenberg varieties was made apparent some years ago by work of Shareshian and Wachs \cite{Shareshian-Wachs2016}, Brosnan and Chow \cite{Brosnan-Chow2018}, and Guay-Paquet \cite{Guay-Paquet2016}. 
We refer the reader to \cite{Harada-Precup2019} for a leisurely exposition of the history; for the purposes of this manuscript we restrict ourselves to recalling that, in order to prove the Stanley--Stembridge conjecture from the point of view of Hessenberg varieties, it suffices to construct a basis of $H^*(\Hess(\mathsf{S},h))$
that is permuted by the dot action (i.e., a permutation basis) and such that the stabilizer of each element is a subgroup of $\Sn$ generated by reflections.  This problem has motivated much research in the field of Hessenberg varieties in the last few years.

In this manuscript, we tackle this problem by using techniques that are available in 
$T$-equivariant cohomology and not ordinary cohomology. We exploit general properties of equivariant cohomology and of the $T$-action on $\Hess(\mathsf{S},h)$, which in particular imply any free $H^*_T(\pt)$-module basis of $H^*_T(\Hess(\mathsf{S},h))$ projects to a $\C$-basis of $H^*(\Hess(\mathsf{S},h))$ under the natural projection. The definition of the dot action in \cite{Tymoczko2008} used this same philosophy, defining an action on $H^*_T(\Hess(\mathsf{S},h))$ and then inducing an action on ordinary cohomology by this same projection. Similarly, our strategy is to first construct a $H^*_T(\pt)$-module basis for $H^*_T(\Hess(\mathsf{S},h))$ which is permuted by the dot action, and then project it to ordinary cohomology. Since this construction of the basis is consistent with the construction of the dot action on $H^*_T(\Hess(\mathsf{S},h))$ and $H^*(\Hess(\mathsf{S},h))$, a set that is permuted by the dot action in equivariant cohomology projects to a set that is permuted also in ordinary cohomology. Section~\ref{subsec: perm basis program} contains a more leisurely account of this approach toward the Stanley--Stembridge conjecture, including an explicit formulation of what we call the ``permutation basis program.''

Our goal in this manuscript is to take preliminary steps toward the construction of a permutation basis of $H^*(\Hess(\mathsf{S},h))$ in the following sense. We explicitly construct 
collections of cohomology classes in $H^*_T(\Hess(\mathsf{S},h))$ which are permuted by the dot action, are $H^*_T(\pt)$-linearly independent, and whose stabilizer groups are  reflection subgroups. From this it follows that these classes form a permutation basis of the subrepresentation in $H^*_T(\Hess(\mathsf{S},h))$ which they span. Moreover, we can identify explicitly this subrepresentation in terms of permutation representations $M^\lambda:=\mathrm{ind}^{\Sn}_{S_\lambda}(\mathbf{1})$ for appropriate partitions $\lambda$ and Young subgroups $S_\lambda$ of $\Sn$. Thus, our results can be viewed as 
achieving some progress toward the larger goal of building a full $H^*_T(\pt)$-module permutation basis of $H^*_T(\Hess(\mathsf{S},h))$, with 
point stabilizers isomorphic to Young subgroups -- which would in turn resolve the Stanley--Stembridge conjecture.

One important subtlety is that we consider equivariant cohomology as a module over a polynomial ring and not as a complex vector space.  This means that, when we equip equivariant cohomology with the structure of a (twisted) representation of the finite group $S_n$, a submodule which is stable under the representation may not be a direct summand.  For instance, with the standard action of the permutation group $S_2$ on $\mathbb{C}[t_1,t_2]$, the symmetric polynomial $t_1+t_2$ generates a $\mathbb{C}[t_1,t_2]$-subrepresentation that cannot be written as a direct summand of $\mathbb{C}[t_1,t_2].$ Thus although this manuscript constructs a linearly independent set of vectors in $H^*_T(\Hess(\mathsf{S},h))$ which are permuted by the dot action and have stabilizer equal to a Young subgroup, it is not a priori guaranteed that our set can be extended to a full permutation basis. Another subtlety is that the $H^*_T(\pt)$-linear independence of our sets of permuted vectors does not necessarily imply that their projections to $H^*(\Hess(\mathsf{S},h))$ are still linearly independent. Together, these subtleties mean that the open question remains, whether we can indeed extend our linearly independent sets in this manuscript to a full permuted basis. (See Section \ref{subsec: perm basis program} for more.) This is a question for future work.

We now summarize the results within this manuscript in a rough form. Our main technical tool is the Goresky--Kottwitz--MacPherson (GKM) theory of 
$T$-equivariant cohomology. Here we consider the maximal torus $T$ of diagonal matrices in $GL(n,\C)$ and the natural $T$-action on $\Hess(\mathsf{S},h) \subseteq \Flags(\C^n)$ induced from the action of $GL(n,\C)$ on 
$\Flags(\C^n) \cong GL(n,\C)/B$ by left multiplication. GKM theory describes explicitly and combinatorially the $T$-equivariant cohomology $H^*_T(\Hess(\mathsf{S},h))$ as a collection of lists of polynomials--- one polynomial for each permutation $w \in \Sn$--- which satisfy compatibility conditions (see~\eqref{GKM for X(h)}); see \cite{Tymoczko2008} for details. 
While the explicit combinatorial nature of the GKM description of $H^*_T(\Hess(\mathsf{S},h))$ is convenient for many purposes, it is worth pointing out that the question of building permutation bases in the language of GKM theory poses its own computational challenges. This is because the dot action exchanges polynomials associated to different permutations $w \in \Sn$, and this complicates the analysis of the linear independence of orbits under the dot action. Nevertheless, in some special cases we are able to overcome these obstacles, 
as we now explain.

Our first results give purely combinatorial constructions of well-defined GKM classes in $H^*_T(\Hess(\mathsf{S},h))$. 
We begin by formalizing a statement which is well-known to experts but which (to our knowledge) has not been recorded in the literature in this generality. 
Let $\lambda = (\lambda_1, \lambda_2, \cdots, \lambda_\ell)$ be a composition of $n$ and $S_\lambda$ denote the associated Young subgroup, generated by the set of simple reflections $\{s_i  \mid i \not \in \{\lambda_1, \lambda_1+\lambda_2, \cdots, \lambda_1 + \lambda_2 + \cdots + \lambda_{\ell-1}\} \}$. Let 
$v_\lambda$ be the permutation obtained by taking the longest permutation $w_0 = [ n, n-1, \ldots, 2, 1 ]$ in $\Sn$ and re-ordering the values $\{1,2,\cdots,\lambda_1\}, \{\lambda_1+1,\ldots, \lambda_1+\lambda_2\}, \cdots, \{\lambda_1 + \lambda_2 + \cdots + \lambda_{\ell-1}+1, \cdots, n\}$ to be increasing.  Recall that $v_\lambda$ is the unique maximal element with respect to Bruhat order in the set of shortest coset representatives for $S_\lambda \backslash \Sn$.  Now define a function $f_\lambda^{(k)}: S_n \to \C[t_1, \ldots, t_n]$ by
$$
f_\lambda (w) := \left\{ \begin{array}{cl} \prod_{t_i-t_j\in N_h^-(v_\lambda)} (t_{w(i)} - t_{w(j)}) & \textup{ if } w=yv_\lambda, \textup{ for some } y\in S_\lambda\\
0 & \textup{ otherwise }   \end{array}\right.
$$
where $N_h^-(v_\lambda) := \{t_i - t_j  \mid i>j \, \textup{ and } \, v_\lambda(i) < v_\lambda(j) \, \textup{ and } \, i \leq h(j) \}$. 
Then it is well-known among experts that $f_\lambda \in H_T^{2 |N_h^-(v_\lambda)|}(\Hess(\mathsf{S},h))$ is a well-defined equivariant cohomology class. 

The above construction yields $T$-equivariant cohomology classes which have the special property that their support set (i.e., the permutations $w \in \Sn$ on which $f_\lambda(w) \neq 0$) is the single coset of the Young subgroup $S_\lambda$ containing the maximal coset representative $v_\lambda$. Thus we call these ``top-coset classes.'' Moreover, it is not difficult to see that the orbit under the dot action of these ``top-coset classes'' is $H^*_T(\pt)$-linearly independent. Specifically, for $f_\lambda$ the top-coset GKM class defined above, 
the $\Sn$-orbit of $f_\lambda$ under the dot action 
\[
\{ w \cdot f_\lambda \mid w\in \Sn \}
\] 
is $H_T^*(\pt)$-linearly independent.  Furthermore, the $H_T^*(\pt)$-subrepresentation of $H_T^*(\Hess(\mathsf{S},h))$ spanned by this set in $H^*_T(\Hess(\mathsf{S},h))$ is an $\Sn$-subrepresentation with the same character as the $\Sn$-representation $\mathrm{ind}_{S_{\lambda}}^{\Sn}(\mathbf{1})\simeq M^{\Par(\lambda)}$, where $\Par(\lambda)$ is the partition of $n$ obtained from $\lambda$ by rearranging the parts in decreasing order. 
We explain these facts in some detail in Section~\ref{sec.top.coset}.

As mentioned above, even in cases for which the Stanley--Stembridge conjecture is known to hold, constructing an explicit basis for the free $H_T^*(\pt)$-module $H^*_T(\Hess(\mathsf{S},h))$ which is permuted by the dot action remains difficult. Progress has been made in two special cases.  The first is $h=(h(1), n, \ldots, n)$, studied by Abe, Horiguchi, and Masuda in~\cite{Abe-Horiguchi-Masuda2019} and the second is $h=(2,3,\ldots, n,n)$ where Cho, Hong, and Lee \cite{Cho-Hong-Lee2020} recently proved a conjecture of Chow~\cite{Chow-conj} which gave an explicit permutation basis in this special case.  In each of these settings, the authors use the top-coset construction outlined above.

In order to make progress on the construction of a permutation basis in the general case we need recipes for constructing classes that have support on more than one coset.  Our first main result takes a step in this direction, in the special case when the composition has two parts. This is a natural first case to consider, as the Stanley--Stembridge conjecture is known to be true in the so-called ``abelian case,'' and in that setting, the permutation representations occurring as summands of the dot action are either trivial or correspond to partitions of $n$ with exactly two parts (see \cite{Harada-Precup2019}). We have the following; for precise definitions see Section~\ref{sec.class.defn}.

\begin{THM}[Theorem~\ref{thm.class}] \label{thm.1}
Let $h:[n]\to [n]$ be a Hessenberg function and $\lambda = (\lambda_1, \lambda_2)$ a composition of $n$. Let $0\leq k \leq \lambda_2$. If $\lambda_1 > 1$ then we additionally assume that $h(k+2)=n$. 
Let $v_k$ denote the permutation whose one-line notation is given in~\eqref{eq: vk} in Section~\ref{sec.class.defn} and let 
$\mathcal{S}_k := \{t_i - t_j \mid  i < j \, \textup{ and } \, v_k^{-1}(i) > v_k^{-1}(j) \, \textup{ and } \, 
v_k^{-1}(i) \leq h(v_k^{-1}(j)) \}$. 
Then the function $f^{(k)}_\lambda: \Sn \to \C[t_1, \ldots, t_n]$ defined by 
\begin{eqnarray*}
f^{(k)}_\lambda (yv): = \left\{ \begin{array}{ll}  \prod_{t_a-t_b\in \S_k} (t_{y(a)}-t_{y(b)}) & \textup{ if } v\geq v_k, y\in S_\lambda  \\  
0 & \textup{ otherwise} \end{array} \right.
\end{eqnarray*}
is a well-defined equivariant cohomology class in $H_T^{2|\S_k|}(\Hess(\mathsf{S},h))$.
\end{THM} 

We recover the top-coset classes from the construction above in the special case where $k=\lambda_2$ (since $v_k = v_\lambda$ and $\S_k = v_k(N_h^-(v_\lambda))$ in that case).  When $k<\lambda_2$ our classes are supported on a union of right cosets and we can give (Lemma~\ref{lem.left-coset-support}) a concrete description of their support and the support of any element in the $\Sn$-orbit of $f^{(k)}_\lambda$ under the dot action.  Since these support sets consist of unions of more than one coset in general, proving that the $\Sn$-orbit of $f_\lambda^{(k)}$ is $H_T^*(\pt)$-linearly independent becomes more difficult.  However, we do obtain a linear independence result analogous to the top-coset case mentioned above in the case where $k=\lambda_2-1$, under some additional hypotheses on the Hessenberg function $h$.  Roughly, the result is as follows; see Theorem~\ref{theorem: orbit independence} for the precise statement.

\begin{THM} [Theorem~\ref{theorem: orbit independence}] \label{thm.2}
Let $h: [n] \to [n]$ a Hessenberg function. Let $\lambda=(\lambda_1, \lambda_2)$ be a composition of $n$ such that $h(1) < \lambda_2$. If $\lambda_1>1$, we place additional assumptions on the Hessenberg function $h$ as in Theorem~\ref{theorem: orbit independence} below.  
Then
\begin{enumerate} 
\item the $\Sn$-orbit of $f^{(\lambda_2-1)}_\lambda$ is $H^*_T(\pt)$-linearly independent, and
\item the stabilizer of each element in the $\Sn$-orbit is conjugate to the Young subgroup $S_\lambda$. 
\end{enumerate} 
In particular, the $H_T^*(\pt)$-submodule of $H_T^*(\Hess(\mathsf{S},h))$ spanned by the $\Sn$-orbit of $f_\lambda^{(\lambda_2-1)}$ 
is an $\Sn$-subrepresentation with the same character as $\mathrm{Ind}_{S_\lambda}^{\Sn}(\mathbf{1})\simeq M^{\Par(\lambda)}$, where $\Par(\lambda)$ is the partition of $n$ obtained from $\lambda$ by rearranging the parts to be in decreasing order.
\end{THM}

Finally, we address the question of combining the permutation bases obtained above to form a permutation basis of a larger subrepresentation. Considering such unions is essential since a permutation basis for $H^*_T(\Hess(\mathsf{S},h))$ will generally consist of a collection of permutation bases, one for each induced permutation representation $\mathrm{ind}_{S_\lambda}^{\Sn}(\mathbf{1})$ contained in $H^*_T(\Hess(\mathsf{S},h))$, where the union of all such bases is still $H^*_T(\pt)$-linearly independent. As in the case of a single permutation representation, however, proving the linear independence of such unions of classes can be technically difficult. Nevertheless we are able to prove the linear independence of a union of two such permutation bases in the special case of $\lambda=(1, n-1)$. A rough statement is as follows; for the precise statement see Theorem~\ref{theorem:linear-ind2}. 

\begin{THM}[Theorem~\ref{theorem:linear-ind2}] \label{thm.3}
Let $\lambda=(1,n-1)$ and assume $h$ is a Hessenberg function such that $h(1)<n-1$, $\deg f_\lambda^{(n-1)} = \deg f_\lambda^{(n-2)}$, and $h(i)>i$ for all $i$.  Then the union of the $\Sn$-orbits of $f_{\lambda}^{(n-1)}$ and $f_{\lambda}^{(n-2)}$, i.e., the set
\[
\{w \cdot f_{\lambda}^{(n-1)} \mid w\in \Sn\} \cup \{w\cdot f_\lambda^{(n-2)} \mid w\in \Sn\},
\]
 is $H_T^*(\pt)$-linearly independent.
In particular, the $H_T^*(\pt)$-submodule of $H_T^*(\Hess(\mathsf{S},h))$ spanned by this union of $\Sn$-orbits is an $\Sn$-subrepresentation isomorphic to the direct sum of two copies of the permutation representation with same character as $\mathrm{Ind}_{S_\lambda}^{\Sn}(\mathbf{1})\simeq  M^{(n-1,1)}$.
\end{THM} 

Example~\ref{ex.(n-2,n-1,n,...,n)} below presents an application of our theorem in the case that $h=(n-2, n-1, n, \ldots, n)$.  Although we do not have a complete description of a permutation basis for $H_T^*(\Hess(\mathsf{S},h))$ in that case, our results do yield a basis for the two copies of $M^{(n-1,1)}$ of minimal degree that do occur.  Our example also motivates a statement of a natural follow-up problem which we give in Problem~\ref{problem: linear indep}.

The advantage of the construction from Theorem~\ref{thm.1} is that we obtain explicit combinatorial formulas for the equivariant classes, their support sets, and a clear description of the dot action on each $f_\lambda^{(k)}$.  It is worth emphasizing that this kind of information can be quite difficult to obtain when the classes are defined geometrically.  Moreover, it is this information that gives us the leverage needed to prove our the main theorems regarding $H^*_T(\pt)$-linear independence.  On the other hand, these linear independence results apply only in special cases, particularly the results of Theorem~\ref{thm.3}. In the recent preprint \cite{Cho-Hong-Lee2020}, Cho, Hong, and Lee give a geometric construction of a basis for $H_T^*(\Hess(\mathsf{S}, h))$ in all cases, by using an affine paving of that variety. Although these ``geometric'' classes are linearly independent, they do not in general form a permutation basis with respect to the dot action, and there is no known general, explicit combinatorial formula for the values of these classes at different permutations $w \in \Sn$, except for the special case $h=(2,3,\ldots, n,n)$.  Therefore, it is currently a compelling open question to express our classes -- which are defined purely combinatorially -- in terms of the basis constructed geometrically in~\cite{Cho-Hong-Lee2020}, particularly in the abelian case.  We discuss this further in Section~\ref{sec.class.defn}; see Problem~\ref{prob.BB-classes} below.

We now give a brief overview the contents of this paper.  Section~\ref{sec:background} discusses relevant background material, including the presentation of $H_T^*(\Hess(\mathsf{S}, h))$ via GKM theory, and an overview of useful facts regarding the combinatorics of $\Sn$ and its coset decompositions.  In addition, we provide in Section~\ref{subsec: perm basis program} an expository account of the broader context in which our manuscript should be placed. In particular, we give a clear statement of what we call the ``permutation basis program'', which seeks to solve the Stanley-Stembridge conjecture using the geometry of Hessenberg varieties. In Section~\ref{sec.top.coset},  although the construction of top-coset classes is known to the experts, we formalize the presentation of these equivariant classes. We then define the equivariant cohomology classes studied in this manuscript in Section~\ref{sec.class.defn}.  Our main theorem proves that, under some minor assumptions on the Hessenberg function $h$, these are well-defined classes in $H_T^*(\Hess(\mathsf{S}, h))$ and we are able to give an explicit description of their supports.  We state the problem of connecting our GKM classes to those defined by Cho, Hong, and Lee in Problem~\ref{prob.BB-classes}.  Sections~\ref{subsec: LI for fk} and~\ref{sec.union.orbits} prove the linear independence results appearing in Theorem~\ref{thm.2} and Theorem~\ref{thm.3}, respectively.  We conclude Section~\ref{sec.union.orbits} with an application of our linear independence theorems to the case of $h=(n-2,n-1, n,\ldots, n)$ and the statement of a natural question, to be analyzed in future work, in Problem~\ref{problem: linear indep}.

\medskip
\noindent \textbf{Acknowledgements.} 
We are grateful to the hospitality of the Mathematical Sciences Research Institute in Berkeley, California, and the Osaka City University Advanced Mathematical Institute in Osaka, Japan, where parts of this research was conducted. 
Some of the material contained in this paper are based upon work supported by the National Security Agency
under Grant No. H98230-19-1-0119, The Lyda Hill Foundation, The McGovern
Foundation, and Microsoft Research, while the authors were in residence at the
Mathematical Sciences Research Institute in Berkeley, California, during the summer of 2019.
The first author is supported by a Natural Science and Engineering Research Council Discovery Grant and a Canada Research Chair (Tier 2) from the Government of Canada. The second author is supported in part by NSF DMS-1954001. 
The third author is supported in part by NSF DMS-1800773.


\section{Background}\label{sec:background}

In this section we briefly recall some notation and terminology needed for discussion of Hessenberg varieties and their associated cohomology rings. We refer to \cite{Harada-Precup2019} for a more leisurely account. 
In the final subsection, Section~\ref{subsec: perm basis program}, we also give an expository account of the larger context of this paper, and give explicit statements of the broader research problems to which this paper contributes. 

\subsection{Hessenberg varieties, Hessenberg functions, and the type A root system.}\label{subsec: Hess intro}

The (full) flag variety
$\Flags(\C^n)$ is the collection of sequences of nested linear subspaces of $\C^n$:
\[
\Flags(\C^n) := 
\{ V_{\bullet} = (\{0\} \subset  V_1 \subset  V_2 \subset  \cdots \subset V_{n-1} \subset 
 \C^n)  \mid  \dim_{\C}(V_i) = i \ \textrm{ for all } \ i=1,\ldots,n\}. 
\]
A Hessenberg variety in $\Flags(\C^n)$ is specified by two pieces of data: a Hessenberg function and a choice of an element in $\mathfrak{gl} (n,\C)$.  A \textbf{Hessenberg function} is a non-decreasing function $h: [n] \to [n]$.  In this paper, we consider only Hessenberg functions such that $h(i)\geq i$ for all $i\in [n]$ and implicitly make this assumption for all such functions appearing below.
We frequently write a Hessenberg function by listing its values in sequence, i.e., $h = (h(1), h(2), \ldots, h(n))$.
Now let $\mathsf{X}$ be an $n\times n$ matrix in $\mathfrak{gl} (n,\C)$, which we also consider as a linear operator $\C^n \to \C^n$. Then the \textbf{Hessenberg variety} $\Hess(\mathsf{X},h)$ associated to $h$ and $\mathsf{X}$ is defined to be  
\begin{align}\label{eq:def-general Hessenberg}
\Hess(\mathsf{X},h) := \{ V_{\bullet}  \in \Flags(\C^n) \mid  \mathsf{X} V_i \subseteq 
V_{h(i)} \text{ for all } i\in[n]\} \subset  \Flags(\C^n).
\end{align}
In this paper we focus on a special case of Hessenberg varieties. Let $\mathsf{S}$ denote a regular semisimple matrix in $\mathfrak{gl} (n,\C)$, that is, a matrix which is diagonalizable with distinct eigenvalues.  Then we call $\Hess(\mathsf{S},h)$ a
\textbf{regular semisimple Hessenberg variety}. Note that $h(i)\geq i$ for all $i\in [n]$ implies $\Hess(\mathsf{S}, h)$ is nonempty. The equivariant cohomology of $\Hess(\mathsf{S},h)$ is the main object of study in this paper.

Now we set some notation associated to type $A$ root systems. Let $\fh\subseteq \mathfrak{gl}_n(\C)$ denote the Cartan subalgebra of diagonal matrices, and let $t_i$ denote the coordinate function on $\fh$ reading off the $(i,i)$-th matrix entry along the diagonal.  We denote the root system of $\mathfrak{gl}_n(\C)$ by $\Phi:=\{t_i-t_j\mid i,j\in [n], i\neq j\}$, with the subset of positive roots given by 
\begin{equation*}
\Phi^+:=\{t_i-t_j \mid 1\leq i<j\leq n\}.
\end{equation*} 
The negative roots in $\Phi$ are $\Phi^- := \Phi\setminus \Phi^+$,
and we denote the simple positive roots in $\Phi^+$ by 
\[
\Delta := \{ \alpha_i:= t_i-t_{i+1}\mid 1\leq i \leq n-1\}.
\] 
Given $h:[n]\to [n]$ a Hessenberg function, it will be convenient to consider variants of the above terminology which incorporate the data of $h$. In particular, we define the notation 
\begin{equation}\label{eq: def Phi h minus} 
\Phi_h^- := \{t_i-t_j \in \Phi^- \mid i\leq h(j) \}.
\end{equation} 
It is clear that the set $\Phi_h^-$ is determined by the Hessenberg function $h$, but it is useful to 
also note that $h$ is uniquely determined by $\Phi_h^-$ (since $h(i)\geq i$ for all $i\in [n]$). 

We also recall some terminology concerning inversions. 
The Weyl group in Lie type $A$ is the symmetric group $\Sn$ on $n$ letters.  Given a permutation $w \in \Sn$, the  \textbf{inversion set} of $w$ is given by 
\begin{equation}\label{eq: def inversions of w} 
N(w) := \{\gamma\in \Phi^+ \mid w(\gamma)\in \Phi^-\}.
\end{equation} 
Note that $\gamma= t_i-t_j$ is an inversion of $w$ if and only if $i<j$ and $w(i)>w(j)$.  Thus the pair $(i,j)$ is an inversion of the permutation $w$ in the classical sense if and only if $\gamma = t_i - t_j \in N(w)$.  We also set 
$$
N^-(w) := \{\gamma\in \Phi^-\mid w(\gamma) \in \Phi^+ \}.
$$  It is straightforward to see that $w(N^-(w)) = N(w^{-1})$. Let $\ell(w)$ denote the (Bruhat) length function on $\Sn$.  Then $\ell(w)=|N(w)|=|N^-(w)|$.  If $\gamma = t_i-t_j\in \Phi$ then we denote by $s_\gamma$ the transposition of $\Sn$ swapping $i$ and $j$. We do not differentiate between positive and negative roots with this notation, so in particular $s_\gamma=s_{-\gamma}$.  It is well known that $\ell(ws_\gamma)< \ell(w)$ for $\gamma\in \Phi$ if and only if $\gamma\in N^-(w)$ \cite[Sections 1.6-1.7]{Humphreys-Coxeter}. When $\alpha_i\in \Delta$ we write $s_i:= s_{\alpha_i}$ for the simple reflection swapping $i$ and $i+1$.

\subsection{The equivariant cohomology of $\Hess(\mathsf{S}, h)$ and the dot action representation}\label{subsec: dot action}

In this section we briefly recall some facts about the ordinary and equivariant cohomology rings of regular semisimple Hessenberg varieties, and the definition of the dot action representation on these rings. We refer the reader to \cite{Tymoczko2005, Tymoczko2008, Harada-Precup2019} for more details. 
Let $h: [n] \to [n]$ be a Hessenberg function and $\Hess(\mathsf{S},h)$ the regular semisimple
Hessenberg variety associated to $h$. 
The maximal torus $T$ of diagonal matrices in $\text{GL}(n,\C)$ acts on
$\Flags(\C^n)$ preserving $\Hess(\mathsf{S}, h)$ and  
\begin{equation*}
\Hess(\mathsf{S},h)^{\Tn}= \Flags(\C^n)^{\Tn} \cong \Sn
\end{equation*}
where we identify $S_n$ with the permutation flags in $\Flags(\C^n)$. 
In this setting, the localization theorem of torus-equivariant topology
applies and the inclusion map of the fixed point set into $\Hess(\mathsf{S},h)$ induces an injection,
\begin{align*}
  &\iota: H^*_T(\Hess(\mathsf{S},h)) \into
  H^*_T(\Hess(\mathsf{S},h)^T) = \bigoplus_{w\in \Sn}H_T^*(\pt)\cong \bigoplus_{w\in\Sn}\C[t_1,\dots,t_n].
\end{align*}
For $f \in H^*_T(\Hess(\mathsf{S},h))$, since $\iota$ is
injective, by slight abuse of notation we denote also 
by $f$ its image in $H_T^*(\Hess(\mathsf{S}, h)^T)$. 
For each $w\in \Sn$ we denote by $f(w) \in \C[t_1,\ldots,t_n]$ the
$w$-th component of $f$ in the decomposition above. 

Applying results of Goresky--Kottwitz--MacPherson, one obtains the following concrete description
of the image of $\iota$ as in \cite{Tymoczko2008}: 
\begin{equation}\label{GKM for X(h)}
H^*_{\Tn}(\Hess(\mathsf{S},h)) \cong \left\{ f \in \bigoplus_{w\in\Sn} \C[t_1,\dots,t_n] \left|
\begin{matrix} \text{for all $w\in \Sn$ and $\gamma = t_i-t_j \in N^-(w)\cap \Phi_h^-$, }  \\
\text{ $f(w)-f(ws_\gamma)$ is divisible by $w(\gamma)=t_{w(i)}-t_{w(j)}$.}  \end{matrix} \right.\right\}
\end{equation}
We call the condition described in the right hand side of ~\eqref{GKM for X(h)} the \textbf{GKM condition for $\Hess(\mathsf{S},h)$.}  Since the set $N^-(w)\cap \Phi_h^-$ appearing in~\eqref{GKM for X(h)} is used so frequently, we define the notation 
\begin{equation}\label{eq: def N h minus of w} 
N^-_h(w):=N^-(w)\cap \Phi_h^-.
\end{equation}

Motivated by the above, the \textbf{GKM graph} of the regular semisimple Hessenberg variety $\Hess(\mathsf{S},h)$ is defined as the (labelled, directed) graph with vertex set $\Sn$ and  edges
\[
\xymatrix{w \ar[r]^{\;w(\gamma)\;\;\;} &  ws_\gamma }
\]
where $\gamma \in N_h^-(w)$. Note that $w \rightarrow ws_\gamma$ an edge implies $ws_\gamma< w$, as $\gamma\in N^-(w)$.
The set of labels of the directed edges in the GKM graph of $\Hess(\mathsf{S},h)$ with $w$ as a source is
\begin{equation}\label{eq: labels w source} 
w(N_h^-(w)) = w(N^-(w) \cap \Phi_h^-) = w(N^-(w)) \cap w(\Phi_h^-) = N(w^{-1})\cap w(\Phi_h^-)
\end{equation}
where we have used the fact that $w(N^-(w)) = N(w^{-1})$. 

\begin{example}\label{ex.GKMgraph} Let $n=3$ and $h=(2,3,3)$.  The GKM graph of $\Hess(\mathsf{S},h)$ is as follows: 
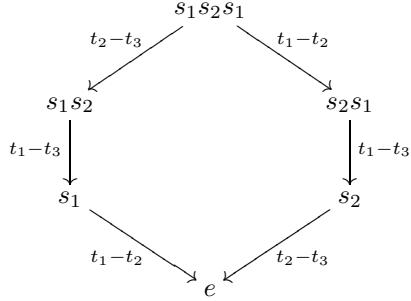
\[\xymatrix{& s_1s_2s_1 \ar[rd]^{t_1-t_2}\ar[ld]_{t_2-t_3} & \\
s_1s_2 \ar[d]_{t_1-t_3} & & s_2s_1\ar[d]^{t_1-t_3} \\
s_1 \ar[rd]_{t_1-t_2} & & s_2 \ar[ld]^{t_2-t_3} \\
& e &\\
}
\]
\end{example}

The GKM graph is the combinatorial data encoding the set of GKM conditions for $\Hess(\mathsf{S},h)$ on the RHS of~\eqref{GKM for X(h)}. When $h = (n,n, \cdots, n)$, we have $\Phi_h^- = \Phi^-$ and $\Hess(\mathsf{S},h)=\Flags(\C^n)$;   the GKM graph of the flag variety is also called the \textbf{Bruhat graph} of $\Sn$.  In this special case, since $N^-(w)$ is a subset of $\Phi^-$ by definition, we see that the set of edges with $w$ as a source in the GKM graph of $\Flags(\C^n)$ is in one-to-one correspondence with $N^-(w)$. Moreover, in this case, the set of these edge labels is $w(N^-(w))=N(w^{-1})$. We can see from~\eqref{GKM for X(h)} that in order to obtain the GKM graph for $\Hess(\mathsf{S},h)$ from the Bruhat graph, we simply delete the edges corresponding to $\gamma$ with $\gamma \not \in \Phi_h^-$. In summary, there are precisely $|\Phi_h^-|=\dim \Hess(\mathsf{S}, h)$ edges adjacent to $w$ in the GKM graph for $\Hess(\mathsf{S},h)$ and exactly $|N_h^-(w)|$ edges with $w$ as a source.

Now we recall the $\Sn$-action, often called the ``dot action'', on $H^*_{\Tn}(\Hess(\mathsf{S},h))$ and $H^*(\Hess(\mathsf{S},h))$ 
constructed explicitly by 
the third author in \cite{Tymoczko2008}. 
First, we define an
$\Sn$-action on the polynomial ring $\C[t_1,\ldots,t_n]$ in the
standard way by permuting the indices of the variables, i.e.~for $t_i \in \C[t_1,\ldots, t_n]$ and $v \in \Sn$ we define 
$v ( t_i ) := t_{v(i)}$.
This induces an $\Sn$-action on $\C[t_1,\ldots,t_n]$ by
$\C$-linear ring homomorphisms. By~\eqref{GKM for X(h)}, 
an element $f \in H^*_T(\Hess(\mathsf{S},h))$ is specified uniquely 
by a list $(f(w))_{w \in \Sn}$ of polynomials in
$\C[t_1,\ldots,t_n]$ satisfying the GKM conditions.  Given $v \in \Sn$ and $f = (f(w))_{w
  \in \Sn}$, the \textbf{dot action} of $v$ on $f$ is defined by
\begin{equation}\label{def of Tymoczko rep} 
(v\cdot f)(w) := v ( f(v^{-1}w)) \ \textup{ for all } \ w \in \Sn.
\end{equation}
It is straightforward to check that the class $v \cdot f$ also satisfies the GKM conditions, and we therefore obtain a well-defined action of $\Sn$ on $H^*_T(\Hess(\mathsf{S},h))$, called \textbf{the dot action representation}.  
This is a \textit{twisted} group action on equivariant cohomology as it acts non-trivially on the underlying ring of scalars $H_T^*(\pt) \simeq \C[t_1, \ldots, t_n]$---the action on $H_T^*(\pt)$ is the standard action of $\Sn$ on the polynomial ring defined above.
The dot action on the equivariant cohomology $H^*_T(\Hess(\mathsf{S},h))$ induces the dot action on ordinary cohomology $H^*(\Hess(\mathsf{S},h))$ by the forgetful map $\pi: H^*_T(\Hess(\mathsf{S},h)) \to H^*(\Hess(\mathsf{S},h))$. Indeed, the forgetful map is known to be surjective and the dot action preserves the kernel \cite{Tymoczko2008}, hence this induces a well-defined action on $H^*(\Hess(\mathsf{S},h))$. 

\begin{remark}\label{rem.forgetful}
It is known that in the case of regular semisimple Hessenberg varieties, the $T$-equivariant cohomology $H^*_T(\Hess(\mathsf{S},h))$ is a free $H^*_T(\pt)$-module, and that the forgetful map $H^*_T(\Hess(\mathsf{S},h)) \to H^*(\Hess(\mathsf{S},h))$ is the surjection obtained by taking the quotient by the ideal $\langle t_1, t_2,\ldots, t_n \rangle \subseteq H^*_T(\pt) \cong \C[t_1,\ldots,t_n]$. From this it follows that the image of a permutation basis (as a $H^*_T(\pt)$-module) of $H^*_T(\Hess(\mathsf{S},h))$ is a permutation $\C$-basis of $H^*(\Hess(\mathsf{S},h))$.  However, as noted in the introduction, a $H_T^*(\pt)$-linearly independent set need not map to a $\C$-linearly independent set under the natural projection.   
\end{remark}

As discussed in the introduction, Shareshian and Wachs conjectured in \cite{Shareshian-Wachs2016} that the above ``dot action'' representation on $H^*(\Hess(\mathsf{S},h))$ is related to the well-known Stanley--Stembridge conjecture. Specifically, they conjectured a tight relationship between the chromatic Hessenberg function of the incomparability graph of a unit interval order to the dot action on $H^*(\Hess(\mathsf{S},h))$ as defined above; we refer to \cite[Conjecture 10.1]{Shareshian-Wachs2016} for the detailed statement. Shareshian and Wachs' conjecture was proven by Brosnan and Chow \cite{Brosnan-Chow2018}, and independently by Guay-Paquet \cite{Guay-Paquet2016}, in 2015. For the purposes of this paper it suffices to recall that these results imply that the Stanley--Stembridge conjecture would follow from the following conjecture, phrased in terms of the dot action on $H^*(\Hess(\mathsf{S},h))$ (see \cite[Conjecture 10.4]{Shareshian-Wachs2016}). 

\begin{conjecture}\label{conjecture: shareshian wachs}
Let $h: [n] \to [n]$ be a Hessenberg function. Then 
there exists a basis of $H^*(\Hess(\mathsf{S},h))$ that is permuted by the dot action, and such that the stabilizer of each element in the basis is a reflection subgroup.  
\end{conjecture} 

The motivation for this manuscript is to take some steps toward addressing Conjecture~\ref{conjecture: shareshian wachs}, but as discussed in the Introduction and due to the observations in Remark~\ref{rem.forgetful}, we opt below to focus exclusively on the \emph{equivariant} version of Conjecture~\ref{conjecture: shareshian wachs}, since a solution to the equivariant version yields a solution to Conjecture~\ref{conjecture: shareshian wachs}.

Before concluding this subsection we make one more simplifying remark.  Recall that a Hessenberg function $h: [n]\to [n]$ is \textbf{connected} if $h(i)>i$ for all $i\in [n-1]$.  This terminology is due in part to the fact that the corresponding regular semisimple Hessenberg variety $\Hess(\mathsf{S}, h)$ is connected if and only if $h$ is connected \cite[Appendix A]{Anderson-Tymoczko2010}. If $h$ is not connected then it is straightforward to argue that the connected components of $\Hess(\mathsf{S}, h)$ are each isomorphic to a direct product of `smaller' connected regular semisimple Hessenberg varieties (see the analogous argument given in \cite[Theorem 4.5]{Drellich-thesis}). In that case, the dot action on $H^*(\Hess(\mathsf{S}, h))$ is induced from the dot action of a reflection subgroup on the cohomology of this connected component (the equivalent statement for chromatic quasisymmetric functions is very well known; c.f.~\cite[Theorem 1.1(B)]{Abreu-Nigro2020}).  Thus, in order to address Conjecture~\ref{conjecture: shareshian wachs} it suffices to consider only those regular semisimple Hessenberg varieties corresponding to connected Hessenberg functions. On the other hand, many of our theorems below hold for Hessenberg functions without this additional restriction.  We therefore note when this assumption is required.

\subsection{Weyl group combinatorics} \label{subsec: weyl group}

We now take a moment to briefly review and set notation regarding combinatorics of $\Sn$. 
Let $\mu = (\mu_1, \mu_2, \ldots, \mu_\ell)$ be a composition of $n$, that is, $\mu_1, \mu_2 \ldots, \mu_\ell$ are positive integers such that $\mu_1+\mu_2+\cdots+\mu_\ell = n$.  Throughout this section, we let $[\mu]_i = \mu_1+\cdots + \mu_i$ for all $i=1, \ldots, \ell$ and set $[\mu]_0 :=0$.  Note that $[\mu]_1=\mu_1$ and $[\mu]_\ell=n$. 

We define the \textbf{Young subgroup corresponding to $\mu$} to be the subgroup 
\[
S_\mu:= \left< s_i \mid i\notin \{ [\mu]_1, [\mu]_2, \ldots, [\mu]_{\ell-1} \} \right> \subseteq \Sn. 
\]
Any subgroup of $\Sn$ generated by simple reflections is of the form $S_\mu$ for some composition $\mu$ of $n$. Moreover, it is well-known that any reflection subgroup of $\Sn$, i.e., a subgroup of $\Sn$ generated by reflections, is conjugate to a Young subgroup $S_\mu$ for some $\mu$. 

In our computations below, we frequently consider the set of right and left cosets, denoted respectively by $S_\mu \backslash \Sn$ and $\Sn/S_\mu$, of a given Young subgroup $S_\mu$.  The \textbf{shortest length right (respectively left) coset representatives} for $S_\mu\backslash \Sn$ (respectively $\Sn/S_\mu$) are defined as follows: 
\[
{^\mu \Sn}:= \{v\in \Sn\mid  v^{-1}(\alpha_i) \in \Phi^+ \ \textup{ for all } \ i\in [n]\setminus \{ [\mu]_1, [\mu]_2, \ldots, [\mu]_{\ell-1} \}  \}
\]
and
\[
\Sn^\mu := \{v\in \Sn \mid v(\alpha_i )  \in \Phi^+ \ \textup{ for all } \ i\in [n]\setminus \{ [\mu]_1, [\mu]_2, \ldots, [\mu]_{\ell-1} \}  \}.
\]
It follows immediately from the definitions above that 
\begin{equation}\label{eq: left and right coset reps inverse}
({^\mu\Sn})^{-1} = \Sn^\mu. 
\end{equation} 

These shortest coset representatives are useful, among other things, for decomposing arbitrary elements of $\Sn$, as the following well-known lemma states \cite[Prop.~1.10]{Humphreys-Coxeter}. 

\begin{lemma}\label{lem.coset-decomp} Let $w\in \Sn$. Then $w$ can be written uniquely as
\begin{enumerate}
\item $w=yv$ for some $y\in S_\mu$ and $v\in {^\mu\Sn}$, and
\item $w=v'y'$ for some $y'\in S_\mu$ and $v'\in \Sn^\mu$.
\end{enumerate} 
Moreover, for such $y, y' \in S_\mu$ and $v \in {^\mu\Sn}$ and $v' \in \Sn^\mu$, we have $\ell(w) = \ell(y)+\ell(v)=\ell(v')+\ell(y')$.
\end{lemma}

\begin{remark}\label{remark: finding coset reps} 
The factors $y$ and $v$ in the decomposition of $w$ given in Lemma~\ref{lem.coset-decomp}(1) have a straightforward interpretation in terms of the one-line notation of $w$, as we now describe.  In order to obtain the one-line notation for $v$, rearrange the values of $\{[\mu]_i+1, [\mu]_i+2, \ldots, [\mu]_{i+1}\}$ in the one-line notation of $w$ to be in increasing order from left to right, for each $i=0, \ldots, \ell-1$.  The result is the one-line notation for $v$, which is the shortest right coset representative of $w$ in $S_\mu \backslash \Sn$.  Now $y$ is simply the element of $\Sn$ which permutes the sets $\{[\mu]_i+1, [\mu]_i+2, \ldots, [\mu]_{i+1}\}$ to be in the same order that was found in the original $w$, for each $i=0, \ldots, \ell-1$.
Similarly, 
there is also a simple method for obtaining the decomposition $w=v'y'$ in Lemma~\ref{lem.coset-decomp}(2) from the one-line notation of $w$. Specifically, we obtain the one-line notation of $v'$ by rearranging the values in the one-line notation of $w$ in \textup{positions} $\{[\mu]_i+1, [\mu]_i+2, \ldots, [\mu]_{i+1}\}$ in increasing order from left to right for all $i=0, \ldots, \ell-1$.  In this case, $y'$ is the element of $\Sn$ which permutes the sets $\{[\mu]_i+1, [\mu]_i+2, \ldots, [\mu]_{i+1}\}$ into the same \textit{relative order} as those in the one-line notation for $w$ in {positions} $\{[\mu]_i+1, [\mu]_i+2, \ldots, [\mu]_{i+1}\}$ for all $i=0, \ldots, \ell-1$.
\end{remark}

\begin{example} \label{ex: shortest coset one-line}
Let $n=7$ and $\mu = (4,3)$. Let $w = [6 , 4, 1 , 7 , 2, 5 , 3]$.  Write $w=yv$ for $y\in S_\mu$ and $v\in {^\mu \Sn}$.  From Remark~\ref{remark: finding coset reps} we obtain
\[
y = [4,1,2,3,6,7,5] \ \textup{ and } \ v=[5,1,2,6,3,7,4].
\] 
Similarly, we have $w=v'y'$ for $y' \in S_\mu$ and $v' \in \Sn^\mu$ with 
\[
y' = [3,2,1,4,5,7,6]  \ \textup{ and } \   v' =[1,4,6,7,2,3,5]. 
\]
\end{example} 

We will also use the unique {(Bruhat) maximal element} contained in ${}^{\mu}\Sn$, and similarly for the set of shortest left coset representatives, which can be described explicitly as follows. Let $w_0 = [n,n-1,\ldots,1]$ denote the maximal element of $\Sn$, i.e.~the longest permutation of $S_n$.  Then the maximal element of ${^\mu\Sn}$, denoted herein as $v_\mu$, is the shortest right coset representative of the right coset $S_\mu w_0$ (see \cite[Prop.~2.5.1]{Bjorner-Brenti2005}).  For example, if $n=7$ and $\mu=(4,3)$ as in Example~\ref{ex: shortest coset one-line} then 
\[
v_\mu = [5,6,7,1,2,3,4]. 
\]
Note also that the maximal element of $\Sn^\mu$ is $v_\mu^{-1}$.
From this description of $v_\mu$, the following is straightforward. 
\begin{lemma}\label{lemma: N minus vmu} 
For $\gamma$ a root, we have $\gamma \in N^-(v_\mu)$ if and only if $s_{v_{\mu}(\gamma)} \not \in S_\mu$. 
\end{lemma}

Given a composition $\mu$ of $n$, let $\mu':=(\mu_{\ell}, \mu_{\ell-1}, \ldots, \mu_{1})$ be the composition obtained by reversing the entries.  For example, if $\mu = (3,4,4)$, we obtain $\mu' = (4,4,3)$.  Note that $[\mu']_i = \mu_\ell + \mu_{\ell-1}+\cdots + \mu_{\ell-i+1}$ for all $i=1, \ldots, \ell$.   The correspondence $\mu \mapsto \mu'$ defines an involution on the set of compositions of $n$, since evidently $(\mu')' = \mu$. We will also need the following simple lemma.

\begin{lemma}\label{lemma.mu-mu'} 
Let $\mu=(\mu_1, \mu_2, \cdots, \mu_\ell)$ be a composition of $n$ and $\mu' = (\mu_{\ell}, \mu_{\ell-1}, \cdots, \mu_1)$. The maximal element of ${^\mu\Sn}$ and the maximal element of $\Sn^{\mu'}$ are equal, i.e., $v_\mu = v_{\mu'}^{-1}$. 
\end{lemma} 

\begin{proof} 
 It follows from the discussion above that 
\begin{equation}\label{eq: vmu formula} 
v_\mu ([\mu']_{\ell-i-1}+j) = [\mu]_i+j \  \textup{ for all } \  0\leq i\leq \ell-1 \textup{ and } 1\leq j \leq \mu_{i+1}.
\end{equation} 
Similarly, 
\begin{equation}\label{eq: vmuprime formula} 
v_{\mu'}([\mu]_{i}+j) = [\mu']_{\ell-i-1} +j \  \textup{ for all } \  0\leq i\leq \ell-1 \textup{ and } 1\leq j \leq \mu_{i+1}.
\end{equation} 
It follows from these formulas that $v_\mu = v_{\mu'}^{-1}$, as desired.  The desired result now follows from the fact that the maximal element element of $\Sn^{\mu'}$ is $v_{\mu'}^{-1}$ by~\eqref{eq: left and right coset reps inverse}.
\end{proof}

The next lemma is completely elementary; we include the statement since we use it repeatedly. 

\begin{lemma}\label{lemma.right-to-left}  Let $W$ be a group and $H\subseteq W$ a subgroup. Suppose $H\sigma$ is a right coset of $H$ in $W$, and consider the subgroup $H_\sigma:= \sigma^{-1}H \sigma$. 
Then there is a well-defined bijection
\[
\phi_\sigma: H\backslash W \to W/H_\sigma,\;\; \phi_\sigma (H\tau) = \tau^{-1}\sigma H_\sigma.
\]
Moreover, given two right cosets $H\sigma_1$ and $H\sigma_2$ of $H$ in $W$, we have $\phi_{\sigma_1}(H\tau) \cap \phi_{\sigma_2}(H\tau)\neq \emptyset$ if and only if $H\sigma_1=H\sigma_2$, and therefore the left cosets $\phi_{\sigma_1}(H\tau)$ and $\phi_{\sigma_2}(H\tau)$ of the (possibly distinct) subgroups $H_{\sigma_1}$ and $H_{\sigma_2}$ in $W$ are either disjoint or equal.
\end{lemma}

We apply Lemma~\ref{lemma.right-to-left} below to obtain a bijection between the right cosets $S_\mu \backslash \Sn$ and left cosets $\Sn/ S_{\mu'}$.  We use this correspondence in the following sections to give a simple description of the GKM classes defined therein.  In the special case where $v$ is maximal element of ${^\mu\Sn}$, the map $\phi_v$ from Lemma~\ref{lemma.right-to-left} further induces a bijection between shortest coset representatives.

\begin{lemma}\label{lemma.right-to-left-top-coset}  Let $\mu$ be a composition of $n$ and $v_\mu$ denote the maximal-length element of ${^\mu\Sn}$.  There is a well-defined bijection 
\begin{eqnarray}\label{eqn.left-to-right}
\phi_{v_\mu}: {^\mu\Sn} \longrightarrow \Sn^{\mu'};\; v\mapsto v^{-1}v_\mu,
\end{eqnarray}
and moreover, we have $v_\mu^{-1} S_\mu v_\mu = S_{\mu'}$.
\end{lemma}

\begin{proof} 
Note that $\phi_{v_\mu}$ can be viewed as a restriction to coset representatives of the bijection 
defined in Lemma~\ref{lemma.right-to-left} at the level of cosets.  To prove the desired result, we need only show that this restriction to shortest coset representatives is well-defined.  Let $v\in {^\mu\Sn}$. 
To show that $\phi_{v_\mu}$ is well-defined we need to show $v^{-1}v_\mu \in \Sn^{\mu'}$. To see this, let 
\[
k\in [n]\setminus \{[\mu']_1, [\mu']_2 , \ldots, [\mu']_{\ell-1}\} = [n] \setminus \{\mu_\ell, \mu_\ell + \mu_{\ell-1}, \ldots, \mu_\ell + \cdots+\mu_2\}. 
\]
By definition of $\Sn^{\mu'}$ it now suffices to show 
$v^{-1}v_\mu (\alpha_k) \in \Phi^+$, i.e.~that $v^{-1}v_\mu(k)<v^{-1}v_\mu(k+1)$.  By our assumptions, it follows that we can write $k = [\mu']_i +j$ for some $0\leq i \leq \ell-1$ and $1\leq j < \mu_{\ell-i}$.  The formula~\eqref{eq: vmu formula} implies 
\[
v^{-1}v_\mu(k) = v^{-1}([\mu]_{\ell-i-1}+j) \ \textup{ and } \ v^{-1}v_\mu(k+1) = v^{-1}([\mu]_{\ell-i-1}+j+1).
\]
Now $v^{-1}([\mu]_{\ell-i-1}+j)$ is the position of $[\mu]_{\ell-i-1}+j$ in the one-line notation of $v$ and $v^{-1}([\mu]_{\ell-i-1}+(j+1))$ is the position of $[\mu]_{\ell-i-1}+j+1$ in the one-line notation of $v$.  Thus, we have only to show that $[\mu]_{\ell-i-1}+j$ appears before $[\mu]_{\ell-i-1}+j+1$ in the one-line notation of $v$.  But this follows from the fact that $v\in {^\mu\Sn}$ and $[\mu]_{\ell-i-1}+j \in [n]\setminus \{[\mu]_1, [\mu]_2, \ldots, [\mu]_{\ell-1}\}$.  We conclude $\phi_{v_\mu}$ is indeed well-defined.  The fact that $\phi_{v_\mu}$ is bijective now follows from Lemma~\ref{lemma.right-to-left}.

Finally, recall that $S_\mu$ is generated by the simple reflections $s_k$ with $k\notin [n] \setminus \{[\mu]_1, [\mu]_2, \ldots, [\mu]_{\ell-1}\}$.  To prove $v_\mu^{-1} S_\mu v_\mu = S_{\mu'}$, we show that $v_{\mu}^{-1}(\alpha_k)= \alpha_{m}$ for some $m\notin [n] \setminus \{[\mu']_1, [\mu']_2, \ldots, [\mu']_{\ell-1}\}$.  This implies that conjugation by $v_\mu^{-1} = v_{\mu'}$ maps the generators of $S_\mu$ to those of $S_{\mu'}$.  Since $k\in [n] \setminus \{[\mu]_1, [\mu]_2, \ldots, [\mu]_{\ell-1}\}$, we may write $k = [\mu]_i + j$ for some $0\leq i \leq \ell-1$ and $1\leq j < \mu_{i+1}$.  Applying~\eqref{eq: vmuprime formula} we obtain 
\[
v_{\mu'} (k) = [\mu']_{\ell-i-1}+j \ \textup{ and } v_{\mu'}(k+1) = [\mu']_{\ell-i-1}+j+1.
\]
Thus $v_{\mu'}(\alpha_k) = \alpha_m$ for $m = [\mu']_{\ell-i-1}+j$.  Since $[\mu']_{\ell-i-1} = \mu_\ell + \mu_{\ell-1}+\cdots + \mu_{i+1}$ and $j< \mu_{i+1}$ we have $m\notin [n] \setminus \{[\mu']_1, [\mu']_2, \ldots, [\mu']_{\ell-1}\}$ as desired.
\end{proof}

We end this subsection with two facts that will be used later. 
The first lemma below describes a decomposition of the sets $N(w^{-1})$ and $N^-(w)$ associated to a permutation $w \in \Sn$.  We will frequently apply this statement below in the context of Lemma~\ref{lem.coset-decomp}; a proof can be found in \cite[Section 1.7]{Humphreys-Coxeter}.

\begin{lemma}\label{lem: y z factor}
Let $w = yv \in \Sn$ such that $\ell(w) = \ell(y)+\ell(v)$. Then $N(w^{-1}) = N(y^{-1}) \sqcup yN(v^{-1})$ and $N^-(w) = N^-(v)\sqcup v^{-1}N^-(y)$. 
\end{lemma} 

\begin{example} 
To illustrate the decomposition $N^-(w) = N^-(v) \sqcup v^{-1}N^-(y)$, consider 
$w = [6, 4, 1, 7, 2, 5, 3]$ as in Example~\ref{ex: shortest coset one-line}. In this case $y=[4,1,2,3,6,7,5]$ and 
$v=[5,1,2,6,3,7,4]$.
Then
\[
N^-(v) = \{ t_2 - t_1, t_3 - t_1, t_5 - t_1, t_7 - t_1, t_5 - t_4, t_7-t_4, t_7-t_6 \}
\]
and $N^-(y) = \{ t_2 - t_1, t_3 - t_1, t_4 - t_1, t_7-t_5, t_7-t_6 \}$ so
\[
v^{-1} N^-(y) = \{ t_3 - t_2, t_5 - t_2, t_7 - t_2, t_6-t_1, t_6-t_4 \}. 
\]
The reader can then check that $N^-(w) = N^-(v) \sqcup v^{-1}N^-(y)$.
\end{example}

We also take a moment to recall a criterion for determining Bruhat order in the
Weyl group $S_n$ (see e.g. \cite{Bjorner-Brenti2005}). 
For \(w \in S_n,\) denote by $D_R(w)$ the right descent set of $w$, namely,
\[
D_R(w) := \{ i \hsm \vert \hsm w(i) > w(i+1), 1 \leq i \leq n-1 \}.
\]
For example, if $w = [3,6,8,4,7,5,9,1,2]$ the
descent set is $D_R(w) = \{3, 5, 7\}$. The following is frequently called the \textbf{tableau criterion}~\cite[Theorem 2.6.3]{Bjorner-Brenti2005}.

\begin{theorem}[The tableau criterion] \label{theorem: tableau criterion}
  For $w,v \in S_n$, let $w_{i,k}$ denote the $i$-th element in the
  increasing rearrangement of $w(1), w(2), \ldots, w(k)$, and
  similarly for $v_{i,k}$. Then $w \leq v$ in Bruhat order if and only
  if $w_{i,k} \leq v_{i,k}$ for all $k \in D_R(w)$ and $1\leq i \leq k$.
\end{theorem}

\subsection{Permutation bases and the Stanley-Stembridge conjecture} \label{subsec: perm basis program}

As we indicated in Section~\ref{subsec: dot action}, the main motivation for this manuscript is the study of the Stanley--Stembridge conjecture, reformulated by Shareshian and Wachs \cite{Shareshian-Wachs2016} into a question about the dot action representation on the cohomology ring $H^*(\Hess(\mathsf{S},h))$ of regular semisimple Hessenberg varieties, as recorded in Conjecture~\ref{conjecture: shareshian wachs} above. To address this problem, we therefore seek to explicitly build permutation bases in $H^*(\Hess(\mathsf{S},h))$ whose stabilizers are reflection subgroups. In fact, in order to achieve this, we first study the analogous question in 
\emph{equivariant} cohomology instead. Specifically, we propose to construct a $H^*_T(\pt)$-module basis of  the free $H^*_T(\pt)$-module $H^*_T(\Hess(\mathsf{S},h))$ consisting of equivariant classes permuted by the dot action and whose stabilizers are reflection subgroups.  We could then project such a basis to ordinary cohomology $H^*(\Hess(\mathsf{S},h))$ using the forgetful map from equivariant to ordinary cohomology.  By Remark~\ref{rem.forgetful}, the projected basis in $H^*(\Hess(\mathsf{S},h))$ would have the desired properties. At first glance, this strategy may seem counterintuitive since equivariant cohomology is much larger than ordinary cohomology, so one may expect the problem to be more difficult. However, as is frequently the case, the additional structure on $H^*_T(\Hess(\mathsf{S},h))$ can frequently make it more tractable (and indeed, as we saw above, the original definition of the dot action was made possible by the GKM theory on equivariant cohomology). 

Based on this point of view, we propose to study the following question: 
\begin{equation}\label{StanleyStembridge Hess version}
\begin{minipage}{0.7\linewidth}
Does there exist a $H^*_T(\pt)$-module basis $\mathcal{B}$ of the free $H^*_T(\pt)$-module $H^*_T(\Hess(\mathsf{S},h))$ which is permuted by the dot action, and such that the stabilizer $\mathrm{Stab}(b) \subseteq S_n$ for any $b \in \mathcal{B}$ is a reflection subgroup? 
\end{minipage}
\end{equation}

The question posed above is well-known among the experts and we do not claim any originality. Moreover, there are already results in the literature which can be interpreted in terms of this question, as we discuss in more detail below. However, as far as we are aware,~\eqref{StanleyStembridge Hess version} has not previously been recorded explicitly in the literature in this form. 
As such we take a moment to discuss the problem and to propose some methods of attack. 

First of all, we expect that GKM theory will be a critical tool for addressing~\eqref{StanleyStembridge Hess version}, just as it was for the original definition of the dot action. There are some inherent challenges in this approach, however. One such challenge is that, in general it is non-trivial to explicitly construct, by purely combinatorial means, an element in the RHS of~\eqref{GKM for X(h)}, i.e., an element in the GKM description of equivariant cohomology. To put it another way, while there do exist formulas for the restrictions to $T$-fixed points of special equivariant cohomology classes of GKM spaces which have, for example, concrete \emph{geometric} descriptions---e.g.~equivariant Schubert classes, or Chern classes of equivariant vector bundles---it is in general difficult to arrive at a purely combinatorial algorithm producing a list of polynomials $(f(w))_{w \in \Sn}$, with $f(w) \in H^*_T(\pt)$, which together satisfy the GKM compatibility (divisibility) conditions. Thus, it is non-trivial to explicitly construct candidates for permutation bases in $H^*_T(\Hess(\mathsf{S}, h))$. Another challenge is that it is difficult in general to prove that a set of GKM classes is $H^*_T(\pt)$-linearly independent, i.e., they satisfy no $H^*_T(\pt)$-linear relations. This is because a GKM class is realized as a vector of polynomials, with coordinates indexed by $T$-fixed points, and the question of linear independence then becomes a complicated linear algebra problem over the polynomial ring $H^*_T(\pt)\simeq \C[t_1, \ldots, t_n]$. This being said, it is not hard to see (and has been noticed before) that if the set has computationally convenient properties, such as ``poset-upper-triangularity'' with respect to Bruhat order on $S_n$ as discussed in \cite{Harada-Tymoczko2017}, then linear independence can be deduced.  However, in the absence of such vanishing properties, the linear algebra over $H^*_T(\pt)$ is not so straightforward. 

Despite these challenges, some results which partly address~\eqref{StanleyStembridge Hess version} already appear in the literature. For instance, Abe, Horiguchi, and Masuda give an explicit presentation of the cohomology ring of $H^*(\Hess(\mathsf{S},h))$ in the special case when $h=(h(1), n, n, \ldots,n)$ in \cite{Abe-Horiguchi-Masuda2019};  their ``$y_i$ classes'', which are a subset of their generators of $H^*(\Hess(\mathsf{S},h))$ in this case, are in fact obtained as images of GKM classes in equivariant cohomology for which they are able to write down an explicit formula. Moreover, it is clear that their ``$y_i$ classes'' form a permutation basis for an $S_n$-subrepresentation in $H^*(\Hess(\mathsf{S},h))$. In another direction, Chow gave in \cite{Chow-conj} a conjectured permutation basis for $H^*(\Hess(\mathsf{S},h))$ in the special case where $h=(2,3,4,\cdots, n-1,n,n)$ (in this case $\Hess(\mathsf{S}, h)$ is the permutohedral variety). Chow's definition of his generators uses the GKM description in equivariant cohomology.  In a recent paper, Cho, Hong, and Lee have shown that Chow's GKM classes have a geometric interpretation in terms of the Bia\l{}ynicki-Birula stratification of the permutohedral variety, and use this to prove Chow's conjecture that these classes are indeed a permutation basis. Thus, this settles the question~\eqref{StanleyStembridge Hess version} in this special case, and it remains to analyze the more general cases.

With the above discussion in mind, we propose to study the following problems, for as a general a Hessenberg function as possible. We refer to this as the ``permutation basis program''.

\begin{PROB} Give a systematic, combinatorial algorithm for constructing GKM classes in $H^*_T(\Hess(\mathsf{S},h))$ beyond those that are already known, and whose stabilizer groups with respect to the dot action are reflection subgroups. 
\end{PROB}

\begin{PROB} Given a GKM class $f \in H^*_T(\Hess(\mathsf{S},h))$, find conditions under which its $\Sn$-orbit 
\[
\{w\cdot f  \mid  w\in \Sn \} 
\]
is $H_T^*(\mathrm{pt})$-linearly independent. 
\end{PROB}

\begin{PROB} Suppose $\{f_\alpha\}_{\alpha \in S}$ is a collection of GKM classes in $H^*_T(\Hess(\mathsf{S},h))$ such that the $\Sn$-orbit of each $f_\alpha$, considered above, is $H^*_T(\pt)$-linearly independent. Find conditions under which the entire collection 
\[
\{ w\cdot f_{\alpha} \mid w\in W \textup{ and $\alpha \in S$ } \}
\]
is $H_T^*(\mathrm{pt})$-linearly independent. 
\end{PROB}

The remainder of this manuscript addresses these problems for a number of special classes of Hessenberg varieties.


\section{GKM classes in $H^*_T(\Hess(\mathsf{S},h))$: the top-coset case} \label{sec.top.coset}

In this and the next section, we address Problem 1 of the ``permutation basis program'' described at the end of  Section~\ref{subsec: perm basis program}.

Specifically, we present in this section a combinatorial construction of GKM classes in $H^*_T(\Hess(\mathsf{S},h))$ which is already well-known to experts and which have the property that the classes evaluate to be non-zero only on a single (``top'' in a suitable sense, to be explained below) coset of a Young subgroup. In particular, we do not claim any originality for the results presented in this section. Then in Section~\ref{sec.class.defn}, we present a variant of this ``top-coset'' construction which results in GKM classes that can be non-zero on more than one coset. We chose this method of exposition for several reasons. First, although the top-coset construction is well-known among experts, as far as we are aware it has not been recorded formally, and in this general form. Second, the intuition behind the construction for both the top-coset case and our construction in Section~\ref{sec.class.defn} is most easily grasped in the top-coset case. Finally, the technical hypotheses on the constructions in this section and the next are such that neither construction is subsumed by the other, so it felt natural to make this distinction clear in the exposition.  We emphasize again that the construction given in the present section has appeared in special cases in the work of Abe--Horiguchi--Masuda \cite{Abe-Horiguchi-Masuda2019}, Chow \cite{Chow-conj}, and Cho--Hong--Lee \cite{Cho-Hong-Lee2020}.

We begin with a lemma which decomposes a certain set of edges in the GKM graph; intuitively, the idea is that some of the edges ``remain'' in a fixed (''top'') coset, while the others point ``down'' toward lower (``non-top'') cosets. The precise statement is in Lemma~\ref{lemma: top-coset-edges}. 
Throughout this section, we fix a composition $\mu=(\mu_1, \mu_2, \ldots, \mu_\ell)$ of $n$ and let $S_\mu$ denote the corresponding Young subgroup. Let $v_\mu$ denote the unique maximal element of ${^\mu\Sn}$ as introduced in Section~\ref{subsec: weyl group}. We refer to the right coset $S_\mu v_\mu$ of $S_\mu$ corresponding to this maximal element as the ``top coset''.  Also recall that edges in the GKM graph with $w$ as a source are indexed by the set $N_h^-(w)$, as in~\eqref{eq: labels w source}. Moreover, by Lemmas~\ref{lem.coset-decomp} and~\ref{lem: y z factor} we know that if $w = yv_\mu$ for $y \in S_\mu$ then $N^-(w) = N^-(v_\mu)\sqcup v_\mu^{-1}N^-(y)$. Thus we have 
\[
N_h^-(w) := N^-(w) \cap \Phi_h^- = (N^-(v_\mu)\sqcup v_\mu^{-1}N^-(y)) \cap \Phi_h^- = 
(N^-(v_\mu) \cap \Phi_h^-) \sqcup (v_\mu^{-1}N^-(y) \cap \Phi_h^-).
\] 
We can now state the lemma.

\begin{lemma}\label{lemma: top-coset-edges} 
Let $w = yv_\mu \in S_\mu v_\mu$ be an element in the top coset of $S_\mu$ where $y \in S_\mu$. Consider an edge of the GKM graph for $\Hess(\mathsf{S},h)$, 
\[
w \xrightarrow{\; w(\gamma) \; }ws_\gamma \ \textup{ for some } \gamma\in N_h^-(w) = 
(N^-(v_\mu) \cap \Phi_h^-) \sqcup (v_\mu^{-1}N^-(y) \cap \Phi_h^-). 
\]
Then 
\begin{enumerate} 
\item if $\gamma\in N^-_h(v_\mu) \cap \Phi_h^-$ then $ws_\gamma \in S_\mu v$ for some $v\in {^\mu\Sn}$ with $v<v_\mu$ and 
\item if $\gamma \in v_\mu^{-1}N^-(y)\cap \Phi_h^-$ then $ws_\gamma\in S_\mu v_\mu$. 
\end{enumerate} 
\end{lemma}

\begin{proof} 
Suppose $\gamma\in N^-(v_\mu) \cap \Phi_h^-$.
Then by Lemma~\ref{lemma: N minus vmu} we know $s_{v_\mu(\gamma)} \not \in S_\mu$, so $ws_\gamma = ys_{v_\mu(\gamma)} v_\mu \not \in S_\mu v_\mu$. Hence $w s_\gamma \in S_\mu v$ for some $v \in {^\mu}\Sn$ with $v\neq v_\mu$. Since $v_\mu$ is the unique maximal element of ${^\mu\Sn}$ we get $v < v_\mu$. 
On the other hand, suppose $\gamma \in v_\mu^{-1}N^-(y)\cap \Phi_h^-$. Then $v_\mu(\gamma) \in N^-(y)$ and $y\in S_\mu$ imply that $s_{v_\mu(\gamma)} \in S_\mu$. This in turn means that $w s_\gamma = y v_\mu s_\gamma = y s_{v_\mu(\gamma)} v_\mu$ lies in the top coset $S_\mu v_\mu$. 
This completes the proof. 
\end{proof}

We can now define the top-coset GKM classes.  We provide a proof for the record. 

\begin{proposition}\label{prop: top coset class} 
Let $\mu = (\mu_1, \mu_2, \cdots, \mu_\ell)$ be a composition of $n$ and let $S_\mu$ denote the associated Young subgroup. Let $v_\mu\in {^\mu\Sn}$ denote the maximal-length right coset representative in ${^{\mu}\Sn}$. Let 
\[
f_\mu (w) := \left\{ \begin{array}{cl} \prod_{t_i-t_j\in N^-_h(v_\mu)} (t_{w(i)} - t_{w(j)}) & \textup{ if } w=yv_\mu, \textup{ for some } y\in S_\mu\\
0 & \textup{ otherwise. }   \end{array}\right.
\]
Then $f_\mu \in H_T^{2 |N_h^-(v_\mu)|}(\Hess(\mathsf{S},h))$, or in other words, $f_\mu$ satisfies the GKM conditions of~\eqref{GKM for X(h)}.  Moreover, $y\cdot f_\mu = f_\mu$ for all $y\in S_\mu$.
 \end{proposition}

\begin{proof} Consider an edge $w \xrightarrow{\;w(\gamma)\;} ws_\gamma$ of the GKM graph of $\Hess(\mathsf{S},h)$.
We take cases. 

 If neither $w$ or $ws_\gamma$ is contained in the top coset $S_\mu v_\mu$ then by definition of $f_\mu$ we have $f_\mu(w)=f_\mu(w s_\gamma) = 0$ so the difference $f_\mu(w) - f_\mu(ws_\gamma)$ is equal to $0$ and the GKM condition for this edge trivially holds.  
 
 Next suppose $w$ and $ws_\gamma$ are both contained in the top coset $S_\mu v_\mu$. 
 In this case, by definition of $f_\mu$ we have
 \[
 f_\mu (ws_\gamma) = \prod_{t_i-t_j\in N^-_h(v_\mu)} (t_{ws_\gamma(i)} - t_{ws_\gamma(j)}) = s_{w(\gamma)}\left( \prod_{t_i-t_j\in N^-_h(v_\mu)} (t_{w(i)} - t_{w(j)})  \right) =s_{w(\gamma)}(f_\mu(w)).
 \]
 Thus $w(\gamma)$ divides $f_\mu(w)-s_{w(\gamma)}(f_\mu(w)) = f_\mu(w)-f_{\mu}(ws_\gamma)$, as required.

Note that we cannot have $w\in S_\mu v$ and $ws_\gamma \in S_\mu v_\mu$ for some $v\in {^\mu\Sn}$ with $v<v_\mu$ since in that case we get 
\[
ws_\gamma \leq w \Rightarrow v_\mu \leq v
\]
by \cite[Proposition 2.5.1]{Bjorner-Brenti2005}. 
which contradicts the assumption that $v<v_\mu$. This implies that the only remaining case to check is when $w\in S_\mu v_\mu$ and $ws_{\gamma} \in S_{\mu} v$ for some $v\in {^\mu\Sn}$ with $v<v_\mu$. In this case, we get
\[
f_\mu(w) - f_\mu (ws_\gamma) = \prod_{t_i-t_j\in N^-_h(v_\mu)} (t_{w(i)} - t_{w(j)})
\]
because $f_\mu(w s_\gamma) = 0$. 
Moreover, by Lemma~\ref{lemma: top-coset-edges}, we know that we are in the situation when $\gamma = t_i-t_j \in N^-(v_\mu) \cap \Phi_h^- = N^-_h(v_\mu)$.  Thus $w(\gamma) = t_{w(i)} - t_{w(j)}$ appears as a factor in the RHS of the above equation, and in particular divides $f_\mu(w) - f_\mu (ws_\gamma)$ as desired.

Finally, suppose $y\in S_\mu$. Since left multiplication by $y^{-1}$ stabilizes all right cosets of $S_\mu$ in $\Sn$ we get that
\[
y\cdot f_\mu(w) = \left\{ \begin{array}{cl} y(f_\mu (y^{-1}w)) & \textup{ if } w\in S_\mu v_\mu\\
0 & \textup{ otherwise. }   \end{array}\right. 
\]
Now we have
\[
y(f_\mu (y^{-1}w)) = y\left( \prod_{t_i-t_j\in N^-_h(v_\mu)} (t_{y^{-1}w(i)} - t_{y^{-1}w(j)})  \right) = \prod_{t_i-t_j\in N^-_h(v_\mu)} (t_{w(i)} - t_{w(j)})  = f_\mu(w)
\]
for all $w\in S_\mu v_\mu$.  This proves $y\cdot f_\mu = f_\mu$.
\end{proof}

\begin{remark} 
The classes constructed in Proposition~\ref{prop: top coset class} can be defined in the more general setting of the equivariant cohomology of a regular semisimple Hessenberg variety contained in the flag variety $G/B$ of any reductive algebraic group $G$.  The GKM graph a regular semisimple Hessenberg variety is well-known, and generalizes the construction presented above (cf.\cite{DPS1992}).  Fix a subgroup $W_J$ in the Weyl group $W$ generated by a subset $J$ of simple reflections.  We can define a GKM class by assigning a nonzero label to each element of the right coset of $W_J$ in $W$ corresponding to the maximal shortest-right-coset representative of $W_J \backslash W$. This nonzero label is a product of roots defined analogously to Proposition~\ref{prop: top coset class}, and yields a well-defined equivariant cohomology class by essentially the same argument.  
\end{remark}

Multiplying the class $f_\mu$ in Proposition~\ref{prop: top coset class} by any $S_\mu$-invariant nonzero homogeneous equivariant cohomology class $g\in H_T^{2j}(\Hess(\mathsf{S},h))$ yields a class of degree $2|N_h^-(v_\mu)|+ 2j$ with the property that $g f_\mu$ is $S_\mu$-invariant and $gf_\mu(w)=0$ unless $w$ is in the right coset of $S_\mu$ indexed by $v_\mu$.  We call any class of this form a \textbf{top coset GKM class} since its support set, i.e. the set of permutations at which it evaluates to be non-zero, is precisely the right coset of the maximal element $v_\mu$ in ${^\mu\Sn}$.

The next lemma tells us that the support set of any class in the $\Sn$-orbit of the top coset class $f_\mu$ has a simple description in terms of certain \textit{left} cosets in $\Sn$. Recall from~\eqref{eq: left and right coset reps inverse} that $(^\mu\Sn)^{-1} = \Sn^\mu$.

\begin{lemma}\label{lem.left-coset-description} 
Let $\mu = (\mu_1, \mu_2, \cdots, \mu_\ell)$ be a composition of $n$ and 
$\mu' = (\mu_\ell, \mu_{\ell-1}, \cdots, \mu_1)$. Let $S_{\mu'}$ be the Young subgroup corresponding to $\mu'$.
 For all $v\in {\Sn^\mu}$ we have
\[
v\cdot f_{\mu}(w) = \left\{ \begin{array}{cl} \prod_{t_i-t_j\in N_h^-(v_\mu)} (t_{w(i)} - t_{w(j)}) & \textup{ if } w=\phi_{v_\mu}(v^{-1})y'  \textup{ for some } y'\in S_{\mu'}\\
0 & \textup{ otherwise }   \end{array}\right.
\]
where $\phi_{v_\mu}: {^\mu \Sn} \to \Sn^{\mu'}$ is the bijection defined in~\eqref{eqn.left-to-right}. In particular, the class $v\cdot f$ has support equal to the left coset of $S_{\mu'}$ in $\Sn$ with shortest coset representative $\phi_{v_\mu}(v^{-1}) := vv_\mu$ and the support of any two classes $v_1\cdot f_\mu$ and $v_2\cdot f_\mu$ where $v_1,v_2\in \Sn^\mu$ with $v_1\neq v_2$ are disjoint.
\end{lemma}

\begin{proof} By definition, $\left( v \cdot f_\mu\right)(w) := v(f_\mu(v^{-1}w))$ is non-zero if and only if $f_\mu(v^{-1}w)\neq 0$.  The latter condition is equivalent to requiring that $v^{-1}w\in S_\mu v_\mu$.  We have
\begin{eqnarray*}
v^{-1}w = yv_\mu \ \textup{ for some } y\in S_\mu &\Leftrightarrow& w= vv_\mu v_\mu^{-1}yv_\mu \ \textup{ for } y\in S_\mu \\
&\Leftrightarrow& w= \phi_{v_\mu}(v^{-1}) y' \ \textup{ for } y':= v_\mu^{-1}yv_\mu \in S_{\mu'}
 \end{eqnarray*}
 where we have used Lemma~\ref{lemma.right-to-left-top-coset} for the last equivalence. 
Thus $\left(v \cdot f_{\mu}\right)(w)\neq 0$ if and only if $w\in \phi_{v_\mu}(v^{-1})S_{\mu'}$. Moreover, if $v^{-1}w \in S_\mu v_\mu$ then 
\[
\left( v \cdot f_{\mu}\right)(w) = v(f_\mu (v^{-1}w)) = v\left( \prod_{t_i-t_j\in N_h^-(v_\mu)} t_{v^{-1}w(i)} - t_{v^{-1}w(j)} \right) = \prod_{t_i-t_j\in N_h^-(v_\mu)} t_{w(i)} - t_{w(j)}
\]
as desired. This proves the first claim. 

Now let $v_1, v_2 \in \Sn^\mu$.  By the above, we know that the support of $v_1 \cdot f_\mu$ (respectively, $v_2 \cdot f_\mu$) is the left coset $\phi_{v_\mu}(v_1^{-1}) S_\mu = v_1 v_\mu S_{\mu'}$ (respectively, $\phi_{v_\mu}(v_2^{-1}) S_\mu = v_2 v_\mu S_{\mu'}$). Applying Lemma~\ref{lemma.right-to-left-top-coset} we know $v_1 v_\mu, v_2 v_\mu \in \Sn^{\mu'}$ are shortest left coset representatives, so $v_1 v_\mu S_{\mu'} \cap v_2 v_\mu S_{\mu'} \neq \emptyset$ if and only if $v_1 v_\mu = v_2 v_\mu$ if and only if $v_1 = v_2$. Hence we conclude if $v_1 \neq v_2$ then the two left cosets $\phi_{v_\mu}(v_1^{-1}) S_\mu$ and $\phi_{v_\mu}(v_2^{-1}) S_\mu$ are disjoint. Thus if $v_1\neq v_2$, then the supports of $v_1 \cdot f_\mu$ and $v_2 \cdot f_\mu$ are disjoint. 
\end{proof}

The main reason for studying top coset classes comes from the following proposition, which is also well-known. 

\begin{proposition}\label{prop: top coset rep}
 Let $\mu = (\mu_1, \mu_2, \cdots, \mu_\ell)$ be a composition of $n$ and let $S_\mu$ be the corresponding Young subgroup of $\Sn$. Let $f_\mu$ be the top-coset GKM class defined in Proposition~\ref{prop: top coset class}.  Then the $\Sn$-orbit of $f_\mu$ under the dot action, given by the set
\[
\{ v \cdot f_\mu \mid v\in \Sn^\mu \}, 
\]
is $H_T^*(\pt)$-linearly independent.  Furthermore, the $H_T^*(\pt)$-subrepresentation of $H_T^*(\Hess(\mathsf{S},h))$ spanned by this set  is an $\Sn$-subrepresentation with the same character as $\mathrm{ind}_{S_\mu}^{S_n}(\mathbf{1})\simeq M^{\Par(\mu)}$, where $\Par(\mu)$ is the partition of $n$ obtained from $\mu$ by rearranging the parts in decreasing order. 
\end{proposition}

\begin{proof} 
We first prove that the set $\{v \cdot f_\mu  \mid v \in \Sn^\mu \}$ is $H^*_T(\pt)$-linearly independent. To see this, suppose that there exist polynomials $c_v \in H^*_T(\pt)$ such that 
\begin{equation}\label{eq: HTpt dependence}
\sum_{v \in \Sn^\mu} c_v \, v \cdot f_\mu = 0 \in H^*_T(\Hess(\mathsf{S},h)).
\end{equation} 
The above equality takes place in $H^*_T(\Hess(\mathsf{S},h))$ which we may identify with its GKM description, as a subring of $\bigoplus_{w \in \Sn} H^*_T(\pt)$. In particular, ~\eqref{eq: HTpt dependence} holds if and only if 
\begin{equation}\label{eq: restrict dependence to fixed points}
\sum_{v \in \Sn^\mu} c_v  \left(v\cdot f_\mu\right)(w) = 0 \, \, \textup{ for all } \, \, w \in \Sn.
\end{equation}
By Lemma~\ref{lem.left-coset-description} the classes $v \cdot f_\mu$ have disjoint supports, so that for any $w \in \Sn$ there exists at most one $v \in \Sn^\mu$ such that $\left(v\cdot f_\mu\right)(w) \neq 0$. Let $w \in \Sn$ and suppose $\left(v\cdot f_\mu\right)(w) \neq 0$ for some $v \in \Sn^\mu$. Then $\left( v' \cdot f_\mu \right)(w) = 0$ for all $v' \in {^\mu}\Sn$ with $v' \neq v$ so~\eqref{eq: restrict dependence to fixed points} implies 
\[
c_v \left( v \cdot f_\mu \right)(w) = 0 \in H^*_T(\pt).
\]
Since $H^*_T(\pt)$ is a polynomial ring over $\C$ and in particular an integral domain, the fact that $(v \cdot f_\mu)(w) \neq 0$ implies $c_v = 0$. Now the fact that $c_v = 0$ for all $v \in \Sn^\mu$ follows from the fact that for any $v \in \Sn^\mu$ there exists at least one $w \in \Sn$ with $(v\cdot f_\mu)(w) \neq 0$, as can be seen from the explicit description of the support of $v \cdot f_\mu$ in Lemma~\ref{lem.left-coset-description}.

To see that the span of $\{v \cdot f_\mu  \mid  v \in \Sn^\mu \}$ is an $\Sn$-submodule (with $H_T^*(\pt)$-coefficients) isomorphic to $\mathrm{ind}_{S_\mu}^{S_n}(\mathbf{1})$ it suffices to show that the stabilizer subgroup of  $f_\mu$ is $S_\mu$. This is clear as $y\cdot f_\mu = f_\mu$ for all $y\in S_\mu$ by Proposition~\ref{prop: top coset class} and $v\cdot f_\mu \neq f_\mu$ for all $v\in \Sn^\mu$ with $v\neq e$ by Lemma~\ref{lem.left-coset-description}. This completes the proof. 
\end{proof}

\begin{example} Let $n=3$ and $h=(2,3,3)$ as in Example~\ref{ex.GKMgraph}.  The following three classes in $H_T^*(\Hess(\mathsf{S},h))$ give the $S_n$-orbit of $f=f_\mu$ for $\mu = (1,2)$. \textup{(}Note that in this case, $\Sn^\mu = \{e, s_1, s_2s_1 \}$.\textup{)}
\[
\begin{array}{c|cccccc}
\empty & e & s_1 & s_2 & s_1s_2 & s_2s_1 & s_1s_2s_1\\ \hline
f &   0 & 0 & 0 & t_1-t_3 & 0 & t_1-t_2 \\ 
s_1\cdot f &   0 & 0 & t_2-t_3 & 0 & t_2-t_1 & 0 \\
s_2s_1\cdot f &  t_3-t_2 & t_3-t_1 & 0 & 0 & 0 & 0
\end{array}
\]
Now $\mathrm{span}_{H_T^*(\pt)}\{ f, s_1\cdot f, s_2s_1\cdot f \}$ in $H_T^*(\Hess(\mathsf{S},h))$ is an $S_n$-subrepresentation isomorphic to $M^{(2,1)}$.
\end{example}

The discussion above makes it evident that these classes are very special in the sense that the support is just one right coset. The question naturally arises: can we give a variant of this ``top-coset'' construction to systematically and explicitly construct GKM classes whose supports may include more than one coset, and which still have stabilizer subgroups which are reflection subgroups? In the next section we answer this question in the affirmative, under some restrictions on the Hessenberg function $h$.


\section{GKM classes in $H^*_T(\Hess(\mathsf{S},h))$ for two-part compositions}\label{sec.class.defn}

In the previous section, we explained how to construct GKM classes in $H^*_T(\Hess(\mathsf{S},h))$ which are supported on a single (``top'') coset of a Young subgroup. Although this property does make these classes more computationally tractable, this is a highly restrictive condition.  In this section, under some technical hypotheses on $h$, we construct GKM classes which can be non-zero on more than one coset. Motivated by the ``abelian case'' as discussed in the introduction, our analysis focuses on compositions of $n$ with two parts.

The setting for this section is as follows. Let $\lambda = (\lambda_1, \lambda_2)$ be a composition of $n$ with two parts.
 Then $S_\lambda= \left< s_i\mid i\neq \lambda_1 \right>$ is the associated Young subgroup.
In order to define our GKM classes, we further decompose the set ${^{\lambda}\Sn}$ of shortest right-coset representatives for $S_\lambda$ as follows. We need some preparation. Consider the composition $\mu = (1,n-1)$. From the discussion in Section~\ref{subsec: weyl group} it is not hard to see that the set of shortest right coset representatives ${^{\mu}\Sn}$ is
given by 
\[
{^{\mu}\Sn} = \{ e, s_1, s_1s_2, \ldots, s_1s_2\cdots s_{n-1} \}.
\]
We define 
\begin{equation}\label{eq: def uk} 
u_k := s_1s_2 \ldots s_k
\end{equation} 
for $k$ with $1 \leq k \leq n-1$ and $u_0:=e$. The maximal element of ${^\mu\Sn}$ is then $u_{n-1}$. Note that the one-line notation for $u_k$ has a $1$ in position $k+1$, and all other entries in increasing order.  Moreover, it is straightforward to check that two permutations $v,w\in \Sn$ are in the same right coset of $S_\mu$ if $1$ is in the same position in their one-line notation, that is, if $v^{-1}(1) = w^{-1}(1)$. 
Returning now to the coset representatives ${^{\lambda}\Sn}$ for $\lambda = (\lambda_1, \lambda_2)$, in this section we denote the maximal element in ${}^{\lambda}\Sn$ by $v_{\lambda_2}$. 
It can be computed explicitly in this case to be 
\[
v_{\lambda_2} = [ \underbrace{\lambda_1 +1, \lambda_1+2, \cdots, n}_{\lambda_2 \textup{ entries}}, \underbrace{1, 2, \cdots, \lambda_1}_{\lambda_1 \textup{ entries}} ] 
\]
which means 
\[
v_{\lambda_2}^{-1} = [ \underbrace{\lambda_2+1, \lambda_2+2, \ldots, n}_{\lambda_1 \textup{ entries}}, \underbrace{1, 2, \ldots, \lambda_2}_{\lambda_2 \textup{ entries}} ].
\]
Since $v_{\lambda_2}$ has a $1$ in the $(\lambda_2+1)$-st entry, it follows that $v_{\lambda_2}$ is contained in the right coset $S_\mu u_{\lambda_2}$ of $S_\mu$. Indeed we have
\begin{eqnarray}\label{eqn.v0formula}
v_{\lambda_2} = v_0 u_{\lambda_2} \ \textup{ for } \ 
v_0 := [ 1, \underbrace{\lambda_1+1, \lambda_1+2, \cdots, n}_{\lambda_2 \textup{ entries}}, \underbrace{2, 3, \cdots, \lambda_1 }_{\lambda_1-1 \textup{ entries}} ] \in S_\mu
\end{eqnarray}
and we can also compute 
\begin{equation}\label{eq: v0 inverse one-line}
v_0^{-1}=[1, \lambda_2+2,\ldots n, 2, 3, \ldots, \lambda_2+1]\in S_\mu.
\end{equation}
We now focus on the elements of ${^\lambda}\Sn$ of the form $v_0 u_{k}$ for $0\leq k \leq \lambda_2$. Define
\[
v_k: = v_0u_k.
\]
The one line notation for $v_k$ is 
\begin{equation}\label{eq: vk} 
v_k = [ \underbrace{\lambda_1+1, \lambda_1+2, \cdots, \lambda_1+k-1, \lambda_1+k}_{k \textup{ entries}}, \underbrace{1}_{\textup{$(k+1)$-st entry}}, \underbrace{\lambda_1+k+1, \cdots, n-1, n}_{\textup{ $(k+2)$-nd to $(\lambda_2+1)$-st entry}}, \underbrace{2,3, \cdots, \lambda_1}_{\textup{ last $\lambda_1-1$ entries}} ] 
\end{equation}
from which it follows that $v_k$ indeed lies in ${^\lambda}\Sn$. We define $({^\lambda}\Sn)_0$ to be the set of such $v_k$, i.e. $({^\lambda}\Sn)_0 := \{ v_0, v_1, \cdots, v_{\lambda_2}\}$. 

We note two facts for future use. 
 First, the one-line notation for $v_k^{-1}$ is
\begin{equation}\label{eq: vk inverse}
v_k^{-1} = [ k+1, \lambda_2+2, \ldots, n, 1, 2, \ldots, k,\widehat{k+1},k+2 \ldots, \lambda_2, \lambda_2+1 ].
\end{equation}
Second, since $v_0$ is contained in $S_\mu$ and the $u_k$ are shortest-coset representatives in ${^\mu}\Sn$, from Lemma~\ref{lem.coset-decomp} we know $\ell(v_k) = \ell(v_0) + \ell(u_k)$.

\begin{remark}\label{remark: lambda is mu} 
In the case that $\lambda = \mu = (1,n-1)$, i.e.~when $\lambda_1 = 1$ and $\lambda_2=n-1$, then from~\eqref{eqn.v0formula} it follows that $v_0$ is equal to the identity permutation, and $u_k=v_k$ for all $0 \leq k \leq \lambda_2=n-1$. So in this case, $({^{\lambda}\Sn})_0 = {^{\mu}\Sn} = \{e, u_1, u_2,\cdots, u_{n-1}\}$. 
\end{remark} 

We focus on this subset $({^\lambda}\Sn)_0$ of ${^\lambda}\Sn$ because it is particularly well-behaved under the Bruhat order.  
To see this, we begin with the following simple lemma. 

\begin{lemma} \label{lemma: w less than v}
Let $\lambda=(\lambda_1, \lambda_2)$ be a composition of $n$ with two parts and suppose $v, w \in {^{\lambda}\Sn}$. Then $w \leq v$ in Bruhat order if and only if $w^{-1}(k) \leq v^{-1}(k)$ for all $1 \leq k \leq \lambda_1$. 
\end{lemma}

\begin{proof} 
This follows from a straightforward application of the tableau criterion in Theorem~\ref{theorem: tableau criterion} together with the fact that a shortest coset representative $v \in {^{\lambda}\Sn}$ is uniquely determined by the locations of the entries $\{1,2,\ldots, \lambda_1\}$, i.e., the set $\{v^{-1}(1), v^{-1}(2), \cdots, v^{-1}(\lambda_1)\}$. 
\end{proof} 

Using Lemma~\ref{lemma: w less than v} above we can show the following.

\begin{lemma}\label{lem.bruhat-order} Let $\lambda = (\lambda_1, \lambda_2)$ be a composition of $n$ as above. Then: 
\begin{enumerate}
\item $({^\lambda}\Sn)_0 = \{v\in {^\lambda\Sn} \mid v_0\leq v\}$, and
\item for any $k, j$ with $0 \leq k, j \leq \lambda_2$, we have $v_k \leq v_j$ in Bruhat order if and only if $k \leq j$. 
\end{enumerate}
\end{lemma}

\begin{proof} 
We first prove the case $\lambda = (1,n-1)$, so $\lambda_1=1, \lambda_2=n-1$. Then $v_0=e$, and it is not hard to see that $({^{\lambda}\Sn})_0 = {^{\lambda}\Sn}$. Since $v_0=e$, the first claim is immediate. The second claim follows straightforwardly from the tableau criterion in Theorem~\ref{theorem: tableau criterion} and the fact that $v_k=u_k$ is the permutation whose one-line notation has a $1$ in the $(k+1)$-st position and all other entries are increasing. 

Now suppose $\lambda_1 \geq 2$. From the one-line notation of $v_0$ in~\eqref{eqn.v0formula} and the tableau criterion, it follows that any other shortest coset representative $v \in {^{\lambda}\Sn}$ with $v_0 \leq v$ must have the entries $\{2,3,\ldots, \lambda_1\}$ appearing in the last $\lambda_1-1$ many entries of the one-line notation of $v$. This then implies that $v$ must equal $v_k$ for some $k$ with $0 \leq k \leq \lambda_2$, as can be seen from the one-line notation of $v_k$ in~\eqref{eq: vk}. Conversely, it is immediate from Lemma~\ref{lemma: w less than v} that each $v_k$ satisfies $v_0 \leq v_k$. Hence the first claim is proved. The second also follows 
from the tableau criterion and~\eqref{eq: vk}.
\end{proof}

\begin{lemma}\label{lemma.vk-invs} For all $k$ with $0\leq k \leq \lambda_2$ we have
\[
N(v_k^{-1}) = \{t_1 - t_b \, \mid \, b \in \{\lambda_1+1, \cdots, \lambda_1+k \} \} \sqcup 
\{ t_a - t_b \, \mid \, a \in \{2,3,\cdots, \lambda_1\}, b \in \{\lambda_1+1, \cdots, n\} \}.
\]
\end{lemma} 

\begin{proof} 
By definition,
\begin{equation} 
\begin{split} 
N(v_k^{-1}) & = \{\gamma \in \Phi^+ \, \mid \, v_k^{-1}(\gamma) \in \Phi^- \} \\
& = \{t_a - t_b \, \mid \, a < b, \, \, v_k^{-1}(a) > v_k^{-1}(b) \}. \\ 
\end{split} 
\end{equation}
The claim now follows from the explicit description of the one-line notation of $v_k^{-1}$ given in~\eqref{eq: vk inverse}. 
\end{proof}

We can now define our GKM classes.  Fix $k$ with $0 \leq k \leq \lambda_2$. 
We define a function $f^{(k)}_\lambda : \Sn \to \C[t_1, \ldots, t_n]$ in~\eqref{eqn.cohom-class} below. Under certain additional hypotheses on $k$, the composition $\lambda$, and the Hessenberg function $h$, we will show in Theorem~\ref{thm.class} that $f^{(k)}_\lambda$ is a well-defined equivariant cohomology class in $H_T^*(\Hess(\mathsf{S}, h))$, i.e., the assignment $f^{(k)}_\lambda: \Sn \to \C[t_1, \ldots, t_n]$ satisfies all the GKM compatibility conditions in~\eqref{GKM for X(h)}. To define $f^{(k)}_\lambda$, we first set the notation 
\begin{equation}\label{eq: def Sk}
\mathcal{S}_k := v_kN_h^-(v_k) = N(v_k^{-1})\cap v_k(\Phi_h^-)
\end{equation}
for the set of roots that label the edges in the GKM graph of $\Hess(\mathsf{S},h)$ with source $v_k$
as in~\eqref{eq: labels w source}. 
Now for any $w\in \Sn$, we first write $w=yv$ for unique $y\in S_\lambda$ and $v\in {^\lambda}\Sn$ and then define 
\begin{eqnarray}\label{eqn.cohom-class}
f^{(k)}_\lambda (yv): = \left\{ \begin{array}{ll}  \prod_{t_a-t_b\in \S_k} (t_{y(a)}-t_{y(b)}) & \textup{ if } v\geq v_k  \\  
0 & \textup{ otherwise.} \end{array} \right.
\end{eqnarray}

The following lemma summarizes some properties of the function $f^{(k)}_\lambda$ which follow immediately from the definition. 

\begin{lemma}\label{lemma: fk def}
Let $f^{(k)}_\lambda: \Sn \to \C[t_1, t_2, \cdots, t_n]$ be as defined in~\eqref{eqn.cohom-class}. Then each of the following hold.
\begin{enumerate} 
\item[(1)] The support of $f^{(k)}_\lambda$ is a union of right $S_\lambda$-cosets, and is the set of permutations Bruhat-greater than $v_k$, i.e., 
\[
\mathrm{supp}(f_\lambda^{(k)}) := \{w\in \Sn  \mid  f^{(k)}_\lambda(w) \neq 0 \} = \bigsqcup_{k \leq j \leq \lambda_2} S_\lambda v_j = \{ w \in \Sn  \mid  w \geq v_k \}.
\]
\item[(2)] The element $f^{(k)}_\lambda$ is fixed by $S_\lambda$ under the dot action, 
\[
y \cdot f^{(k)}_\lambda = f^{(k)}_\lambda \ \textup{ for all } \ y\in S_\lambda.
\]
\item[(3)] For any $y \in S_\lambda$ and $w \in \Sn$, we have 
\[
f^{(k)}_\lambda(yw) = y(f^{(k)}_\lambda(w))
\]
where the RHS denotes the standard action of $S_\lambda\subseteq S_n$ on a polynomial in $\C[t_1,\ldots,t_n]$. 
\end{enumerate} 
\end{lemma}

\begin{proof} 
The first equality of (1) follows from the definition~\eqref{eqn.cohom-class} and Lemma~\ref{lem.bruhat-order}, since $f_\lambda^{(k)}(w) = f_\lambda^{(k)}(yv)$ is defined to be nonzero exactly when $w\in S_\lambda v$ for $v\geq v_k$ and $\{v\in {^\lambda\Sn} \mid v \geq v_k\} = \{v_k, \ldots, v_{\lambda_2}\}$. To prove the second equality, first note that the inclusion $\bigsqcup_{k\leq j \leq \lambda_2} S_\lambda v_j \subseteq \{w\in \Sn \mid w \geq v_k\}$ follows from Lemma~\ref{lem.bruhat-order}(2) and Lemma~\ref{lem.coset-decomp}.  On the other hand, let $w\in \Sn$ such that $w\geq v_k$ and write $w=yv$ with $y\in S_\lambda$ and $v\in {^\lambda\Sn}$ as in Lemma~\ref{lem.coset-decomp}.  By \cite[Proposition 2.5.1]{Bjorner-Brenti2005}, $v_k\leq w$ implies $v_k\leq v$.  Thus $v=v_j$ for some $j$ such that $k\leq j \leq \lambda_2$ by Lemma~\ref{lem.bruhat-order} as desired.  This proves the first claim.

To see the second claim, first observe that for $y \in S_\lambda$ the definition of the dot action implies 
\[
(y \cdot f^{(k)}_\lambda)(w) = y(f^{(k)}_\lambda(y^{-1}w))
\]
and since $y \in S_\lambda$, the two elements $y^{-1}w$ and $w$ are in the same right $S_\lambda$-coset. We take cases. If $f^{(k)}_\lambda(w)=0$ then by the above $f^{(k)}_\lambda(y^{-1}w)$ is also equal to $0$, hence $y(f^{(k)}_\lambda(y^{-1}w)) = 0$ also.  If $f^{(k)}_\lambda(w) \neq 0$ then $w = y'v$ for some $y' \in S_\lambda$ and $v \geq v_j$. Then $y^{-1}w = (y^{-1}y')v\in S_\lambda v$ implies $f_\lambda^{(k)}(y^{-1}w)\neq 0$ and by~\eqref{eqn.cohom-class} we obtain 
\begin{eqnarray*}
(y \cdot f^{(k)}_\lambda)(w) &=& y(f^{(k)}_\lambda(y^{-1}y'v)) = y\left(\prod_{t_a-t_b \in \mathcal{S}_k} y^{-1}(t_{y'(a)}-t_{y'(b)}\right)\\
&=& \prod_{t_a-t_b \in \mathcal{S}_k} (t_{y'(a)}-t_{y'(b)}) = f^{(k)}_\lambda(y'v) = f^{(k)}_\lambda(w) 
\end{eqnarray*}
as desired. This proves (2). We now have
\[
f_\lambda^{(k)}(w) = y^{-1}\cdot f_\lambda^{(k)} (w) = y^{-1}(f_\lambda^{(k)}(yw)) \ \textup{ for all } y\in S_\lambda.
\]
Hence (3) follows.
\end{proof}

Our construction recovers the top-coset classes for compositions with two parts that were discussed in the previous section. 

\begin{remark}\label{rem.top-coset-case} In the special case where $k=\lambda_2$, Lemma~\ref{lemma: fk def} tells us that $f_\lambda^{(\lambda_2)}$ is supported on the coset $S_\lambda v_{\lambda_2}$ corresponding to the Bruhat-maximal element of ${^\lambda\Sn}$. In this case, given $w=yv_{\lambda_2}$ we have
\[
f_\lambda^{(\lambda_2)} (w) = \prod_{t_a-t_b\in \S_k} (t_{y(a)} - t_{y(b)}) = \prod_{t_i-t_j \in N_h^-(v_{\lambda_2})} (t_{w(i)}-t_{w(j)}).
\]
This shows that $f_\lambda^{(\lambda_2)}$ is precisely the top-coset GKM class $f_\lambda$ introduced in the previous section.
\end{remark}

The function $f^{(k)}_\lambda: \Sn \to \C[t_1, \cdots, t_n]$ defined above sometimes, but does not always, yields a well-defined class in $H_T^*(\Hess(\mathsf{S},h))$, as we illustrate in the next example.

\begin{example}\label{example: nonexample and ex} Let $n=6$ and fix a Hessenberg function $h=(3,4,5,6,6,6)$.  In this case we have
\[
\Phi_h^- = \{ t_2-t_1, t_3-t_2, t_4-t_3, t_5-t_4, t_6-t_5, t_3-t_1, t_4-t_2, t_5-t_3, t_6-t_4 \}.
\]
For this example, we take $\lambda = (2,4)$.  We get:
\[
({^\lambda}\Sn)_0 = \{v_0, v_1, v_2, v_3, v_4 \}
\]
where 
\[
v_0^{-1}=[1,6,2,3,4,5], v_1^{-1} = [2,6,1,3,4,5], v_2^{-1}=[3,6,1,2,4,5], v_3^{-1}=[4,6,1,2,3,5], v_4^{-1} = [5,6,1,2,3,4].
\]
Consider the case when $k=1$.  We have $N(v_1^{-1}) = \{t_1-t_3, t_2-t_3, t_2-t_4, t_2-t_5, t_2-t_6\}$ and $\S_1 = \{t_1-t_3, t_2-t_5, t_2-t_6\}$ so,
\[
f^{(1)}_\lambda (yv): = \left\{ \begin{array}{ll}  (t_{y(2)}-t_{y(5)})(t_{y(2)}-t_{y(6)})(t_{y(1)}-t_{y(3)})   & \textup{ if } v\in \{v_1, v_2, v_3, v_4\}  \\  
0 & \textup{ otherwise }  \end{array} \right.
\]
For example, we have that 
\[
f^{(1)}_\lambda(v_3) = f^{(1)}_\lambda(v_1) = (t_2-t_5)(t_2-t_6)(t_1-t_3) \textup{ and } f^{(1)}_\lambda(s_4 v_1) = s_4(f_\lambda^{(1)}(v_1)) = (t_2-t_4)(t_2-t_6)(t_1-t_3).
\]
Consider $v_3^{-1}=[4,6,1,2,3,5]$.  Since $t_4-t_2\in \Phi_h^-$ and swapping the numbers 2 and 4 in $v_3^{-1}$ yields the permutation $(s_4v_1)^{-1} = [2,6,1,4,3,5]$ of length strictly less than $v_3^{-1}$ we know that the GKM graph of $\Hess(\mathsf{S},h)$ contains the following edge:
\begin{eqnarray}\label{eqn.edge}
\xymatrix{ v_3 \ar[rr]^{\;t_1-t_4\;}& & s_4v_1 }
\end{eqnarray}
where $t_1-t_4= v_3(t_4-t_2)$.  But $f^{(1)}_\lambda(v_3) - f^{(1)}_\lambda(s_4v_1)$ is not divisible by $t_1-t_4$, so $f^{(1)}_\lambda$ does not satisfy the GKM-conditions.  Now consider the case in which $k=2$.  As $\S_2 = \{ t_2-t_5, t_2-t_6, t_1-t_3, t_1-t_4 \}$ we have
\[
f^{(2)}_\lambda (yv): = \left\{ \begin{array}{ll}  (t_{y(2)}-t_{y(5)})(t_{y(2)}-t_{y(6)})(t_{y(1)}-t_{y(3)})(t_{y(1)}-t_{y(4)}) & \textup{ if } v\in \{v_2, v_3, v_4\}  \\  
0 & \textup{ otherwise. } \end{array} \right.
\]
In this case, the right $S_\lambda$ cosets in the support set of $f_\lambda^{(2)}$ are those with coset representatives $v_2$, $v_3$ and $v_4$. Note that $f^{(2)}_\lambda$ clearly satisfies the GKM conditions for the edge in~\eqref{eqn.edge} since $f_\lambda^{(2)}(s_4v_1) = 0$ and $t_1-t_4$ divides $f_\lambda^{(2)}(v_3)$.  As another example, by similar reasoning as above we obtain another edge of the GKM graph:
\[
\xymatrix{ v_4 \ar[rr]^{\;t_1-t_5\;}& & s_5v_2. }
\]
In this case, we have
\[
f^{(2)}_\lambda(v_4) = (t_2-t_5)(t_2-t_6)(t_1-t_3)(t_1-t_4) = f^{(2)}_\lambda(s_5v_2)
\]
since $s_5$ stabilizes the product $(t_2-t_5)(t_2-t_6)(t_1-t_3)(t_1-t_4)$.  Thus $f^{(2)}_\lambda$ satisfies the GKM conditions for this edge also.  The reader can check that $f^{(2)}_\lambda$ defines an equivariant cohomology class in $H^8_T(\Hess(\mathsf{S},h))$; this fact will also follow from Theorem~\ref{thm.class} below.
\end{example}

The content of the next result, which is also the first main theorem of this manuscript, is that when we impose an additional hypothesis on the integer $k$ in relation to the Hessenberg function $h$, 
then $f^{(k)}_\lambda$ is a well-defined GKM class. 
Theorem~\ref{thm.class} gives us a new construction of GKM-classes in $H_T^*(\Hess(\mathsf{S},h))$ that differs from that in the literature, since now more than one coset may get a non-zero label.

\begin{theorem}\label{thm.class}  Let $h:[n]\to [n]$ be a Hessenberg function and $\lambda = (\lambda_1, \lambda_2)$ a composition of $n$ with exactly two nonzero parts. Let $0\leq k \leq \lambda_2$. If $\lambda_1 > 1$ then we additionally assume that $h(k+2)=n$. 
Then the function $f^{(k)}_\lambda: \Sn \to \C[t_1, \ldots, t_n]$ defined in~\eqref{eqn.cohom-class} is a well-defined equivariant cohomology class in $H_T^{2 |\S_k|}(\Hess(\mathsf{S},h))$.
\end{theorem}

Before beginning the proof, we emphasize that the assumption of $h(k+2)=n$ in the statement above is necessary, as noted for $k=1$ in~Example~\ref{example: nonexample and ex} above.

\begin{proof}[Proof of Theorem~\ref{thm.class}]
To prove the theorem, we must show the following. Let 
\[
\xymatrix{w \ar[rr]^{\;w(\gamma) \; } &&  ws_\gamma }
\]
be an edge of the GKM graph of $\Hess(\mathsf{S},h)$.  Then $w,ws_\gamma \in\Sn$ are permutations such that $\ell(ws_\gamma)<\ell(w)$ and  $w(\gamma) \in N(w^{-1})\cap w(\Phi_h^-)$.  We must prove that $w(\gamma)$ divides $f^{(k)}_\lambda(w) - f^{(k)}_\lambda(ws_\gamma)$. We argue on a case-by-case basis.

\smallskip
\noindent \textbf{Case (1):} Suppose $w$ is not contained in the support of $f^{(k)}_\lambda$, i.e., $f^{(k)}_\lambda(w) = 0$. In this case we claim that $s_\gamma w$ is also not contained in the support of $f^{(k)}_\lambda$. This is because if $ws_\gamma$ is contained in some $S_\lambda v_j$ with $k \leq j \leq \lambda_2$, then $ws_\gamma \geq v_j$ in Bruhat order, which means $w > ws_\gamma \geq v_j \geq v_k$ in Bruhat order. 
By Lemma~\ref{lemma: fk def} this implies $w$ is contained in the support of $f^{(k)}_\lambda$, contradicting our assumption. Hence 
$f^{(k)}_\lambda$ vanishes at both $w$ and $ws_\gamma$, and the claim follows trivially.

\smallskip

\noindent \textbf{Case (2):} We now assume $w$ lies in the support of $f^{(k)}_\lambda$. By Lemma~\ref{lemma: fk def}, this is equivalent to the condition that there exists $j$ with $k \leq j \leq \lambda_2$ such that $w \in S_\lambda v_j$. We write $w=yv_j$ for $y\in S_\lambda$. It will be convenient to divide this further into sub-cases, according to the coset in which $ws_\gamma$ lies. In fact, we first argue that $ws_\gamma$ cannot lie in certain right cosets;  more precisely, we claim that, under the given hypotheses, it cannot happen that $ws_\gamma \in S_\lambda v_\ell$ for $j < \ell \leq \lambda_2$. Indeed, if $ws_\gamma = y_1v_\ell$ for such an $\ell$ and $y_1 \in S_\lambda$ then we have
\[
ws_\gamma \leq w \Rightarrow y_1 v_\ell \leq y v_j \Rightarrow v_\ell \leq v_j \Rightarrow \ell\leq j
\]
where the second implication is by~\cite[Proposition 2.5.1]{Bjorner-Brenti2005} and the third follows from Lemma~\ref{lem.bruhat-order}. Hence we obtain a contradiction.

Throughout the arguments below, we fix $j$ as the integer such that $w \in S_\lambda v_j$ and write $w=yv_j$ for $y\in S_\lambda$. The above discussion implies that the three remaining cases we must consider are as follows: 

\begin{itemize} 
\item[(2-a)] $ws_\gamma \not \in \mathrm{supp}(f^{(k)}_\lambda)$, or equivalently, $ws_\gamma \not \in S_\lambda v$ for any $v \in \{v_k, \cdots, v_{j-1}, v_j\}$, 
\item[(2-b)] $ws_\gamma \in S_\lambda v_j$, or 
\item[(2-c)] $ws_\gamma \in S_\lambda v_\ell$ for some $k \leq \ell < j$. 
\end{itemize}

Before proceeding, we give an explicit description of the root $\beta:= v_j(\gamma)$ in each of these cases.  By assumption, $\gamma\in N_h^-(w) = N^-(w)\cap \Phi_h^-$ and $N^-(w) = N^-(v_j ) \sqcup v_j^{-1}N^-(y)$ by Lemma~\ref{lem: y z factor}. Thus we have
\[
\beta = v_j(\gamma) \in v_j N^-(w) = N(v_j^{-1}) \sqcup N^-(y).
\]
We can decompose the set $N(v_j^{-1})$ appearing in the RHS of the above equation even further. The formula for $N(v_k^{-1})$ for different values of $k$ as given in Lemma~\ref{lemma.vk-invs} implies
\[
N(v_j^{-1} ) = N(v_k^{-1}) \sqcup \{t_1-t_b \mid b\in \{ \lambda_1+k+1, \ldots, \lambda_1+j \}\}, 
\]
and therefore, combining the previous two statements, we obtain 
\begin{eqnarray}\label{eq: 31}
\beta \in N(v_k^{-1}) \sqcup  \{t_1-t_b \mid b\in \{ \lambda_1+k+1, \ldots, \lambda_1+j \}\} \sqcup N^-(y).
\end{eqnarray}
Now consider $ws_\gamma = yv_j s_\gamma = y s_{v_j(\gamma)} v_j = y s_\beta v_j$. 
Write $\beta = t_a-t_b$.  Then we obtain the one-line notation for $s_\beta v_j$ from that of $v_j$ by swapping the entries $a$ and $b$.  Equivalently, the one-line notation for $(s_\beta v_j)^{-1}$ is obtained from that of $v_j^{-1}$ by exchanging the entries \textit{ in positions} $a$ and $b$. Using the formulas for the one-line notation of $v_j^{-1}$ from~\eqref{eq: vk inverse} (or equivalently, the formula for the one-line notation of $v_j$ from~\eqref{eq: vk}) and Lemma~\ref{lemma.vk-invs}, it is now straightforward to check the following characterizations of the three cases (2-a), (2-b), (2-c) above.  
First we consider case (2-a). We claim that if $w s_\gamma \not \in \mathrm{supp}(f^{(k)}_\lambda)$ then $\beta \in N(v_k^{-1}) \cap v_j(\Phi_h^-)$. 
From~\eqref{eq: 31} we know that $\beta$ can lie in one of $3$ sets: 
\[
N(v_k^{-1}), \, \{ t_1 - t_b \, \mid \, b \in \{\lambda_1+k+1, \cdots, \lambda_1+j\}\}, \, \textup{ and } \, N^-(y).
\]
If $\beta \in N^-(y)$ then $s_\beta \in S_\lambda$ and hence $ys_\beta v_j \in S_\lambda v_j$, which would imply $y s_\beta v_j \in \mathrm{supp}(f_\lambda^{(k)})$. Hence this cannot occur. If $\beta \in \{t_1 - t_b \, \mid \, b \in \{\lambda_1 + k+1, \cdots, \lambda_1 + j\} \}$ then from the formula for the one-line notation of $v_j$ in~\eqref{eq: vk} and from the description of shortest coset representatives given in Remark~\ref{remark: finding coset reps} it follows that $s_\beta v_j$ lies in the right coset of an element $v \in \{v_k, \cdots, v_{j-1}\}$, hence $y s_\beta v_j \in \mathrm{supp}(f_\lambda^{(k)})$. Thus this also cannot occur. We conclude that if $s_\beta v_j \not \in \mathrm{supp}(f_\lambda^{(k)})$ then $\beta \in N(v_k^{-1})$.  
Since $\gamma\in \Phi_h^-$, we are always assuming $\beta \in v_j(\Phi_h^-)$ and we now conclude that if $w s_\gamma \not \in 
\mathrm{supp}(f^{(k)}_\lambda)$ then $\beta \in N(v_k^{-1}) \cap v_j(\Phi_h^-)$.  
Second, for case (2-b), observe that $ws_\gamma = y s_\beta v_j \in S_\lambda v_j$ if and only if $s_\beta v_j \in S_\lambda v_j$ since $y \in S_\lambda$, and the latter is equivalent to $s_\beta \in S_\lambda$. From the decomposition~\eqref{eq: 31} and Lemma~\ref{lemma.vk-invs} it follows that $s_\beta \in S_\lambda$ if and only if $\beta \in N^-(y)$. Thus we obtain that $ws_\gamma \in S_\lambda v_j$ if and only if $\beta \in N^-(y) \cap v_j(\Phi_h^-)$. 
Third, for case (2-c), we can use similar reasoning to see that $ws_\gamma = y s_\beta v_j$ lies in $S_\lambda v_\ell$ for some $k \leq \ell < j$ if and only if 
\[
\beta \in \{t_1 - t_b \, \mid \, b \in \{\lambda_1+k+1, \cdots, \lambda_1 + j \}\} \cap v_j(\Phi_h^-).
\]

We can now argue each case separately, based on the above characterizations of the root $\beta$.

\smallskip
\noindent \textbf{Sub-case (2-a):} In this case $f^{(k)}_\lambda(ws_\gamma) = 0$, so in order to prove the GKM condition it suffices to prove that $w(\gamma)$ divides $f^{(k)}_\lambda(w)$, i.e.~that $w(\gamma) = y(\beta)$ for some $\beta\in \mathcal{S}_k$. 
Since $w = yv_j$ this is equivalent to $v_j(\gamma) = \beta \in \mathcal{S}_k$. Recall from~\eqref{eq: def Sk} that $\mathcal{S}_k = N(v_k^{-1}) \cap v_k(\Phi_h^-)$. As we saw above, in this case we have $\beta\in N(v_k^{-1}) \cap v_j(\Phi_h^-)$, so it remains to establish that $\beta\in v_k(\Phi_h^-)$.  Write $\beta=t_a-t_b$.  The assumption that $t_a-t_b\in v_j(\Phi_h^-)$ implies that $v_j^{-1}(a)\leq h(v_j^{-1}(b))$.  Since 
$\beta \in N(v_k^{-1})$, from Lemma~\ref{lemma.vk-invs} it follows that 
$a\in \{1, \ldots, \lambda_1\}$ and $b\in \{\lambda_1+1, \ldots, n\}$. Now the explicit formula in~\eqref{eq: vk inverse} for the one-line notation of $v_k^{-1}$ and $v_j^{-1}$ implies  $v_k^{-1}(a)\leq v_j^{-1}(a)$ and $v_k^{-1}(b)\geq v_j^{-1}(b)$.  Thus $v_{k}^{-1}(a) \leq v_j^{-1}(a) \leq h(v_j^{-1}(b)) \leq h(v_k^{-1}(b))$ implying $\beta=t_a-t_b\in v_k(\Phi_h^-)$ as desired, and case (2-a) is complete.

\smallskip
\noindent \textbf{Sub-case (2-b):} In this case, we have $\beta = v_j(\gamma)\in N^-(y)$, so $s_\beta \in S_\lambda$. This implies that $s_{y(\beta)} = s_{y v_j(\gamma)} = s_{w(\gamma)} \in S_\lambda$ also. 
Now from Lemma~\ref{lemma: fk def} we conclude that
\[
f^{(k)}_\lambda(ws_\gamma) = f^{(k)}_\lambda(s_{w(\gamma)}w) = s_{w(\gamma)}(f^{(k)}_\lambda (w)).
\]
It is a classical fact that $w(\gamma)$ divides $f - s_{w(\gamma)}f$, so we obtain our result.
 This completes case (2-b).

\smallskip
\noindent \textbf{Sub-case (2-c):} 
In this case we have $ws_\gamma \in S_\lambda v_\ell$ for $k \leq \ell < j$, so in particular $f^{(k)}_\lambda(ws_\gamma) \neq 0$. We aim to show that $f^{(k)}_\lambda(w s_\gamma) = f^{(k)}_\lambda(w)$, from which it follows that $f^{(k)}_\lambda(w) - f^{(k)}_\lambda(ws_\gamma) = 0$, which is clearly divisible by $w(\gamma)$.

First observe that the only way we can have $ws_\gamma \in S_\lambda v_\ell$ is if $\gamma = t_{j+1} - t_{\ell+1}$. 
This is because $w \in S_\lambda v_j$, which implies the entries $\{1, \ldots, \lambda_1\}$ are in the $(j+1)$-th and the last $\lambda_1 -1$ positions in the one-line notation of $w$. Any element in $S_\lambda v_\ell$ must have the $\{1, \ldots, \lambda_1\}$ entries in the $(\ell+1)$-th and last $\lambda_1-1$ positions of its one-line notation. In order for this to happen, we must have $s_\gamma$ exchange the positions $j+1$ and $\ell+1$. Next 
recall the decomposition $w=yv_j = yv_0u_j$ where $v_0$ and $u_j$ are as defined in~\eqref{eqn.v0formula} and~\eqref{eq: def uk}, respectively.  Since $s_\gamma$ is the reflection swapping $j+1$ and $\ell+1$, an explicit computation yields 
\[
ws_\gamma = \left\{\begin{array}{ll} y v_0 s_{\ell+2}\cdots s_{j} u_\ell  & \textup{ if $j>\ell+1$} \\ 
yv_0u_\ell = y v_\ell & \textup{ if $j=\ell+1$} \end{array}\right. 
\]
which implies that $w s_\gamma = y v_\ell$ if $j = \ell+1$. Hence for the case $j=\ell+1$ it is immediate that $f^{(k)}_\lambda(w s_\gamma) = f_\lambda^{(k)}(y v_\ell) = \prod_{\eta \in \mathcal{S}_k} y(\eta) = f^{(k)}_\lambda(w)$ by definition of $f_\lambda^{(k)}$. Thus $f_\lambda^{(k)}(w s_\gamma) - f_\lambda^{(k)}(w) = 0$ and we are done. Therefore, in what follows we may assume that $j > \ell+1$. In this case we claim that
\begin{eqnarray*}
ws_\gamma  = y v_0 s_{\ell+2}\cdots s_{j} u_\ell  = y' v_0 u_\ell 
\end{eqnarray*}
for some $y'\in S_\lambda$.  Indeed, we get
\[
y' = y(v_0 s_{\ell+2} \cdots s_j v_0^{-1}) = y s_{v_0(\ell+2)} \cdots s_{v_0(j)}
\]
and we know $v_0(i) = \lambda_1+i-1$ for all $i=2, \ldots, \lambda_2$ from~\eqref{eqn.v0formula}. 
Since $\ell+1 < j$ by assumption and $j \leq \lambda_2$ we know $\ell+2 \leq \lambda_2$ and also since $j > \ell+1$ where $\ell \geq 0$, we know $j \geq 2$. So 
it follows that
\begin{eqnarray}\label{eqn.y'}
y' = ys_{\lambda_1+\ell+1}s_{\lambda_1+\ell+2}\cdots s_{\lambda_1+j-1}\in S_\lambda.
\end{eqnarray}
Let $y_1:=s_{\lambda_1+\ell+1}s_{\lambda_1+\ell+2}\cdots s_{\lambda_1+j-1}$. 
To summarize, we have shown $ws_\gamma = yy_1v_\ell$, where $y y_1 \in S_\lambda$.

By definition of $f_\lambda^{(k)}$ we have 
\[
f_\lambda^{(k)}(ws_\gamma) = \prod_{\eta\in \mathcal{S}_k} yy_1(\eta) \quad \textup{ and }  \quad f_\lambda^{(k)}(w) = \prod_{\eta\in \mathcal{S}_k} y(\eta).
\]
If we establish the following equality
\begin{eqnarray}\label{eqn.y1-action}
y_1 \left(\prod_{\eta\in \mathcal{S}_k} \eta \right) = \prod_{\eta\in \mathcal{S}_k} \eta
\end{eqnarray}
then it would follow that $f_\lambda^{(k)}(w s_\gamma) = f_\lambda^{(k)}(w)$, hence
$f_\lambda^{(k)}(w s_\gamma) - f_\lambda^{(k)}(w) = 0$ and we are done. In the remainder of the argument we therefore focus on proving~\eqref{eqn.y1-action}. 

To prove~\eqref{eqn.y1-action}, first notice that $y_1\in \Stab(1,2, \ldots, \lambda_1+\ell)$.  
Motivated by this, using the explicit description of $N(v_k^{-1})$ in Lemma~\ref{lemma.vk-invs} we 
decompose the elements of $\mathcal{S}_k$ into two subsets $\mathcal{S}_k = \mathcal{S}_{k}^{(1)}\sqcup \mathcal{S}_k^{(2)}$, 
where we define 
\begin{eqnarray}\label{eqn.set1}
\mathcal{S}_k^{(1)} := \{t_a-t_b\in \mathcal{S}_k \mid a,b\in \{1,2,\ldots, \lambda_1+k\}\} 
\end{eqnarray}
and
\begin{eqnarray}\label{eqn.set2}
\mathcal{S}_k^{(2)} =  \{ t_a-t_b \in \mathcal{S}_k \mid (a,b)\in \{2, \ldots, \lambda_1\}\times \{\lambda_1+k+1,\ldots, n\} \}. 
\end{eqnarray}
Since $k\leq \ell$, it is clear that if $\eta\in \mathcal{S}_k^{(1)}$ then $y_1(\eta) = \eta$.  
Next
we analyze the set $\mathcal{S}_k^{(2)}$. 
Note first that in the case $\lambda_1=1$, then $\mathcal{S}_k = \mathcal{S}_k^{(1)}$ and hence we are done. Thus for the remainder of the argument we may assume $\lambda_1 > 1$. 
For any $(a,b)$ with $a \in \{2, \ldots, \lambda_1\}$ and $b \in \{\lambda_1+k+1, \ldots, n\}$ 
then 
\[
v_{k}^{-1}(a)\in \{\lambda_2+2, \ldots, n\} \textup{ and } v_k^{-1}(b) \in \{k+2, k+3,\ldots, \lambda_2+1\}.
\]
Since $\lambda_1 > 1$, we have by the hypothesis in the statement of the theorem that $h(k+2)=n$. Thus 
$h(v_k^{-1}(b)) = n$ for any $v_k^{-1}(b) \in \{k+2, \ldots, \lambda_2+1\}$ and 
$v_k^{-1}(a)\leq n = h(v_k^{-1}(b))$, 
 implying $t_a - t_b \in v_k(\Phi_h^-)$ for all $a \in \{2, \ldots, \lambda_1\}$ and $b \in \{\lambda_1+k+1, \ldots, n\}$. 
 The above discussion implies that 
\[
\mathcal{S}_k^{(2)} = \{ t_a-t_b  \mid (a,b)\in \{2, \ldots, \lambda_1\}\times \{\lambda_1+k+1,\ldots, n\} \}.
\]
Since $y_1 = s_{\lambda_1+\ell+1}s_{\lambda_1+\ell+2}\cdots s_{\lambda_1+j-1}$ permutes the elements of $\{\lambda_1+k+1, \ldots, n\}$ (because $\lambda_1+k+1\leq \lambda_1+\ell+1$ and $\lambda_1+j-1\leq \lambda_1+\lambda_2-1 = n-1$ ) and stabilizes the elements of $\{2, \ldots, \lambda_1\}$, it follows that $y_1(\mathcal{S}_k^{(2)}) = \mathcal{S}_k^{(2)}$.   Since we already saw $y_1$ stabilizes the elements in $\mathcal{S}_k^{(1)}$, we conclude~\eqref{eqn.y1-action} holds, as desired. This completes the (2-c) case and hence the proof. 
\end{proof}

By Lemma~\ref{lemma: fk def}, the class $f_\lambda^{(k)}$ constructed in Theorem~\ref{thm.class} is fixed by $S_\lambda$ under the dot action.  As in the case of the top-coset classes of the previous section, we consider the orbit of $f_\lambda^{(k)}$ under the dot action:  
\[
\{ v \cdot f_\lambda^{(k)} \mid v\in \Sn^\lambda  \}.
\]
We will prove in Section~\ref{subsec: LI for fk} that this set of classes is $H^*_T(\pt)$-linearly independent, in a special case and under further assumptions on the Hessenberg function $h$.

We now discuss potential connections between our Theorem~\ref{thm.class} and some other recent results by Cho, Hong, and Lee on equivariant cohomology classes for the regular semisimple Hessenberg variety.  
We first remark that, as noted in the introduction, the advantages of our construction are: (1) we have an explicit formula for the value of $f_\lambda^{(k)}$ at each $w\in \Sn$, namely that in~\eqref{eqn.cohom-class} and (2) we can give simple, concrete descriptions of the elements in the $\Sn$-orbit of $f_\lambda^{(k)}$ as well as their support sets.  On the other hand, the drawbacks of our construction are that, in its current form, the construction only applies in the case that the composition $\lambda$ has two parts, and we are not yet able to use such classes to construct a basis for the free module $H_T^*(\Hess(\mathsf{S}, h))$.  
In contrast, Cho, Hong, and Lee recently gave a geometric construction of an $H_T^*(\pt)$-module basis for $H_T^*(\Hess(\mathsf{S}, h))$ for general Hessenberg functions \cite{Cho-Hong-Lee2020}.  Their classes arise from a Bia\l{}ynicki-Birula decomposition of $\Hess(\mathsf{S}, h)$.  While the existence of such classes is not new, it is in general a difficult question to compute the values of these classes at different permutations $w \in \Sn$. The results of~\cite{Cho-Hong-Lee2020} are significant in that they make progress toward describing these classes explicitly.  For example, the authors describe the support set of each class combinatorially in terms of the Hessenberg function $h$.  On the other hand, these classes do not give a permutation basis of $H_T^*(\Hess(\mathsf{S}, h))$, and an explicit formula for their values at $w \in \Sn$ similar to that given in~\eqref{eqn.cohom-class} is only known in the case where $h=(2,3,\ldots, n,n)$.  In that special case, the authors express certain equivariant classes defined by Chow in the statement of his \textit{Erasing marks conjecture} \cite{Chow-conj} as linear combinations of their ``Bia\l{}ynicki-Birula classes'', and use their results to prove Chow's conjecture. 
 As we have already noted, Chow's classes are top-coset classes for appropriately chosen Young subgroups.  
 
 This recent progress, together with our Theorem~\ref{thm.class} above, naturally suggest the following open problem. We expect a solution to this problem to lead to further progress in the ``permutation basis program'' in more general cases of Hessenberg functions. 

\begin{problem}\label{prob.BB-classes} 
Let $h: [n] \to [n]$ be a Hessenberg function and $\lambda=(\lambda_1, \lambda_2)$ a composition of $n$ with exactly two parts satisfying the assumptions of Theorem~\ref{thm.class}.  Compute that expansion of $f_\lambda^{(k)}$ as an $H_T^*(\pt)$-linear combination of the $H_T^*(\pt)$-module basis of  ``Bia\l{}ynicki-Birula classes'' for $H_T^*(\Hess(\mathsf{S}, h))$ studied by Cho--Hong--Lee in \cite{Cho-Hong-Lee2020}. 
\end{problem}

Finally, in the last part of this section, we prove some properties of our classes $f_\lambda^{(k)}$ which will be useful in the analysis in the following sections. 
We begin with a proof of the analogue of Lemma~\ref{lem.left-coset-description}, describing the support set of each $v\cdot f_\lambda^{(k)}$ for $v\in \Sn^\lambda$. Given the composition $\lambda$ of $n$, recall that for each $v_j\in ({^\lambda\Sn})_0$ we obtain from Lemma~\ref{lemma.right-to-left} a bijection 
\[
\phi_{v_j}: S_\lambda\backslash \Sn \to \Sn/ S_\lambda^{(j)};\; \phi_{v_j}(S_\lambda v) = v^{-1}v_j S_\lambda^{(j)}
\]
where $S_\lambda^{(j)}:= v_j^{-1}S_\lambda v_j$. Recall also that $(\Sn^\lambda)^{-1} = {^\lambda\Sn}$ by~\eqref{eq: left and right coset reps inverse}. 

\begin{lemma}\label{lem.left-coset-support} 
Let $\lambda=(\lambda_1, \lambda_2)$ be a composition of $n$ with two parts and $0\leq k\leq \lambda_2$.  For each $v\in \Sn^\lambda$ we have
\[
\left( v\cdot f_\lambda^{(k)}\right)(w) = \left\{ \begin{array}{cl} \prod_{\eta \in \S_k} wv_j^{-1}(\eta) & \textup{if $w\in \phi_{v_j}(S_\lambda v^{-1})$ for some $k\leq j \leq \lambda_2$}\\ 0 & \textup{otherwise.} \end{array} \right.
\]
In particular, $v\cdot f_\lambda^{(k)} : \Sn \to \C[t_1, \ldots, t_n]$ has support equal to the union of left cosets, 
\[
\bigsqcup_{k\leq j \leq \lambda_2} \phi_{v_j}(S_\lambda v^{-1}) = \bigsqcup_{k\leq j \leq \lambda_2} vv_j S_\lambda^{(j)}.
\] 
\end{lemma}

\begin{proof} Let $v\in \Sn^\lambda$. We have $\left(v \cdot f_\lambda^{(k)}\right)(w)\neq 0$ if and only if $f_\lambda^{(k)}(v^{-1}w)\neq 0$.  The latter condition is equivalent by Lemma~\ref{lemma: fk def} to the condition that $v^{-1}w\in S_\lambda v_j$ for some $k\leq j \leq \lambda_2$.  We have 
\begin{eqnarray*}
v^{-1}w = y v_j \textup{ for $y\in S_\lambda$ } &\Leftrightarrow& w = vv_j v_j^{-1}yv_j \textup{ for $y\in S_\lambda$}\\
&\Leftrightarrow& w \in vv_j S_\lambda^{(j)} = \phi_{v_j}(S_\lambda v^{-1}). 
\end{eqnarray*}
This proves the assertion about the support of $v \cdot f_\lambda^{(k)}$.  Given such a $w$, write $v^{-1}w = yv_j$ for $y\in S_\lambda$.  Then $vy = wv_j^{-1}$ and we get
\[
\left(v\cdot f_\lambda^{(k)}\right)(w) := v(f_{\lambda}^{(k)} (v^{-1}w)) = v\left(  \prod_{\eta\in \S_k} y(\eta) \right) = \prod_{\eta\in \S_k} wv_j^{-1}(\eta).
\]
This proves the lemma.
\end{proof}

Our last lemma of this section states that the stabilizer of the element $f_\lambda^{(k)}$ is precisely $S_\lambda$. 

\begin{lemma}\label{lemma: stabilizer is Slambda} Let $\lambda= (\lambda_1, \lambda_2)$ be a composition of $n$ with two parts and $0\leq k \leq \lambda_2$.  If $\lambda_1=1$ then we also assume that $k\geq 1$.  Then the stabilizer in $\Sn$ of $f_\lambda^{(k)}$ is equal to $S_\lambda$. 
\end{lemma} 

\begin{proof} 
We have already seen in Lemma~\ref{lemma: fk def}(2) that $S_\lambda$ stabilizes $f_\lambda^{(k)}$. Hence it suffices to show that if $v \in \Sn$ satisfies $v \cdot f_\lambda^{(k)} = f_\lambda^{(k)}$, then $v \in S_\lambda$. Since we already know that $S_\lambda$ is contained in the stabilizer, it suffices to prove the statement for $v \in S_n^{\lambda}$ a shortest left coset representative. So suppose $v \in S_n^{\lambda}$ and suppose that $v \cdot f_\lambda^{(k)} = f_\lambda^{(k)}$. We wish to show that $v \in S_\lambda$, which means $v$ is the identity permutation (since the shortest left coset representative for the identity coset is the identity). Since $v \cdot f_\lambda^{(k)} = f_\lambda^{(k)}$, their supports sets must be equal, and by Lemma~\ref{lem.left-coset-support} it follows that 
\[
 \bigsqcup_{k\leq j \leq \lambda_2} vv_j S_\lambda^{(j)} =  \bigsqcup_{k\leq j \leq \lambda_2} v_j S_\lambda^{(j)}
 \]
 or equivalently 
 \[
  \bigsqcup_{k\leq j \leq \lambda_2} v S_\lambda v_j =  \bigsqcup_{k\leq j \leq \lambda_2} S_\lambda v_j.
  \]
In particular this means that, for every $j$ with $k \leq j \leq \lambda_2$, we must have that $vv_j$ is contained in some coset $S_\lambda v_\ell$ for $\ell$ with $k \leq \ell \leq \lambda_2$. In particular, there exists some $\ell$ with $k\leq \ell \leq \lambda_2$ such that $vv_k \in S_\lambda v_\ell$. We consider the cases $\lambda_1>1$ and $\lambda_1=1$ separately.

Suppose $\lambda_1>1$.  Recall that $v_k$ has one-line notation as given in~\eqref{eq: vk}, and in particular that the last $(\lambda_1-1)$-many entries of the one-line notation for $v_k$ are given by the sequence $2, 3, \ldots, \lambda_1$, and similarly for $v_\ell$.  Thus $vv_k \in S_\lambda v_\ell$ for $k \leq \ell \leq \lambda_2$ implies that 
$\{v(2), v(3), \cdots, v(\lambda_1)\}$ is a subset of $\{1,2,\cdots, \lambda_1\}$. 
 Now recall that $v$ is a shortest left coset representative. From Remark~\ref{lem.coset-decomp} it follows that we may assume its first $\lambda_1$ entries are increasing, i.e., $v(1) < v(2) < \cdots < v(\lambda_1)$. If $1 \in \{v(2),\cdots, v(\lambda_1)\}$, then since the entries must be increasing we conclude $v(2)=1$, but then we come to a contradiction since there is no value of $v(1)$ which can be less than $v(2)=1$. Thus $1 \not \in \{v(2),\ldots, v(\lambda_1)\}$, but then $\{v(2), \ldots, v(\lambda_1)\} = \{2,\ldots, \lambda_1\}$ and we conclude that $v(2)=2, v(3)=3, \cdots, v(\lambda_1)=\lambda_1$. Now the condition that $v(2)=2$ and $v(1)<v(2)$ forces $v(1)=1$ also. Thus $v$ must be the identity (since it is a shortest left coset representative, and acts as the identity on $\{1,2,\cdots,\lambda_1\}$, so it must also act as the identity on $\{\lambda_1+1,\cdots,n\}$).

Now we suppose $\lambda_1=1$ and $k\geq 1$.  In this case, $v_\ell=u_\ell$ is the unique permutation with $1$ in position $\ell+1$ and all other entries in increasing order. In particular, since $y(1)=1$ for all $y\in S_\lambda$ we get that $vv_k \in S_\lambda v_\ell$ for $k \leq \ell \leq \lambda_2$ implies that $1=vv_k(\ell+1)\in \{v(1), v(k+2), \ldots, v(n)\}$. Now suppose $1$ lies in $\{v(k+2),\ldots,v(n)\}$. Since $v$ is a shortest left coset representative, we may assume $v(2) < v(3) < \cdots < v(n)$. In particular, if $v(1)\neq 1$ then we must have $v(2)=1$. This contradicts the assertion that $1\in \{v(k+2),\ldots,v(n)\}$, since $k \geq 1$. Hence the only possibility is that $v(1)=1$, which in turn implies that $v$ must be the identity by the same reasoning as above. This concludes the proof.
\end{proof}


\section{Linear independence for an $\Sn$-orbit: special cases}\label{subsec: LI for fk}

In the previous two sections, we gave a purely combinatorial algorithm that produces, in certain situations, classes $f^{(k)}_{\lambda} \in \bigoplus_{w \in \Sn} H^*_T(\pt)$ which satisfy the GKM conditions for a Hessenberg function $h$, and hence can be viewed as equivariant cohomology classes in $H^*_T(\Hess(\mathsf{S},h))$. Moreover, Lemma~\ref{lemma: stabilizer is Slambda} proves that the stabilizer of the class $f^{(k)}_{\lambda}$ under the dot action is $S_\lambda$. Thus we can view the results of Section~\ref{sec.top.coset} and Section~\ref{sec.class.defn} as a partial answer to the first problem posed at the end of Section~\ref{subsec: perm basis program}. 

The purpose of this section is to take the theory developed in Sections~\ref{sec.top.coset} and~\ref{sec.class.defn} one step further, by addressing the main question posed in Problem 2 at the end of Section~\ref{subsec: perm basis program}, namely: under what conditions is the $\Sn$-orbit of $f^{(k)}_{\lambda}$ linearly independent over $H^*_T(\pt)$? Note that, in the case of the ``top coset'' classes, the $\Sn$-orbit is indeed $H^*_T(\pt)$-linearly independent, 
as we have already recorded in Proposition~\ref{prop: top coset rep}. Therefore, in this section we focus on proving the linear independence statement -- in some special cases -- for the classes we constructed in Section~\ref{sec.class.defn} which have supports that are a union of \emph{more than one} right coset.

We begin by stating the main result of this section.  We need some notation to state one of the (technical) hypotheses. Let $v_{\lambda_2-1}^{-1}$ be the permutation defined as in~\eqref{eq: vk inverse}; for the reader's convenience, we record its one-line notation here as well:
\begin{eqnarray}\label{eqn.vlambda}
v_{\lambda_2-1}^{-1} = [\lambda_2, \lambda_2+2, \ldots, n, 1, 2, \ldots, \lambda_2-1, \lambda_2+1].
\end{eqnarray}
We note in particular that 
\begin{equation}\label{eq: prop of v lambda2-1}
v_{\lambda_2-1}^{-1}(b) = b-\lambda_1 \, \, \textup{ if } \, \, \lambda_1+1 \leq b \leq n-1, \, \, \textup{ and } \, \, 
v_{\lambda_2-1}^{-1}(a) = a+\lambda_2 \, \, \textup{ if } \, \, 2 \leq a \leq \lambda_1. 
\end{equation}
The above remarks will be useful in the arguments below.
We also define 
\begin{equation}\label{eq: def j0}
j_0 := \min \{ b \in \{ \lambda_1+1, \cdots, n-1\} \, \mid \, \lambda_2 \leq h(v_{\lambda_2-1}^{-1}(b)) \}. 
\end{equation} 
The index $j_0$ is used in our proof to describe the set $\S_{\lambda_2-1} = N(v_{\lambda_2-1}^{-1}) \cap v_{\lambda_2-1}(\Phi_h^-)$.
We can now state our theorem.

\begin{theorem}\label{theorem: orbit independence}
Let $n$ be a positive integer and $h: [n] \to [n]$ a Hessenberg function. Let $\lambda=(\lambda_1, \lambda_2) \vdash n$ be a composition of $n$ with two parts. Assume $h(1) < \lambda_2$. In addition, if $\lambda_1>1$, we also assume $h(\lambda_2+1)=n$ and $h(v_{\lambda_2-1}^{-1}(j_0)) \leq \lambda_2+1$. Then
\begin{enumerate} 
\item the $\Sn$-orbit of $f^{(\lambda_2-1)}_\lambda$ is $H^*_T(\pt)$-linearly independent, and
\item the stabilizer of each element in the $\Sn$-orbit of $f_\lambda^{(\lambda_2-1)}$ is a conjugate of the reflection subgroup $S_\lambda$. 
\end{enumerate} 
In particular, the $H_T^*(\pt)$-submodule of $H_T^*(\Hess(\mathsf{S},h))$ spanned by the $\Sn$-orbit of $f_\lambda^{(\lambda_2-1)}$ is an $\Sn$-subrepresentation with the same character as $\mathrm{Ind}_{S_\lambda}^{\Sn}(\mathbf{1}) \simeq M^{\Par(\lambda)}$ where $\Par(\lambda)$ is the partition of $n$ obtained from $\lambda$ by rearranging the parts to be in decreasing order.
\end{theorem}

Note that since we are taking $k=\lambda_2-1$ in the above theorem (with respect to the construction of the $f^{(k)}_{\lambda}$ in the previous section), we have $k+2=\lambda_2+1$, so the assumption $h(\lambda_2+1)=n$ in Theorem~\ref{theorem: orbit independence} is equivalent to the necessary hypothesis $h(k+2=\lambda_2+1)=n$ in the statement of Theorem~\ref{thm.class}. Hence, under the hypotheses of Theorem~\ref{theorem: orbit independence}, we do know from Theorem~\ref{thm.class} that the classes $f^{(\lambda_2-1)}_{\lambda}$ are well-defined in $H^*_T(\Hess(\mathsf{S},h))$.

We also note that the claim regarding the stabilizers of the elements in the $\Sn$-orbit is a straightforward consequence of the construction of the $f^{(k)}_\lambda$ and Lemma~\ref{lemma: stabilizer is Slambda} (see also Proposition~\ref{prop: top coset rep}), so the main task at hand is to prove the $H^*_T(\pt)$-linear independence, and this is what occupies the bulk of this section. More specifically, we begin with the following.

We introduce some notation. Since $\lambda = (\lambda_1, \lambda_2)$ is a two-part composition, 
shortest left coset representatives in $\Sn^\lambda$ are parameterized by subsets of $n$ of cardinality $\lambda_1$. Indeed, given a subset $J = \{j_1 < j_2 < \cdots < j_{\lambda_1}\} \subseteq [n]$, the corresponding shortest left coset representative is the permutation defined as 
\begin{eqnarray} \label{eqn.one-line-vJ}
v_J := [j_1, j_2, \ldots, j_{\lambda_1},  j_1', \ldots, j_{\lambda_2}'] \in \Sn^\lambda
\end{eqnarray}
where $[n]\setminus J = \{j_1'< \cdots< j_{\lambda_2}'\}$ and it is straightforward to see that all shortest left coset representatives arise in this way. Moreover, 
given a permutation $w$ we obtain the one line notation for the shortest left coset representative of $w$ in $\Sn^\lambda$ by rearranging the values in positions $1,2,\ldots, \lambda_1$ and those in $\lambda_1+1,\ldots, n$ so that they are in increasing order.
With this notation in place we can write 
\[
\Sn^\lambda = \{v_J \mid  J \subseteq [n], \lvert J \rvert = \lambda_1\}.
\]
Note that since $(^\lambda\Sn)_0 \subseteq {^\lambda\Sn} = (\Sn^\lambda)^{-1}$ for all $k$ with $0\leq k \leq \lambda_2$,  we have
\[
v_k^{-1} = v_{\{ k+1, \lambda_2+2, \ldots, n \}}
\]
where $v_k$ is the permutation~\eqref{eq: vk} considered in Section~\ref{sec.class.defn}.  
Lemma~\ref{lemma: fk def}(2) shows that $y \cdot f^{(k)}_{\lambda} = f^{(k)}_\lambda$ for any $y \in S_{\lambda}$. This implies that the $\Sn$-orbit of $f^{(k)}_\lambda$ under the dot action is 
\begin{equation}\label{eq: orbit of fk}
\{ v_J \cdot f^{(k)}_\lambda  \mid  v_J \in {\Sn^\lambda} \}.
\end{equation}

Our linear independence argument requires the following statement.

\begin{proposition}\label{prop.support} Suppose $I,J, K\subseteq [n]$ are subsets with cardinality $\lambda_1$ and let $k=\lambda_2-1$.  Then 
\begin{enumerate}
\item $\mathrm{supp}(v_J\cdot f_\lambda^{(k)}) \cap \mathrm{supp}(v_I \cdot f_\lambda^{(k)})\neq \emptyset$ if and only if $I=J$ or $|J\cap I| = \lambda_1-1$, and
\item if $I,J,K$ are pairwise distinct subsets of $[n]$, then $\mathrm{supp}(v_J\cdot f_\lambda^{(k)}) \cap \mathrm{supp}(v_I \cdot f_\lambda^{(k)}) \cap \mathrm{supp}(v_K\cdot f_\lambda^{(k)}) = \emptyset$.
\end{enumerate}
\end{proposition}

\begin{proof} 

We begin by proving statement (1). First, it is clear that if $I=J$ then 
$$\mathrm{supp}(v_J\cdot f_\lambda^{(k)}) \cap \mathrm{supp}(v_I \cdot f_\lambda^{(k)})\neq \emptyset.
$$
 Thus to prove the statement it suffices to show that, in the case that $I \neq J$, the condition $\mathrm{supp}(v_J\cdot f_\lambda^{(k)}) \cap \mathrm{supp}(v_I \cdot f_\lambda^{(k)})\neq \emptyset$ is equivalent to $\lvert J \cap I \rvert = \lambda_1 - 1$. So now suppose $I \neq J$. 
Recall that $S_\lambda^{(j)} := v_j^{-1} S_\lambda v_j$ for any $0\leq j \leq \lambda_2$, so 
$v_j S_{\lambda}^{(j)} = S_\lambda v_j$.  By Lemma~\ref{lem.left-coset-support} and using the fact that $k=\lambda_2-1$ and $k+1 = \lambda_2$ (so $f^{(k)}_{\lambda}$ has support consisting of exactly \emph{two} cosets), we have
\begin{eqnarray}\label{eqn.support1}
\mathrm{supp}(v_J \cdot f_\lambda^{(k)}) \cap \mathrm{supp}(v_I \cdot f_\lambda^{(k)}) &=& \left(v_J v_k S_\lambda^{(k)} \sqcup v_Jv_{k+1}S_\lambda^{(k+1)}\right) \cap \left(v_Iv_k S_\lambda^{(k)} \sqcup v_Iv_{k+1}S_\lambda^{(k+1)}\right).
\end{eqnarray} 
Since $I \neq J$,
$v_I$ and $v_J$ are distinct shortest-left-coset representatives of $S_\lambda$, from which it follows 
 that $v_J v_k S_\lambda^{(k)} \cap v_I v_k S_\lambda^{(k)}  = \emptyset$ and similarly $v_J v_{k+1} S_\lambda^{(k+1)} \cap v_I v_{k+1} S_\lambda^{(k+1)} = \emptyset$ (see Lemma~\ref{lemma.right-to-left}). Hence we can continue the computation started in~\eqref{eqn.support1} to obtain 
\begin{eqnarray}\label{eqn.support1b}
\mathrm{supp}(v_J \cdot f_\lambda^{(k)}) \cap \mathrm{supp}(v_I \cdot f_\lambda^{(k)}) &=& \left(v_J v_{k+1}S_\lambda^{(k+1)} \cap v_{I} v_k S_\lambda^{(k)}\right) \sqcup \left(v_Jv_k S_\lambda^{(k)} \cap v_I v_{k+1}S_\lambda^{(k+1)}\right) \\ \nonumber
&=& \left(v_J S_\lambda v_{k+1}\cap v_I S_\lambda v_{k}\right) \sqcup \left(v_J S_\lambda v_{k} \cap v_I S_\lambda v_{k+1}\right).
\end{eqnarray}
This proves that, in the case $I\neq J$, the intersection of the two support sets is nonempty if and only if 
\[
v_JS_\lambda v_{k+1}\cap v_I S_\lambda v_{k} \neq \emptyset \;\; \textup{ or } \;\;  v_JS_\lambda v_{k} \cap v_I S_\lambda v_{k+1}\neq \emptyset.
\]
To complete the proof of statement (1), it now suffices to argue that each of these conditions is equivalent to the condition that $|J\cap I| = \lambda_1-1$.  First, we have
\begin{eqnarray*}
v_J S_\lambda v_{k+1}\cap v_I S_\lambda v_{k} \neq \emptyset &\Leftrightarrow& v_J y_1v_{k+1} = v_Iyv_k \;\;\textup{ for some } \;\; y,y_1\in S_\lambda\\
&\Leftrightarrow& v_Jy_1 = v_Is_{y(\theta)}y \;\; \textup{ for some } \;\; y,y_1\in S_\lambda\; \textup{ and }\; \theta = t_1-t_n
\end{eqnarray*}
where the second equivalence follows from the fact that $v_kv_{k+1}^{-1} = s_\theta$ since $k=\lambda_2-1$, as can be readily checked by computation. We conclude that $v_J S_\lambda v_{k+1}\cap v_I S_\lambda v_{k} \neq \emptyset$ if and only if there exists $y\in S_\lambda$ such that the shortest left coset representative of $v_I s_{y(\theta)}$ in $\Sn^\lambda$ is $v_J$.

The one line notation for $v_I s_{y(\theta)}$ is obtained from the one line notation of $v_I$ by exchanging the values in positions $y(1)$ and $y(n)$.  Since $y\in S_\lambda$ we know $y(1)\in \{ 1, \ldots, \lambda_1 \}$ and $y(n)\in \{\lambda_1+1,\ldots, n\}$.  In particular, the description of the one line notation for $v_J$ and $v_I$ given in~\eqref{eqn.one-line-vJ} implies that the desired condition holds if and only if we can obtain $J$ from $I$ by changing a single element, or more precisely, if and only if $|J\cap I| = \lambda_1-1$.  This proves the desired result in this case.

Next, consider the condition that $v_JS_\lambda v_{k} \cap v_I S_\lambda v_{k+1}\neq \emptyset$. By the same logic as above, this intersection is nonempty if and only if there exists $y\in S_\lambda$ such that the shortest left coset representative of $v_I s_{y(\theta)}$ is $v_J$ for some $y\in S_\lambda$. By the same reasoning as in the paragraph above we obtain $|J\cap I| = \lambda_1-1$. This proves statement (1).

We now prove statement (2). Suppose $I, J, K$ are pairwise distinct. Using the same reasoning as above the intersection of the three support sets is
\begin{equation}\label{eq: three intersection}
\left[\left(v_J v_{k+1}S_\lambda^{(k+1)} \cap v_{I} v_k S_\lambda^{(k)}\right) \sqcup \left(v_Jv_k S_\lambda^{(k)} \cap v_I v_{k+1}S_\lambda^{(k+1)}\right)\right] \cap \left(v_K v_kS_{\lambda}^{(k)} \sqcup v_K v_{k+1}S_\lambda^{(k+1)}\right).
\end{equation}
As before, since $J \neq K$ and $I \neq K$ we know that $v_Jv_{k+1}S_\lambda^{(k+1)}\cap v_K v_{k+1}S_\lambda^{(k+1)} = \emptyset$ and $v_Iv_{k}S_\lambda^{(k)}\cap v_K v_{k}S_\lambda^{(k)} = \emptyset$. In particular we obtain
\[
\left(v_J v_{k+1}S_\lambda^{(k+1)} \cap v_{I} v_k S_\lambda^{(k)}\right) \cap \left(v_K v_kS_{\lambda}^{(k)} \sqcup v_K v_{k+1}S_\lambda^{(k+1)}\right)=\emptyset.
\]
Similarly we obtain 
\[
\left(v_Jv_k S_\lambda^{(k)} \cap v_I v_{k+1}S_\lambda^{(k+1)}\right) \cap \left(v_K v_kS_{\lambda}^{(k)} \sqcup v_K v_{k+1}S_\lambda^{(k+1)}\right)=\emptyset.
\]
Hence the set in~\eqref{eq: three intersection} is empty, i.e. 
\[
\left[\left(v_J v_{k+1}S_\lambda^{(k+1)} \cap v_{I} v_k S_\lambda^{(k)}\right) \sqcup \left(v_Jv_k S_\lambda^{(k)} \cap v_I v_{k+1}S_\lambda^{(k+1)}\right)\right] \cap \left(v_K v_kS_{\lambda}^{(k)} \sqcup v_K v_{k+1}S_\lambda^{(k+1)}\right) = \emptyset
\]
as desired. This proves statement (2).
\end{proof}

We are now ready to prove Theorem~\ref{theorem: orbit independence}.

\begin{proof}[Proof of Theorem~\ref{theorem: orbit independence}]

First, Theorem~\ref{thm.class} implies that the class $f^{(\lambda_2-1)}_{\lambda}$ is indeed a well-defined GKM class under the hypotheses of Theorem~\ref{theorem: orbit independence}. Now 
we want to show that the set~\eqref{eq: orbit of fk} is $H^*_T(\pt)$-linearly independent.
Suppose there is a $H^*_T(\pt)$-linear combination of $\{v_J \cdot f_\lambda^{(\lambda_2-1)} \mid v_J \in \Sn^\lambda\}$ that gives the zero class, i.e., 
\begin{equation}\label{eq: zero linear comb}
\sum_J c_J \, v_J\cdot f^{(\lambda_2-1)}_{\lambda} = 0 \in H^*_T(\Hess(\mathsf{S},h)) \subseteq \bigoplus_{w \in \Sn} H^*_T(\pt)
\end{equation}
for some $c_J \in H^*_T(\pt)$. 
We must show that $c_J = 0 \in H^*_T(\pt)$ for all $J \subset [n]$ with $\lvert J \rvert = \lambda_1$. Since~\eqref{eq: zero linear comb} holds as an equality of GKM classes, then 
in particular the LHS must evaluate to $0$ at any permutation $w \in \Sn$.

By Proposition~\ref{prop.support}, only two elements in the set~\eqref{eq: orbit of fk} can be nonzero when evaluated at any given $w \in \Sn$.  Consider, in particular, the evaluation at $v_{\lambda_2}$ of the LHS of~\eqref{eq: zero linear comb}. Let $K=\{2,3,\ldots, \lambda_1, n\}$. 
We now show that $v_{\lambda_2} \in \mathrm{supp}(f_{\lambda}^{(k)}) \cap \mathrm{supp}(v_K \cdot f_\lambda^{(k)})$. To see this, 
 we apply equation~\eqref{eqn.support1b} to $J = \{1,2,\ldots, \lambda_1\}$ and $I = K = \{2,3,\ldots, \lambda_1, n\}$ (so $v_J = e$ is the identity permutation) to obtain
\begin{equation}\label{eq:support2}
\mathrm{supp}(f_\lambda^{(k)}) \cap \mathrm{supp}(v_K\cdot f_\lambda^{(k)}) = (S_\lambda v_{k+1} \cap v_K S_\lambda v_k) \sqcup (S_\lambda v_k \cap v_K S_\lambda v_{k+1}).
\end{equation}
Recall that 
$v_{\lambda_2-1} v_{\lambda_2}^{-1} = s_\theta$ as in the proof of Proposition~\ref{prop.support} 
where $\theta = t_1-t_n$ so $v_{\lambda_2} v_{\lambda_2-1}^{-1} = s_\theta$ also (since $s_\theta^{-1} = s_\theta$). Also, from the definition of $v_K$ in~\eqref{eqn.one-line-vJ} it is not hard to see that $v_K s_\theta \in S_\lambda$.  We then have 
\begin{equation}\label{eq: vKinv stheta}
v_K^{-1} s_\theta \in S_\lambda \Rightarrow v_K^{-1}v_{\lambda_2}v_{\lambda_2-1}^{-1} \in S_\lambda \Rightarrow v_{\lambda_2} = v_K y v_{\lambda_2-1} \;\; \textup{ for some }\;\; y\in S_\lambda
\end{equation}
so $v_{\lambda_2} \in v_K S_\lambda v_k = v_K S_\lambda v_{\lambda_2-1}$. Since $v_{\lambda_2} \in S_\lambda v_{k+1} = S_\lambda v_{\lambda_2}$ also, we see that $v_{\lambda_2} \in S_\lambda v_{k+1} \cap v_K S_\lambda v_k$ so by~\eqref{eq:support2} we conclude $v_{\lambda_2} \in \mathrm{supp}(f_\lambda^{(k)}) \cap \mathrm{supp}(v_K\cdot f_\lambda^{(k)})$.

The discussion above implies that when we evaluate the LHS of~\eqref{eq: zero linear comb} at $w = v_{\lambda_2}$ we obtain 
\begin{equation}\label{eq: eval at v lambda 2}
\begin{split} 
c_{J} f^{(\lambda_2-1)}_{\lambda}(v_{\lambda_2}) + c_K \,\left( v_K \cdot f^{(\lambda_2-1)}_{\lambda}\right)(v_{\lambda_2}) 
& = c_{J} f^{(\lambda_2-1)}_{\lambda}(v_{\lambda_2}) + 
c_K \, v_K\left(f^{(\lambda_2-1)}_{\lambda}(v_K^{-1} v_{\lambda_2} )\right) \\
& = c_{J} f^{(\lambda_2-1)}_{\lambda}(v_{\lambda_2}) + 
c_K \, v_K\left(f^{(\lambda_2-1)}_{\lambda}(v_K^{-1} v_{\lambda_2}v_{\lambda_2-1}^{-1}v_{\lambda_2-1} )\right) \\
& = c_{J} \prod_{\beta \in \mathcal{S}_{\lambda_2-1}} \beta + 
c_K \, v_K \left( \prod_{\beta \in \mathcal{S}_{\lambda_2-1}} v_K^{-1} v_{\lambda_2} v_{\lambda_2-1}^{-1} (\beta) \right)   \\
& = c_{J} \prod_{\beta \in \mathcal{S}_{\lambda_2-1}} \beta + 
c_K  \prod_{\beta \in \mathcal{S}_{\lambda_2-1}} s_\theta(\beta) \\
\end{split} 
\end{equation} 
where in the third equality we have used that $v_K^{-1} v_{\lambda_2} v_{\lambda_2-1}^{-1}\in S_{\lambda}$ as we saw in~\eqref{eq: vKinv stheta} and in the last equality we have used that $v_{\lambda_2} v_{\lambda_2-1}^{-1} = s_{\theta}$. 
Since we have the equality~\eqref{eq: zero linear comb} we conclude that 
\[
c_{J} \prod_{\beta \in \mathcal{S}_{\lambda_2-1}} \beta + 
c_K  \prod_{\beta \in \mathcal{S}_{\lambda_2-1}} s_\theta(\beta)  = 0.
\]

The above analysis gives us one linear equation relating two of the coefficients appearing in the LHS of~\eqref{eq: zero linear comb}.
We need at least one more equation to be able to conclude that $c_{J}$ and $c_K$ are both equal to $0$. To do this, we need to evaluate~\eqref{eq: eval at v lambda 2} at another permutation in $\mathrm{supp}(f_\lambda^{(k)}) \cap \mathrm{supp}(v_K\cdot f_\lambda^{(k)})$. To find such a permutation, 
it will be useful to set some notation. We define 
\[
\mathcal{A}_1 :=  \{t_a - t_b \, \mid \, (a,b) \in \{2,3,\cdots, \lambda_1\} \times \{\lambda_1+1, \cdots, n\}, \, v_{\lambda_2-1}^{-1}(a) \leq h(v_{\lambda_2-1}^{-1}(b)) \}
\]
and 
\[
\mathcal{A}_2 := \{t_1-t_b \, \mid \, b \in \{\lambda_1+1, \cdots, n-1\}, \lambda_2 \leq h(v_{\lambda_2-1}^{-1}(b)) \}
\]
Note that $\mathcal{A}_1 = \emptyset$ if $\lambda_1 = 1$.
It follows from the definition of $\mathcal{S}_{\lambda_2-1}$, Lemma~\ref{lemma.vk-invs}, and properties of $v_{\lambda_2-1}$ that 
\begin{equation}\label{eq: S lambda2-1}
\mathcal{S}_{\lambda_2-1} = \mathcal{A}_1 \sqcup \mathcal{A}_2. 
\end{equation} 
Recall that
\[
j_0 := \min \{ b \in \{ \lambda_1+1, \cdots, n-1\} \, \mid \, \lambda_2 \leq h(v_{\lambda_2-1}^{-1}(b)) \}.
\]
From the one-line notation of $v_{\lambda_2-1}^{-1}$ in~\eqref{eqn.vlambda} it follows that $v_{\lambda_2-1}^{-1}(b) = b -\lambda_1$, and together with the fact that Hessenberg functions are non-decreasing, this implies
\begin{equation}\label{eq: A2}
\mathcal{A}_2= \{ t_1-t_{j_0}, t_1-t_{j_0+1},\ldots, t_1-t_{n-1} \}.
\end{equation} 
Since $h(1) < \lambda_2$ by assumption, we have $h(v_{\lambda_2-1}^{-1}(\lambda_1+1)) = h(1)<\lambda_2$ so we conclude that $\lambda_1+1 < j_0$. Consider the simple reflection $s_{j_0-1}$ exchanging $j_0-1$ and $j_0$.  Since
$\lambda_1+1 \leq j_0-1 \leq n-2$ we have $s_{j_0-1} \in S_\lambda$ and $s_{j_0-1}(1)=1$ and $s_{j_0-1}(n)=n$. 

Recall that one of the hypotheses of Theorem~\ref{theorem: orbit independence} is that $h(v_{\lambda_2-1}^{-1}(j_0)) \leq \lambda_2+1$. Using this, we conclude that, in the case when $\lambda_1 >1$, we have 
\begin{equation}\label{eq: ta - tb in A1}
\begin{split} 
t_a - t_b \in \mathcal{A}_1 & \Rightarrow v_{\lambda_2-1}^{-1}(a) \in \{\lambda_2+2, \cdots,n\} \,\,\textup{ by~\eqref{eqn.vlambda} } \\
 & \Rightarrow h(v_{\lambda_2-1}^{-1}(j_0)) \leq \lambda_2+1 <v_{\lambda_2-1}^{-1}(a)  \\
 & \Rightarrow h(v_{\lambda_2-1}^{-1}(j_0)) < h(v_{\lambda_2-1}^{-1}(b)) \, \, \textup{ by definition of $\mathcal{A}_1$} \\ 
 & \Rightarrow v_{\lambda_2-1}^{-1}(b) > v_{\lambda_2-1}^{-1}(j_0) \\
  & \Rightarrow b > j_0 \textup{ since } v_{\lambda_2-1}^{-1} \textup{ is increasing on } \{\lambda_1+1, \cdots, n\}, \textup{ and } b, j_0 \in \{\lambda_1+1, \cdots, n\} \\
 & \Rightarrow s_{j_0-1}(t_a - t_b) = t_a - t_b \, \,  \textup{since $j_0 < b$ and } a \leq \lambda_1 < j_0-1. 
\end{split} 
\end{equation}
Since~\eqref{eq: ta - tb in A1} holds for any $\beta = t_a - t_b \in \mathcal{A}_1$ we conclude $s_{j_0-1}(\beta) = \beta$ for all $\beta \in \mathcal{A}_1$. 
(In the case $\lambda_1 = 1$, the set $\mathcal{A}_1$ is empty so this statement is vacuously true.) 
From~\eqref{eq: A2} and~\eqref{eq: S lambda2-1} we then 
see immediately that 
\begin{equation} \label{eqn.formula}
\begin{split} 
f^{(\lambda_2-1)}_\lambda(v_{\lambda_2}) & = (t_1 - t_{j_0})(t_1 - t_{j_0+1}) \cdots (t_1 - t_{n-1})  \prod_{\beta \in \mathcal{A}_1} \beta \\
f^{(\lambda_2-1)}_\lambda(s_{j_0-1} v_{\lambda_2}) & = s_{j_0-1}(f^{(\lambda_2-1)}(v_{\lambda_2})) = 
(t_1 - t_{j_0-1}) (t_1 - t_{j_0+1}) \cdots (t_1 - t_{n-1})  \prod_{\beta \in \mathcal{A}_1} \beta \\
\end{split} 
\end{equation}
since $s_{j_0-1} \in S_\lambda$ and $s_{j_0-1}$ acts non-trivially only on the first factor of $f^{(\lambda_2-1)}(v_{\lambda_2})$. (Here we interpret the product over $\mathcal{A}_1$ to be equal to $1$ if $\lambda_1 = 1$ and $\mathcal{A}_1 = \emptyset$.) Note that the above equation implies $s_{j_0-1}v_{\lambda_2}\in \mathrm{supp}(f_\lambda^{(\lambda_2-1)})$.

In~\eqref{eq: eval at v lambda 2} we computed $v_K \cdot f^{(\lambda_2-1)} (v_{\lambda_2})$ and obtained 
\begin{eqnarray}\label{eqn.vKformula}
v_K \cdot f^{(\lambda_2-1)} (v_{\lambda_2}) =  \prod_{\beta \in \mathcal{S}_{\lambda_2-1}} s_{\theta}(\beta) 
 = (t_n - t_{j_0})(t_n- t_{j_0+1}) \cdots (t_n - t_{n-1}) \prod_{\beta \in \mathcal{A}_1} s_{\theta}(\beta).
\end{eqnarray}
 We can compute 
\begin{equation} \label{eqn.vKformula2}
\begin{split} 
v_K \cdot f^{(\lambda_2-1)}_\lambda (s_{j_0-1} v_{\lambda_2}) & = v_K\left(f^{(\lambda_2-1)}(v_K^{-1} s_{j_0-1} v_{\lambda_2})\right)  \\
 & = v_K \left( f^{(\lambda_2-1)}_\lambda(v_K^{-1} s_{j_0-1} s_{\theta} v_{\lambda_2-1})\right) 
 \, \, \textup{ since $s_\theta = v_{\lambda_2} v_{\lambda_2-1}^{-1}$}  \\
 & = v_K \left( f^{(\lambda_2-1)}_\lambda(v_K^{-1} s_{\theta} s_{j_0-1} v_{\lambda_2-1}) \right)  \textup{ because $s_{\theta}$ and $s_{j_0-1}$ commute } \\ 
 & = s_{\theta} s_{j_0-1} \left( f^{(\lambda_2-1)}_\lambda(v_{\lambda_2-1})\right) \textup{ because $v_K^{-1} s_{\theta} \in S_\lambda$ and $s_{j_0-1}\in S_\lambda$ } \\
 & = s_{\theta}  \left( (t_1 - t_{j_0-1}) (t_1 - t_{j_0+1}) \cdots (t_1 - t_{n-1}) \prod_{\beta \in \mathcal{A}_1} \beta \right) \\
 & \phantom{move over} \textup{ by definition of $f_\lambda^{(\lambda_2-1)}$ and~\eqref{eqn.formula} } \\ 
 & = (t_n - t_{j_0-1})(t_n- t_{j_0+1}) \cdots (t_n - t_{n-1}) \prod_{\beta \in \mathcal{A}_1} s_{\theta}(\beta). \\ 
\end{split} 
\end{equation}
In particular, our computations imply $s_{j_0-1}v_{\lambda_2} \in \mathrm{supp}(f_\lambda^{\lambda_2-1}) \cap \mathrm{supp}(v_K\cdot f_\lambda^{(\lambda_2)})$.

Evaluating~\eqref{eq: zero linear comb} at $v_{\lambda_2}$ and $s_{j_0-1}v_{\lambda_2}$ we obtain equations 
\begin{equation}\label{eq: star 1} 
c_J\, f_\lambda^{(\lambda_2-1)}(v_{\lambda_2}) + c_K \,v_K \cdot f_\lambda^{(\lambda_2-1)}(v_{\lambda_2}) = 0 
\end{equation}
and 
\begin{equation}\label{eq: star 2} 
c_J\, f_\lambda^{(\lambda_2-1)}(s_{j_0-1} v_{\lambda_2}) + c_K \,v_K \cdot f_\lambda^{(\lambda_2-1)}(s_{j_0-1} v_{\lambda_2}) = 0. 
\end{equation}
Subtracting~\eqref{eq: star 2} from~\eqref{eq: star 1} and using the formulas given in~\eqref{eqn.formula}, \eqref{eqn.vKformula}, and~\eqref{eqn.vKformula2} we obtain 
\[
c_J (t_{j_0-1} - t_{j_0}) (t_1 - t_{j_0+1}) \cdots (t_1 - t_{n-1})  \prod_{\beta \in \mathcal{A}_1} \beta   
+ c_K (t_{j_0-1} - t_{j_0}) (t_n - t_{j_0+1}) \cdots (t_n - t_{n-1}) \prod_{\beta \in \mathcal{A}_1} s_{\theta}(\beta)  = 0.
\]
Dividing by $t_{j_0-1} - t_{j_0}$ and rearranging yields 
\[
c_J (t_1 - t_{j_0+1}) \cdots (t_1 - t_{n-1})  \prod_{\beta \in \mathcal{A}_1} \beta    = - c_K (t_n - t_{j_0+1}) \cdots (t_n - t_{n-1})  \prod_{\beta \in \mathcal{A}_1} s_{\theta}(\beta) . 
\]
Substituting this expression back in to~\eqref{eq: star 1} and using~\eqref{eqn.formula} and~\eqref{eqn.vKformula} we obtain 
\begin{eqnarray*}
(t_1 - t_{j_0}) \left(- c_K (t_n - t_{j_0+1}) \cdots (t_n - t_{n-1}) \prod_{\beta \in \mathcal{A}_1} s_{\theta}(\beta) \right)  + c_K  (t_n-t_{j_0})(t_n - t_{j_0+1}) \cdots (t_n - t_{n-1})  \prod_{\beta \in \mathcal{A}_1} s_{\theta}(\beta)  = 0
\end{eqnarray*}
from which it follows that 
\begin{eqnarray*}
c_K(t_n-t_1)\left((t_n - t_{j_0+1}) \cdots (t_n - t_{n-1}) \prod_{\beta \in \mathcal{A}_1} s_{\theta}(\beta) \right) = 0
\end{eqnarray*}
and we therefore conclude $c_K=0$, and hence $c_J=0$ also.

We have now shown that two of the coefficients, namely $c_J$ and $c_K$ for $J=\{1,2,\ldots, \lambda_1\}$ and $K=\{2,3,\ldots, \lambda_1, n\}$, appearing in the linear combination from~\eqref{eq: zero linear comb} are equal to $0$. 
Now for any $I \subseteq [n]$ of cardinality $\lambda_1$ with $\lvert J \cap I \rvert = \lambda_1-1$, by Proposition~\ref{prop.support} we can find a permutation $w$ such that $v_I\cdot f_{\lambda}^{(k)}$ and $v_J \cdot f_{\lambda}^{(k)}$ do not vanish at $w$, and $ v_K\cdot  f_{\lambda}^{(k)} (w) = 0$ for all $K \neq I$ and $K \neq J$. But then the fact that $c_J = 0$ implies $c_I = 0$ also. Now we can use the fact that any $I \subseteq [n]$ of cardinality $\lambda_1$ can be obtained from $J$ in finitely many steps by changing $1$ element in the subset at a time, so that by iterating this argument we conclude that $c_I =0$ for all coefficients appearing in the LHS of~\eqref{eq: zero linear comb}, as desired. This shows that the classes $\{v_J \cdot f^{(\lambda_2-1)}_{\lambda} \mid v_J \in \Sn^\lambda \}$ are $H^*_T(\pt)$-linearly independent, as desired. 

Finally, the assertion that the submodule of $H_T^*(\Hess(\mathsf{S},h))$ spanned by the $\Sn$-orbit of $f_{\lambda}^{(\lambda_2-1)}$ has the same character as $\mathrm{Ind}_{S_\lambda}^{\Sn}(\mathbf{1})$ follows immediately from the fact that the stabilizer of $f_\lambda ^{(\lambda_2-1)}$ is $S_\lambda$ by Lemma~\ref{lemma: stabilizer is Slambda}.  The stabilizer of $v_J\cdot f_\lambda^{(k)}$ is $v_JS_\lambda v_J^{-1} \simeq S_\lambda$. This completes the proof of the theorem.
\end{proof}


\section{Linear independence between two $\Sn$-orbits} \label{sec.union.orbits}

In this section, we seek to partially address Problem 3 of Section~\ref{subsec: perm basis program}, in a special case. 
Recall that Problem 2 asks when a single $\Sn$-orbit is $H^*_T(\pt)$-linearly independent. Problem 3 then asks for conditions under which a union of more than one $\Sn$-orbit is also $H^*_T(\pt)$-linearly independent. In this section we focus exclusively on the case where the $\Sn$-orbits under consideration consist of homogeneous elements of the \emph{same} degree. This is a reasonable condition, since the dot action preserves degrees. We now state precisely the hypotheses for the special case we consider in this section. 
First, we restrict to the case $\lambda=(1,n-1)$, so $\lambda_1 = 1$ and $\lambda_2=n-1$.  In this setting, by Theorem~\ref{thm.class} we know that $f_\lambda^{(k)}$ is a well-defined equivariant cohomology class of $H^*_T(\Hess(\mathsf{S},h))$ for all $0 \leq k \leq n-1$.  Second, we also assume $h(1)<n-1$ so Theorem~\ref{theorem: orbit independence} holds and thus the set of cohomology classes in the $\Sn$-orbit of $f_\lambda^{(\lambda_2-1)} = f_\lambda^{(n-2)}$ is $H^*_T(\pt)$-linearly independent.

We now consider the two GKM classes $f^{(\lambda_2)}_{\lambda} = f^{(n-1)}_{\lambda}$ and $f^{(\lambda_2-1)}_{\lambda} = f^{(n-2)}_{\lambda}$ as defined by~\eqref{eqn.cohom-class} corresponding to the choices $k=\lambda_2=n-1$ and $k=\lambda_2-1=n-2$, respectively.  Since $\lambda = (1,n-1)$, the ${\Sn}$-orbit of both $f^{(n-1)}_{\lambda}$ and $f^{(n-2)}_{\lambda}$ are given by taking the images under the dot action of the elements of
$\Sn^\lambda=\{e, u_1^{-1}, u_2^{-1}, \ldots, u_{n-1}^{-1}\}$, where $u_k$ was defined in~\eqref{eq: def uk}. 
For the purpose of this section only, we define notation as follows: 
\begin{eqnarray}\label{eqn.worbits}
f_i := u_i^{-1}\cdot f^{(n-2)}_{\lambda} \ \textup{ and } \ g_i:= u_i^{-1}\cdot f^{(n-1)}_{\lambda} \ \textup{ for } \ i=0, \ldots, n-1.
\end{eqnarray}
As explained above, we restrict our considerations to the case in which $\deg(f_0) = \deg(g_0)$.
The main result of this section is Theorem~\ref{theorem:linear-ind2}, which states that the set $\{f_0, f_1, \cdots, f_{n-1}, g_0, g_1, \cdots, g_{n-1}\}$ is $H_T^*(\pt)$-linearly independent whenever $\deg(f_0) = \deg(g_0) \geq 2$.  In other words, we show that the union of the two permutation bases $\{f_0, \ldots, f_{n-1}\}$ and $\{g_0, \ldots, g_{n-1}\}$, shown individually to be linearly independent in Theorem~\ref{theorem: orbit independence}, is still linearly independent when considered together. 
This therefore represents another step toward the larger goal of building a global permutation basis for the entire cohomology ring $H^*_T(\Hess(\mathsf{S},h))$, as proposed in Problem 3 of Section~\ref{subsec: perm basis program}.

Before embarking on the proof of Theorem~\ref{theorem:linear-ind2} we 
consider  the hypothesis that $\deg(f_0) = \deg(g_0)$. 
Recall that $\S_k:= N(v_k^{-1})\cap v_k(\Phi_h^-)$.  Since $\lambda = (1, n-1)$, we have that $v_k=u_k$ for all $0\leq k \leq n-1$ as noted in Remark~\ref{remark: lambda is mu}.  In particular, by Lemma~\ref{lemma.vk-invs} we have
\[
N(u_{n-1}^{-1}) = \{ t_1-t_b \mid 2\leq b \leq n \} \; \textup{ and } \; N(u_{n-2}^{-1}) = \{t_1-t_b \mid 2\leq b\leq n-1\}.
\]
We also have 
\[
u_{k}\Phi_h^- = \{ t_i-t_j \mid u_k^{-1}(j)<u_k^{-1}(i) \leq h(u_k^{-1}(j)) \}
\]
from which it follows that
\begin{eqnarray*}\label{eqn.S1}
\S_{n-1} = \{ t_1-t_b \mid 2\leq b \leq n \textup{ and } n \leq h(b-1) \} = \{ t_1-t_{i+1} \mid 1\leq i \leq n-1, h(i) = n  \}
\end{eqnarray*}
and
\begin{eqnarray*}\label{eqn.S2}
\S_{n-2} = \{ t_1-t_b \mid 2\leq b\leq n-1 \textup{ and } n-1\leq h(b-1)  \} = \{t_1-t_{i+1} \mid 1 \leq i \leq n-2, h(i)\geq n-1\}.
\end{eqnarray*}
Since the degrees of $f_0 := f_{\lambda}^{(n-2)}$ and $g_0 := f_{\lambda}^{(n-1)}$ are given by the cardinalities of the sets $\S_{n-2}$ and $\S_{n-1}$, respectively, we obtain
\[
\deg(f_0)  = |\{ i \mid i < n-1, \ h(i) \geq n-1   \}| \; \textup{ and } \; \deg(g_0) = |\{i\mid i<n, \ h(i)=n\}|.
\]
Thus, in order to ensure that our classes have the same degree we assume throughout this section that the Hessenberg function $h:[n]\to [n]$ has the property that 
\begin{eqnarray}\label{eqn.degrees}
 \lvert \{i\mid i<n-1,\ h(i)\geq n-1\} \rvert = \lvert \{i\mid i<n, \  h(i)= n \} \rvert. 
\end{eqnarray}

The following lemma records some properties of Hessenberg functions satisfying~\eqref{eqn.degrees}.
\begin{lemma} \label{lemma.j-defn}
Suppose $h: [n] \to [n]$ is a connected Hessenberg function such that~\eqref{eqn.degrees} holds. Then
\begin{enumerate} 
\item the set 
$\{i \mid i < n-1, \ h(i) \geq n-1 \}$ is non-empty, 
\item if we let $j:=\min \{i\mid i<n-1,\ h(i)\geq n-1  \}$, then  
$j$ is the unique element of $[n]$ such that $h(j)=n-1$, 
\item $\deg(f_0) = \deg(g_0) = n-j-1$ for the classes $f_0, g_0$ defined above, and 
\item if $ \lvert \{i\mid i<n-1,\ h(i)\geq n-1\} \rvert = \lvert \{i\mid i<n, \  h(i)= n \} \rvert \geq 2$, then $j < n-2$. 
\end{enumerate}  
\end{lemma}

\begin{proof} 
For the first claim, observe that under the assumption~\eqref{eqn.degrees}, it suffices to show that 
$\{ i \mid i < n-1, \ h(i)=n\}$ is non-empty.  But it follows from the connectedness of $h$ that $h(n-1)=n$ so $n-1 \in \{i \mid i<n, \ h(i)=n\}$ and hence the set is non-empty as desired. 
To prove the second claim we first observe that
\[
\{i\mid i<n-1, \ h(i)\geq n-1\} = \{i \mid i<n-1,\ h(i)= n-1\} \sqcup \{i \mid i<n-1,\ h(i)=n\}
\]
and
\[
\{i\mid i<n,\ h(i)=n \} = \{i \mid i< n-1, \ h(i) =n \} \sqcup \{ n-1 \}
\]
where we have again used that $h$ is connected, so $h(i) \geq i+1$ for all $1 \leq i \leq n-1$. 
Combining these equations with assumption~\eqref{eqn.degrees} we conclude 
\begin{eqnarray}\label{eqn.degrees2}
|\{i \mid i<n-1, \ h(i)= n-1\}| = 1.
\end{eqnarray}
The assertion that $j$ is unique and $h(j)=n-1$ now follows. 
To see the third claim, note that by definition of $j$ we have 
\[
\{ i \mid i < n-1, \ h(i) \geq n-1\} = \{j, j+1, j+2, \cdots, n-2\}
\]
which implies $\lvert \{ i \mid i < n-1, \ h(i) \geq n-1\} \rvert = n-j-1$, as claimed. Finally, the last claim follows immediately from the third claim, since $\lvert \{i \mid i < n-1, \ h(i) \geq n-1 \} \rvert = n-j-1 \geq 2$ implies $j \leq n-3$, or equivalently $j < n-2$. 
\end{proof}

We now state our main theorem.

\begin{theorem}\label{theorem:linear-ind2}
Let $\lambda=(1,n-1)$ and assume that $h: [n] \to [n]$ is a connected Hessenberg function 
satisfying condition~\eqref{eqn.degrees} and such that $h(1)<n-1$. 
Let $f_i= u_{i}^{-1}\cdot f_\lambda^{(n-2)}$ and $g_i = u_i^{-1}\cdot f_\lambda^{(n-1)}$ for all $i=0, \ldots, n-1$. 
Suppose that $\deg(f_0) = \deg(g_0)\geq 2$. 
  Then the union of the ${\Sn}$-orbits of $f_0$ and $g_0$, namely the set $\{f_0, \ldots, f_{n-1}, g_0, \ldots, g_{n-1}\}$, is $H_T^*(\pt)$-linearly independent.
\end{theorem}

Let us make some preliminary observations. 
In order to show that $\{f_0, f_1, \cdots, f_{n-1}, g_0, g_1, \cdots, g_{n-1}\}$ is $H^*_T(\pt)$-linearly independent, 
we need to show that if the following equality holds 
\begin{eqnarray}\label{eqn.linear-ind2}
c_0f_0 + c_1f_1+\cdots + c_{n-1}f_{n-1}+d_0g_0 + d_1g_1+\cdots+ d_{n-1}g_{n-1} = 0
\end{eqnarray}
in $H^*_T(\Hess(\mathsf{S},h))$, where $c_0, \cdots, c_{n-1}, d_0, \cdots, d_{n-1} \in H^*_T(\pt)$, then the coefficients are all zero, i.e. $c_0 = c_1 = \cdots = c_{n-1} = d_0 = d_1 = \cdots = d_{n-1} = 0$.

In the course of our arguments, will make use of the following \cite[p.65, Exercise 11]{CoxLittleOShea}. 

\begin{proposition}\label{prop.matrix} Let $A$ be a $(m-1)\times m$ matrix with entries in the polynomial ring $R=k[t_1, \ldots, t_n]$ where $k$ is a field.  Suppose also that the $m-1$ rows of $A$ are linearly independent over $R$.  Then:
\begin{enumerate}
\item The vector $\mathbf{b}^{tr} = (a_1, a_2, \ldots, a_m)\in R^m$ defined by $a_i = (-1)^{i+1} \det(A_i)$ satisfies  $A\mathbf{b} = 0$.  Here $A_i$ is the $(m-1)\times (m-1)$ sub-matrix of $A$ obtained by deleting the $i$-th column of $A$. 
\item Any solution $\mathbf{b}_0$ of the equation $A\mathbf{x}=0$ is of the form $g\mathbf{b}$ for some $g\in R$, i.e. any solution must be a polynomial multiple of $\mathbf{b}$.
\end{enumerate}
\end{proposition}

We make some preliminary calculations. 
Set $j: = \min\{i\mid i<n-1, h(i)\geq n-1\}$ as in Lemma~\ref{lemma.j-defn}. Using the definition of $f_\lambda^{(n-2)}$  
we can calculate the value of $f_0$ at $u_{n-1}$ and $u_{n-2}$ to obtain
\begin{eqnarray}\label{eqn.highest-coset-value-f}
f_0(u_{n-1}) = \prod_{\substack{i<n-1\\ h(i)\geq n-1}}(t_1-t_{i+1}) = \prod_{j \leq i \leq n-2} (t_1-t_{i+1}) = f_0(u_{n-2}). 
\end{eqnarray}
Similarly for $g_0$ we can use the definition of $f_\lambda^{(n-1)}$ to compute 
\begin{eqnarray}\label{eqn.highest-coset-value-g}
g_0(u_{n-1}) = \prod_{\substack{i<n\\ h(i) =  n}} (t_1-t_{i+1}) = \prod_{j+1\leq i \leq n-1} (t_1-t_{i+1}).
\end{eqnarray}
Next, recall that by definition $f_0 = f^{(n-2)}_{\lambda}$ is nonzero on precisely two right cosets of ${S_\lambda}$ in ${\Sn}$, namely ${S_\lambda} u_{n-2}$ and ${S_\lambda} u_{n-1}$ where 
\begin{eqnarray}\label{eqn.one-line}
u_{n-2} = [2,3,\ldots, n-1, 1, n] \ \textup{ and } \ u_{n-1} = [2,3,\ldots , n, 1 ].
\end{eqnarray}
It is easy to confirm by a direct calculation that:
\begin{eqnarray} \label{eqn.coset-rep-sp}
u_{n-1}^{2} = s_1s_2\cdots s_{n-1}s_1s_2\cdots s_{n-1} = s_2s_3\cdots s_{n-1}u_{n-2}.
\end{eqnarray}
The following are also straightforward computations: 
\[
u_{n-1} s_j = s_{j+1} u_{n-1} \quad \textup{ which means } \quad s_j u_{n-1}^{-1} = u_{n-1}^{-1} s_{j+1}
\]
and also 
\begin{equation}\label{eq: theta}
u_{n-1}^{-1} s_2 s_3 \cdots s_{n-1} = s_\theta
\end{equation} 
where $\theta = t_1 - t_n$, so $s_\theta$ is the transposition exchanging $1$ and $n$ only. 
We can now compute 
\begin{eqnarray*}
f_{n-1}(u_{n-1}) &=& \left( u_{n-1}^{-1} \cdot f^{(n-2)}_\lambda \right)(u_{n-1}) \, \, \textup{ by definition of $f_{n-1}$}  \\
&=& u_{n-1}^{-1}(f_0(u_{n-1}^2)) \, \, \textup{ by definition of the dot action and $f_0$}  \\ 
&=& u_{n-1}^{-1} (f_0(s_2s_3\cdots s_{n-1} u_{n-2})) \, \, \textup{ by~\eqref{eqn.coset-rep-sp}}   \\ 
&=& u_{n-1}^{-1} s_2 s_3 \cdots s_{n-1}(f_0(u_{n-2})) \ \textup{ by Lemma~\ref{lemma: fk def}(3) with $y=s_2s_3\cdots s_{n-1}$ } \\
&=& s_\theta(f_0(u_{n-2})) \ \textup{ by~\eqref{eq: theta}. }
\end{eqnarray*}
From the above and~\eqref{eqn.highest-coset-value-f} we immediately obtain 
\begin{equation}\label{eq: fn-1 at un-1}
f_{n-1}(u_{n-1}) = \prod_{j\leq i \leq n-2} (t_n - t_{i+1}).
\end{equation}

Next, note that $2\leq j \leq n-2$ by Lemma~\ref{lemma.j-defn} and because we have assumed that $h(1)<n-1$.  We therefore have $s_j\in {S_\lambda}$ and
\begin{eqnarray}\label{eqn.eqn2}
s_j u_{n-1} \in {S_\lambda} u_{n-1} \ \textup{ and } \ u_{n-1}s_j u_{n-1}  = s_{j+1}u_{n-1}^2 
= s_{j+1} s_2 s_3 \cdots s_{n-1} u_{n-2} \in {S_\lambda} u_{n-1}
\end{eqnarray}
where we have used~\eqref{eqn.coset-rep-sp} 
and the computations above. 
We can now compute:
\begin{eqnarray}\label{eqn.eqn3}
f_0(s_j u_{n-1})  =  (t_1-t_j)\prod_{j<i\leq n-2}(t_1-t_{i+1})
\end{eqnarray}
since $f_0(s_j u_{n-1}) = s_j f_0(u_{n-1})$ and by~\eqref{eqn.highest-coset-value-f}. 
Recall that we also know $f_0(u_{n-1}) = f_0(u_{n-2})$ by definition of the class $f_0 = f^{(n-2)}_\lambda$  as computed in~\eqref{eqn.highest-coset-value-f}.  Therefore,
\begin{eqnarray}\label{eqn.eqn4}
f_{n-1}(s_j u_{n-1})  &:=&  \left( u_{n-1}^{-1} \cdot f_0 \right)(s_j u_{n-1}) = u_{n-1}^{-1}(f_{0}(u_{n-1}s_ju_{n-1})) = u^{-1}_{n-1}(f_0(s_{j+1}s_2s_3\cdots s_{n-1}u_{n-2})) \textup{ by~\eqref{eqn.eqn2}}   \nonumber \\
&=& u_{n-1}^{-1}s_{j+1}s_2s_3\cdots s_{n-1} (f_0(u_{n-2})) = s_js_\theta (f_0(u_{n-1})) \textup{ by the computations above} \\ 
&=& s_\theta (f_{0}(s_ju_{n-1})) =   (t_n-t_j) \prod_{j<i\leq n-2}(t_n-t_{i+1}) \textup{ since $s_j$ and $s_\theta$ commute and by~\eqref{eqn.eqn3}}. \nonumber
\end{eqnarray}
Finally, we also note that $g_0(s_ju_{n-1}) = s_j(g_0(u_{n-1})) = g_0(u_{n-1})$ since $s_j\in S_\lambda$ and $s_j$ fixes the product appearing in~\eqref{eqn.highest-coset-value-g}.

 With these preliminaries in place, we can now begin our proof of  Theorem~\ref{theorem:linear-ind2}.

\begin{proof}[Proof of Theorem~\ref{theorem:linear-ind2}] 
The computations above show that $f_0$, $f_{n-1}$, and $g_0$ are all nonzero at $u_{n-1}$ and $s_j u_{n-1}$.  Lemma~\ref{lem.left-coset-description} and Proposition~\ref{prop.support} tells us that all other $f_i$ and $g_i$ evaluate to be zero at $u_{n-1}$ and $s_j u_{n-1}$. It follows that, when restricted to the $T$-fixed point $u_{n-1}$, equation~\eqref{eqn.linear-ind2}
 becomes 
 \[
 c_0 f_0(u_{n-1}) + c_{n-1} f_{n-1}(u_{n-1}) + d_0 g_0(u_{n-1}) = 0 
 \]
 and when restricted to $s_j u_{n-1}$ the same equation~\eqref{eqn.linear-ind2} becomes 
 \[
 c_0 f_0(s_j u_{n-1}) + c_{n-1}f_{n-1}(s_j u_{n-1}) + d_0 g_0(s_j u_{n-1}) = 0.
 \]
This is equivalent to the statement that the vector of polynomials $(c_0, c_{n-1}, d_0)^T$ is a solution to the matrix equation $A X = 0$, considered over the ring $H^*_T(\pt)$, where $A$ is the $2 \times 3$ matrix  
\[
A := \begin{bmatrix} f_0(u_{n-1}) & f_{n-1}(u_{n-1}) & g_0(u_{n-1}) \\ 
f_0(s_j u_{n-1}) & f_{n-1}(s_j u_{n-1}) & g_0(s_j u_{n-1}) \end{bmatrix}. 
\]
The entries in $A$ are elements of $H^*_T(\pt) \cong \C[t_1, \ldots, t_n]$, a polynomial ring over the field $\C$. 

We wish to apply Proposition~\ref{prop.matrix} with $m=3$, for which we need first to check that the rows of $A$ are linearly independent over $H^*_T(\pt)$. To do this it suffices to see that the determinant of at least one of the $2 \times 2$ minors of $A$ is non-zero. Let $A_i$ for $i=1,2,3$ denote the minor of $A$ with the $i$-th column deleted. It is a straightforward computation to see that 
\[
A_3 = \left( \prod_{j+1\leq i \leq n-1} (t_1 - t_{i+1}) \right) 
\left( \prod_{j<i\leq n-2} (t_n - t_{i+1}) \right)  (t_{j+1} - t_j)
\]
and 
\[
A_1 = \left( \prod_{j+1\leq i \leq n-1} (t_1 - t_{i+1})\right) 
\left( \prod_{j < i\leq n-2} (t_n - t_{i+1}) \right) (t_j-t_{j+1}).
\]
In particular, we see that $A_1 \neq 0$ and $A_3 \neq 0$ and $A_1 = - A_3$. Thus we may apply Proposition~\ref{prop.matrix}, and from it we conclude that $(c_0, c_{n-1}, d_0) = c (A_1, - A_2, A_3)$ for some $c \in H^*_T(\pt)$.  
Since we saw above that $A_1 = - A_3$, it follows immediately that $c_0 = - d_0$.

We now give the idea of the next steps in our argument before giving the details. From Lemma~\ref{lem.left-coset-description} and Proposition~\ref{prop.support} we know that at any given $w \in \Sn$, exactly two of the $f_i$'s and one of the $g_i$'s evaluate to be non-zero. In the above argument we chose two permutations $u_{n-1}$ and $s_j u_{n-1}$ which had the property that it was exactly $f_0, f_{n-1}$ and $g_0$ which evaluated to be non-zero at these permutations, thus isolating the $3$ coefficients $c_0, c_{n-1}$ and $d_0$ for analysis. By using Proposition~\ref{prop.matrix} we were then able to conclude that $(c_0, c_{n-1}, d_0)$ must be a scalar multiple of a certain vector obtained by taking minors of a $2 \times 3$ matrix, constructed from the values of $f_0, f_{n-1}$ and $g_0$ at these permutations. 
In the next part of our argument, our strategy is to find another permutation $w'$ such that $f_0, f_{n-1}$ and $g_0$ are exactly the three elements in $\{f_0, f_1, \ldots, f_{n-1}, g_0, g_1, \ldots, g_{n-1}\}$ which evaluate to be non-zero at $w'$.  Replacing $s_ju_{n-1}$ with $w'$, a similar argument as that given above creates a new $2 \times 3$ matrix $B$ and yields the conclusion that $(c_0, c_{n-1}, d_0)$ must be a scalar multiple of a vector defined using the minors of $B$. Thus, if we can find a permutation $w'$ such that the vector of minors of $B$ and the vector of minors of $A$ are linearly independent, then we can conclude that $(c_0, c_{n-1}, d_0)$ must be equal to $0$.

We now turn to the details of the argument sketched above. Recall from Lemma~\ref{lemma.j-defn} that $\deg(f_0)=\deg(g_0) = n-j-1$, and our assumption $\deg(f_0)=\deg(g_0)\geq 2$ implies $j<n-2$.
This in turn implies that the simple reflection $s_{j+1}$ commutes with $s_\theta$.  We now argue that we may take $w'=s_{j+1}u_{n-1}$. Indeed we can compute that since $s_{j+1}\in S_\lambda$, using~\eqref{eqn.highest-coset-value-g} we have 
\begin{eqnarray}\label{eqn.g0}
g_0(s_{j+1} u_{n-1}) = s_{j+1}(g_0(u_{n-1}))  = (t_1-t_{j+1}) \prod_{j+1<i\leq n-1} (t_1-t_{i+1}) 
\end{eqnarray}
and 
\begin{eqnarray}\label{eqn.f0}
f_0(s_{j+1}u_{n-1}) = s_{j+1}(f_0(u_{n-1})) = f_0(u_{n-1})
\end{eqnarray}
since $s_{j+1}$ fixes the product appearing in~\eqref{eqn.highest-coset-value-f}. Next, using similar reasoning as in~\eqref{eqn.eqn2} and~\eqref{eqn.eqn4}, we obtain
\begin{eqnarray*}
f_{n-1}(s_{j+1}u_{n-1}) = s_{j+1}(f_{n-1}u_{n-1}) = f_{n-1}(u_{n-1})
\end{eqnarray*}
since $s_{j+1}$ also fixes the product appearing in~\eqref{eq: fn-1 at un-1}.  
Thus $f_0, f_{n-1}, g_0$ are precisely the $3$ classes that evaluate to be non-zero at $w'$, and the other classes are all zero at $w'$. The above computations allow us to analyze the relevant $2 \times 3$ matrix 
\[
B := 
\begin{bmatrix} f_0(u_{n-1}) & f_{n-1}(u_{n-1}) & g_0(u_{n-1}) \\ 
f_0(s_{j+1} u_{n-1}) & f_{n-1}(s_{j+1} u_{n-1}) & g_0(s_{j+1} u_{n-1}) \end{bmatrix}.
\]

Recall that we had already observed that $A_1 = - A_3$ in the vector of minors obtained from the original matrix $A$. Let $B_i$ be the analogous minor of $B$ obtained by deleting the $i$-th column. As argued above, it suffices to show that $(A_1, A_2, A_3)$ is linearly independent from $(B_1, B_2, B_3)$, for which it suffices to see that $B_1 \neq - B_3$ (since $H^*_T(\pt)$ is an integral domain). From the above computations we obtain,
\[
B_1 = f_{n-1}(u_{n-1})g_0(s_{j+1}u_{n-1}) - g_0(u_{n-1}) f_{n-1}(s_{j+1}u_{n-1}) = f_{n-1}(u_{n-1}) [g_0(s_{j+1}u_{n-1}) - g_0(u_{n-1}) ]
\]
and thus
\[
B_1 = \left( \prod_{j\leq i \leq n-2} (t_n-t_{i+1}) \right) \left( \prod_{j+1<i\leq n-1}(t_1-t_{i+1}) \right) (t_{j+2}-t_{j+1})
\]
so $B_1\neq 0$. On the other hand, we have
\[
B_3 = f_0(u_{n-1})f_{n-1}(s_{j+1}u_{n-1}) - f_{n-1}(u_{n-1})f_0(s_{j+1}u_{n-1}) 
=0
\]
and the result now follows.

Thus we have seen that $(B_1, B_2, B_3)$ is $H^*_T(\pt)$-linearly independent from 
$(A_1, A_2, A_3)$, which shows that $c_0 = c_{n-1}=d_0 = 0$. 

To complete the argument, we must now show that $c_i = d_i=0$ for all $i$, $0 \leq i \leq n-1$.

Consider $w \in \Sn$ such that $w(n)=1$ and $w(n-1)=i+1$ for $i \not \in \{0,n-1\}$. Since $w(n)=1$, we obtain $w\in S_\lambda u_{n-1}$ which implies that $f_0(w)\neq 0$ and $g_0(w)\neq 0$. Furthermore, $w(n-1) = i+1$ implies that $u_i w (n-1) = u_i(i+1)=1$ so $u_iw\in S_\lambda u_{n-2}$ which tell us that
\[
f_i(w) = u_i^{-1}\cdot f_0(w) = u_i^{-1}(f_0(u_iw)) \neq 0
\] 
also.  Thus, evaluating equation~\eqref{eqn.linear-ind2} at $w$ we get
\[
c_0 f_0(w) + c_i f_i(w) + d_0 g_0(w) = 0.
\]
However, since $c_0=d_0 = 0$, this implies that $c_i = 0$, since $f_i(w) \neq 0$ and $H^*_T(\pt)$ is an integral domain.  Thus $c_i=0$ for all $0 \leq i\leq n-1$. This means that the original linear dependence relation is among the $g_0, g_1, \ldots, g_{n-1}$, but we have already proved these are linearly independent, so $d_i=0$ for all $0 \leq i \leq n-1$. This concludes the proof. 
\end{proof}

We conclude with a motiving example and open problem.  As noted in the introduction, one reason for focusing on partitions with two parts is the fact that when $h: [n] \to [n]$ is an \textit{abelian} Hessenberg function (that is, when $h(1)\geq \max\{ i \mid h(i)<n \}$), the only irreducible representations which occur in the dot action representation are those corresponding to partitions with at most two parts (see \cite[Cor.~5.12]{Harada-Precup2019}). In this case, the Stanley--Stembridge conjecture is known to hold and work of the first two authors gives an inductive formula for number of permutation representations $M^\mu$ that appear in each graded part \cite{Harada-Precup2019}.  The following example considers a special case of abelian Hessenberg functions.  Using the constructions of this manuscript, we are able to define the correct number of equivariant cohomology classes generating the representations $M^{(n-1,1)}$ in certain graded pieces of the dot action representation.

\begin{example} \label{ex.(n-2,n-1,n,...,n)} Let $n$ be a positive integer with $n\geq 5$ and $h=(n-2,n-1,n, n, \ldots, n)$.  We consider the decomposition of each graded piece of the dot action representation into permutation representations, given by
\begin{eqnarray}\label{eqn.c-coeff}
H^{2i} (\Hess(\mathsf{S}, h)) = \bigoplus_{\substack{\mu \vdash n\\ \mu=(\mu_1, \mu_2)}}c_{\mu, i} M^{\mu}.
\end{eqnarray}
In the special case under consideration, we apply the results of~\cite{Harada-Precup2019}.  The possible two-element sink sets (i.e.~independent sets) of the ``incomparability graph'' of $h=(n-2,n-1,n,n,\ldots,n)$ are $\{1,n-1\}$, $\{2,n\}$, and $\{1,n\}$.  Now the inductive formula of \cite[Thm.~6.1]{Harada-Precup2019} tells us that 
\[
c_{\mu, i} = 0 \; \textup{ unless } \; \mu\in \{(n), (n-1,1), (n-2,2)\}.
\]
and
\[
c_{(n-1,1),i} = 0 \;  \textup{ for all $0\leq i \leq n-4$ and } \; c_{(n-1,1),n-3}=2.
\]
(The interested reader can find a similar computation in~\cite[Example 6.2]{Harada-Precup2019}.)
In other words, the minimal degree in which $M^{(n-1,1)}$ appears is  $2(n-3)$, and there are exactly two copies of $M^{(n-1,1)}$ in this degree.  By assumption, $h$ is a connected Hessenberg function satisfying all assumptions of Theorem~\ref{theorem:linear-ind2} above.  In particular:
\[
|\{i\mid i<n-1, h(i)\geq n-1 \}| = | \{ i\mid i<n, h(i)=n \}| =  n-3
\]  
in this case.  Thus, the classes $\{f_0, f_1, \ldots, f_{n-1}, g_0, g_1, \ldots, g_{n-1}\}$ give us a linearly independent set of equivariant classes in $H_T^{2(n-3)}(\Hess(\mathsf{S},h))$ that together span exactly two $H_T^*(\pt)$-modules, each of which is isomorphic to $M^{(n-1,1)}$.
\end{example}

The example above shows that our Theorem~\ref{theorem:linear-ind2} yields part of a permutation basis for $H^{2(n-3)}(\Hess(\mathsf{S},(n-2, n-1, n, \ldots, n)))$.  Indeed, one easily confirms that the only other representations appearing in this degree are trivial.  We therefore recover a permutation basis for $H^{2(n-3)}_T(\Hess(\mathsf{S},(n-2, n-1, n, \ldots, n)))$ by adding to our collection an appropriate number of $\Sn$-invariant classes of degree  $2(n-3)$.   It is still an open question how to build, in the other degrees, linearly independent sets of classes spanning permutation modules.

More interestingly, since  $2(n-3)$ is the \textit{minimal} degree in which $M^{(n-1,1)}$ occurs in $H^*(\Hess(\mathsf{S},(n-2, n-2, n, \ldots, n)))$, one could hope to obtain classes in higher degree generating an isomorphic $H_T^*(\pt)$-submodule by multiplying each of the $f_i$'s (or $g_i$'s) by some appropriately chosen $\Sn$-invariant class.  

\begin{problem}\label{problem: linear indep} Let $n$ be a positive integer with $n\geq 5$ and set $h= (n-2, n-1, n, \ldots, n)$. Suppose $k>n-3$ and $c_{(n-1,1),k}\neq 0$ where $c_{(n-1,1),k}$ is the coefficient defined as in~\eqref{eqn.c-coeff} above.  Identify $\Sn$-invariant classes $h_1, \ldots, h_m \in H_T^{2(k-(n-3))}(\Hess(\mathsf{S},h))$, where $m=c_{(n-1,1),k}$, and $r$ with $1\leq r\leq m$ so that the set
\[
\{ h_j f_i \mid 1\leq j \leq r,0\leq i \leq n-1 \} \cup \{ h_j g_i \mid r+1\leq k\leq m, 0 \leq i \leq n-1 \}
\]
is $H_T^*(\pt)$-linearly independent.
\end{problem}

Any solution to this open problems is another step toward the construction of a permutation basis for the dot action representation in this case. In general, one may hope to show that our construction always yields a linearly independent basis for those $M^\mu$ of minimal degree that appear as summands of the $\Sn$-representation on $H^*_T(\Hess(\mathsf{S},h))$, whenever $h$ is abelian.

\bibliographystyle{alpha}

\end{document}